\documentclass[11pt,a4paper]{article}
\usepackage[body={155mm,230mm},top=34mm]{geometry}
\usepackage[utf8]{inputenc}
\usepackage[refpage]{nomencl}
\makenomenclature

\usepackage{graphicx}
\usepackage{placeins}
\usepackage{bm}
\usepackage{amscd}
\usepackage{mathrsfs}
\usepackage{subcaption}
\usepackage{epsfig}
\usepackage{amsmath}
\usepackage{amssymb}
\usepackage{amsthm}
\usepackage{epsfig}
\usepackage{verbatim}
\usepackage{url}
\usepackage{color}

\include{pspicture}

\newtheorem{theorem}{Theorem}[section]
\newtheorem{lemma}{Lemma}[section]

\newtheorem{remark}{Remark}[section]
\newtheorem{assumption}{Assumption}
\newtheorem{proposition}{Proposition}[section]
\newtheorem*{theorem*}{Theorem}

\usepackage{makeidx}
\usepackage{mathabx}
\usepackage{mathrsfs} 

\begin{document}
\title{Approximation of the Exit Probability of a Stable Markov Modulated Constrained Random Walk}
\author{Fatma Ba\c{s}o\u{g}lu Kabran\footnote{Middle East Technical University, Institute of Applied Mathematics, Ankara, Turkey and
İzmir Kavram Vocational School, Department of Finance, Banking and Insurance, İzmir, Turkey}, Ali Devin Sezer
\footnote{Middle East Technical University, Institute of Applied Mathematics, Ankara, Turkey}}
\providecommand{\keywords}[1]{\textbf{Keywords:} #1}
\maketitle
\begin{abstract}
Let $X$ be the constrained random walk on ${\mathbb Z}_+^2$ having increments $(1,0)$, $(-1,1)$, $(0,-1)$ with jump probabilities $\lambda(M_k)$, $\mu_1(M_k)$, and $\mu_2(M_k)$ where $M$ is an irreducible aperiodic finite state Markov chain.  The process $X$ represents the lengths of two tandem queues with arrival rate $\lambda(M_k)$, and service rates $\mu_1(M_k)$, and $\mu_2(M_k)$.  We assume that the average arrival rate with respect to the stationary measure of $M$ is less than the average service rates, i.e., $X$ is assumed stable.  Let $\tau_n$ be the first time when the sum of the components of $X$ equals $n$ for the first time.  Let $Y$ be the random walk on ${\mathbb Z} \times {\mathbb Z}_+$ having increments $(-1,0)$, $(1,1)$, $(0,-1)$ with probabilities $\lambda(M_k)$, $\mu_1(M_k)$, and $\mu_2(M_k)$.  Let $\tau$ be the first time the components of $Y$ are equal.  For $x \in {\mathbb R}_+^2$, $x(1) + x(2) < 1$, $x(1) > 0$, and $x_n = \lfloor nx \rfloor$, we show that $P_{(n-x_n(1),x_n(2)),m)}( \tau <  \infty)$ approximates $P_{(x_n,m)}(\tau_n < \tau_0)$ with exponentially vanishing relative error as $n\rightarrow \infty$.  For the analysis we define a characteristic matrix in terms of the jump probabilities of $(X,M).$ The $0$-level set of the characteristic polynomial of this matrix defines the characteristic surface; conjugate points on this surface and the associated eigenvectors of the characteristic matrix are used to define (sub/super) harmonic functions which play a fundamental role both in our analysis and the computation / approximation of $P_{(y,m)}(\tau < \infty).$\\

\noindent \keywords{Markov modulation, regime switch, multidimensional constrained random walks, exit probabilities, rare events, queueing systems, characteristic surface, superharmonic functions, affine transformation}
~
\end{abstract}

\section{Introduction and Definitions}
A stochastic processes $X$ is said to be Markov modulated if
its dynamics depend on the state of  a secondary Markov process $M$
modeling the environment within which $X$ operates \cite{prabhu1989markov}.
Markov modulation/regime switch is one of the most popular methods of building richer models for a wide range of applications from finance to computer networks
to queueing theory. 
This paper studies the approximation
of the probability of a large excursion in the busy cycle of
a constrained random walk $X$ 
whose dynamics are modulated by a Markov process $M$.
We assume $M$ to be external, i.e, the transition probabilities of $M$
do not depend on $X$.
Constrained random walks arise naturally when there are barriers that
keep a process within a domain, for example: computers/ algorithms 
sharing resources on a system,
financial positions that have shortselling constraints, or queueing
systems.
If $X$ represents a queueing system, a large excursion corresponds
to a buffer overflow event; the
analysis, simulation and approximation 
of probabilities of such events for ordinary (non-modulated) constrained
random walks have received considerable attention at least since
\cite{ParWal,glasserman1995analysis}; for further references
and a literature review we refer the reader to \cite{sezer2018approximation}.
To the best of our knowledge,
there is hardly 
any study on
the same probability
for modulated constrained random walks: we are aware of
only \cite{sezer2009importance} treating the development of 
asymptotically optimal
importance sampling algorithms for the approximation of the buffer overflow event.
For this reason, this work will focus on one of the simplest
multidimensional constrained random walks, the tandem walk, arising
from the modeling of two servers working in tandem. Next we 
describe the dynamics of this process and
give a precise definition of the buffer overflow probability of interest.

\nomenclature{$X$}{the constrained random walk}

Our main process is a random walk $X$ with
increments
$\{I_1,I_2,I_3,...\}$,
\nomenclature{$I_1,I_2,..$}{Increments of $X$}
constrained to remain in ${\mathbb Z}_+^2$:

\begin{align*}
X_0 &=x \in {\mathbb Z}_+^2, ~~~X_{k+1} \doteq X_k + \pi ( X_k,I_{k} ), k=1,2,3,...\\
\pi(x,v) &\doteq \begin{cases} v, &\text{ if } x +v  \in {\mathbb Z}^2_+, \\
0      , &\text{otherwise.}
\end{cases}
\end{align*}
\nomenclature{$\pi$}{Constraining map for $X$}
The map $\pi$ ensures that when $X$ is on the constraining boundaries
\[
\partial_i \doteq \{ x \in {\mathbb Z}^2: x(i) = 0 \},~~ i=1,2,
\]
\nomenclature{$\partial_i$}{Constraining boundaries}
it cannot jump out of ${\mathbb Z}_+^2.$
We assume the distribution of the increments $I_k$ to be modulated
by a Markov Chain $M$ with state space ${\mathcal M}$
\nomenclature{$M$}{Modulating finite state Markov chain}
\nomenclature{${\mathcal M}$}{State space of $M$}
(with finite size $|{\mathcal M}|$)
and with transition matrix 
${\bm P} \in {\mathbb R}_+^{|{\mathcal M}| \times |{\mathcal M}|}$. 
\nomenclature{${\bm P}$}{Transition matrix of $M$}
To ease analysis and notation we will assume ${\bm P}$ to be irreducible and aperiodic,
which implies that it has a unique
stationary measure $\bm\pi$ on ${\mathcal M}$, i.e., $\bm\pi = \bm\pi \bm P.$
\nomenclature{$\bm\pi$}{Stationary measure of $M$}
Let ${\mathscr F}_k \doteq \sigma(\{M_j, j \le k + 1\}, \{X_j, j\le  k \})$, i.e., the $\sigma$-algebra generated by $M$ and $X$. 
The increments $I$ form an independent sequence given $M$ and the
increment $I_k$ has the following distribution given ${\mathscr F}_{k-1}$:
\begin{align*}
&I_k \in \{(0,0),(1,0), (-1,1), (0,-1)\},\\
&{\mathbb P}(I_k = (0,0) | {\mathscr F}_{k-1} ) = 1_{\{M_k \neq M_{k-1}\}}\\
&{\mathbb P}(I_k = (1,0) | {\mathscr F}_{k-1}) = \lambda(M_k) 1_{\{M_k = M_{k-1}\}}\\
&{\mathbb P}(I_k = (-1,1) | {\mathscr F}_{k-1}) = \mu_1(M_k) 1_{\{M_k = M_{k-1}\}}\\
&{\mathbb P}(I_k = (0,-1) | {\mathscr F}_{k-1}) = \mu_2(M_k) 1_{\{M_k = M_{k-1}\}}.
\end{align*}
\nomenclature{$\mathcal{F}_k$}{sigma-algebra generated by $M$ and $X$}
The dynamics of $X$ are shown in Figure \ref{f:dynamics}.

\begin{figure}[h]
\begin{center}
\scalebox{0.9}{
\centerline{\input{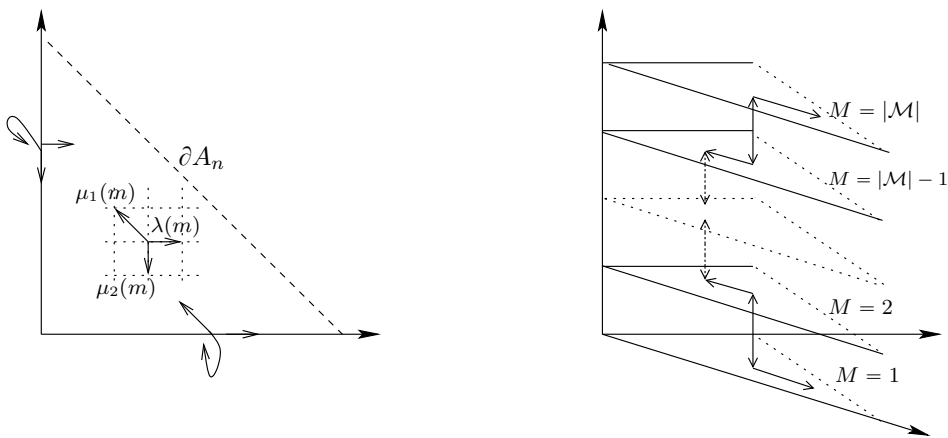}}}
\end{center}

\vspace{-0.65cm}
\caption{\hspace{0.25cm}Markov modulated constrained random walk $(X, M)$; the left figure shows dynamics in a given layer, the right figure shows jumps
between layers representing regime switches\label{f:dynamics}}
\end{figure}

\nomenclature{$(X, M)$}{Markov modulated constrained random walk}
The process $(X,M)$ is the embedded random walk of
a continuous time queueing system
consisting of two tandem queues
whose arrival and service rates are determined by a finite state Markov process
$M$.

We assume $(X,M)$ to be stable:
\begin{equation}\label{e:stabilityas}
\sum_{m \in {\mathcal M}} ( \lambda(m) - \mu_i(m)) \bm\pi(m) {\bm P}(m,m) < 0,~
i=1,2.
\end{equation}
In addition to \eqref{e:stabilityas}, 
we need two further technical assumptions for our analysis
see \eqref{as:conj} and \eqref{as:mu1neqmu2}.
Stability means that the queueing system represented by $(X,M)$ serves
customers faster, on average, than the customer arrival rate; this
keeps the lengths
of both queues close to $0$ at all times with high probability; but
$(X,M)$ being a random process, components of $X$ can grow arbitrarily
large if one waits long enough.
For a stable constrained random walk such as $(X,M)$ it is natural to measure
time in cycles that restart each time $X$ hits $0$. If the system represented by this walk has
a shared buffer where all customers wait 
(or where packets are stored, if, e.g., $(X,M)$ represents
a network of two computers / processes)
then a natural question is the following: what is the probability that
the shared buffer overflows in a given cycle? 
To express this problem mathematically we introduce the
following notation:
the region
\begin{equation}\label{e:defA}
A_n = \left\{ x \in {\mathbb Z}_+^2: x(1) + x(2) \le n \right\}
\end{equation}
\nomenclature{$A_n$}{Domain of $X$}
and the exit boundary
\begin{equation}\label{e:defBoundaryA}
\partial A_n =
\left\{ x \in {\mathbb Z}_+^2: x(1) + x(2) = n \right\}.
\end{equation}
\nomenclature{$\partial A_n$}{Exit boundary of $A_n$}
$A_n^o$ denotes the interior $A_n -\partial_1 \cup \partial_2.$
\nomenclature{$A_n^o$}{Interior of $A_n$, }
Similarly, ${\mathbb Z}_+^{2,o}$ denotes ${\mathbb Z}_+^{2} -
\partial_1 \cup \partial_2.$
Let $\tau_n$ be the first time $X$ hits $\partial A_n$:
\begin{equation}\label{d:taun}
\tau_n \doteq \inf\{k \ge 0: X_k \in \partial A_n\}, n=0,1,2,3,..
\end{equation}
\nomenclature{${\mathbb Z}_+^{2,o}$}{${\mathbb Z}_+^{2} -\partial_1 \cup \partial_2$}
\nomenclature{$\tau_n$}{First hitting time of $X$ to $\partial A_n$}
Then the buffer overflow probability described above is
\begin{equation}\label{d:pn}
p_n(x,m) \doteq {\mathbb P}_{(x,m)} ( \tau_n < \tau_0).
\end{equation}
\nomenclature{$p_n(x,m)$}{${\mathbb P}_{(x,m)}(\tau_n < \tau_0)$}
\nomenclature{$\tau_0$}{First hitting time of $X$ to origin} 
The Markov property of $(X,M)$ implies that $p_n$ is $(X,M)$-harmonic
i.e., it satisfies
\[
p_n(x,m) = {\mathbb E}_{(x,m)}[p_n(x + \pi(x,I_1),M_1)], x \in A_n -\partial A_n,~~~ p_n(x,m) =1, x \in \partial A_n.
\]
This is a system of equations satisfied by $p_n$, where the number
of unknowns is in the order of $|{\mathcal M}|n^2$. 
More generally, for a $d$ dimensional
system the number of unknowns grows like $|{\mathcal M}|n^d$, 
making the computation
of $p_n$ via a direct solution of the linear system resource intensive
even for moderate values of $n$. This justifies the development of
approximations of $p_n$ and
the main goal of the present work is to find easily computable and accurate
approximations of $p_n$.
Stability and the bounded increments of $X$ suggest 
that when $x$ is away from the exit boundary $\partial A_n$,
$p_n$ decays exponentially in $n$, making the buffer overflow event
rare. 
The approximation of $p_n$, even when there is no modulation turns out to be 
a nontrivial problem.
There are two sources of difficulty: multidimensionality, and the 
discontinuous dynamics of the problem on the constraining boundaries. 
Asymptotically optimal importance sampling algorithms for the non-modulated setup
were constructed in \cite{DSW}, which
proposed a dynamic importance sampling algorithm based on 
subsolutions of a related Hamilton Jacobi Bellman  (HJB)
and its boundary conditions.
The approach of \cite{DSW} is tightly connected to the large deviations
analysis of $p_n$, which identifies the exponential decay rate of $p_n.$
Large deviations analysis is based on transforming 
$p_n$ to $V_n=-(1/n)\log p_n$, scaling space by $1/n$ and
\nomenclature{$V_n$}{$-1/n \log p_n$}
taking limits; the limit $V$ of $V_n$ satisfies the HJB equation mentioned
\nomenclature{$V$}{Limit of $V_n$}
above. 
The works \cite{sezer2015exit, sezer2018approximation} 
obtained sharp estimates of $p_n$ for the non-modulated two dimensional tandem 
walk using an affine transformation of the process $X$;
see Figure \ref{f:Tn2} and the summary below.
Another goal of the present work is to show that this affine
transformation approach can be extended to the analysis of $p_n$
of the Markov modulated constrained random walk.
As the present article shows, this extension turns out to be possible
but
Markov modulation complicates almost every aspect of the problem: the underlying functions, the geometry of the characteristic
surfaces, the limit analysis, etc. 
A detailed comparison with the non-modulated case is given in Section 
\ref{s:comparison}.

To the best our knowledge, there is very limited research on the analysis of 
the overflow probability $p_n$ for Markov modulated constrained
random walks; we are only aware of the article 
\cite{sezer2009importance} which develops asymptotically optimal importance sampling algorithms for the approximation of $p_n$ for the $(X,M)$ process studied in the present work.
In doing this, a necessary step is also to compute the large deviation decay rate of $p_n$; this was also done for $x=0$ in 
\cite{sezer2009importance}.
The analysis in this work is based on the sub and supersolutions of a limit HJB equation.
 Next is a summary of our analysis and main results.
\subsection{Summary of analysis and main results}
The starting point of our analysis is transforming $X$ to another
process $Y^n$
by an affine transformation
moving the origin to the point $(n,0)$ on the exit boundary;
as $n$ goes to infinity, $Y^n$ converges to the limit process $Y$ 
constrained only on $\partial_2$; 
 Figure \ref{f:Tn2} shows these transformations. 
\nomenclature{$(Y, M)$}{Markov modulated limit process}

\begin{figure}[h]
\begin{center}
\scalebox{0.8}{
\centerline{\input{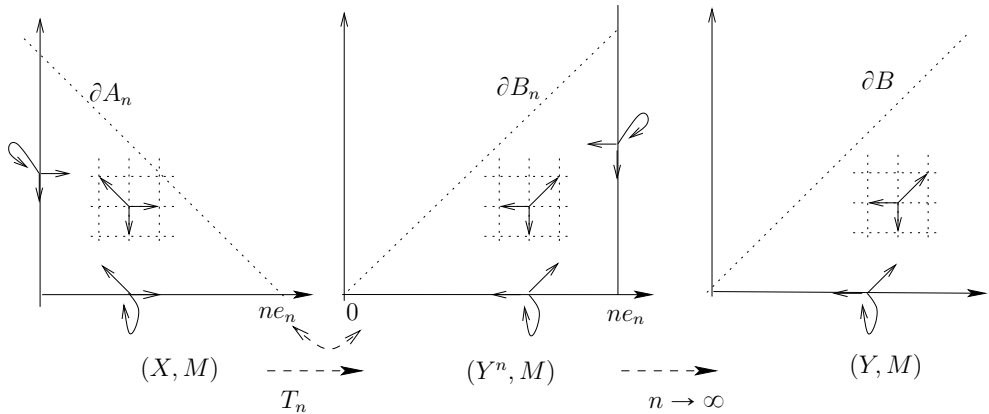}}}
\end{center}

\vspace{-0.65cm}
\caption{\hspace{0.25cm}The transformation of  $(X, M)$\label{f:Tn2}}
\end{figure}

The formal definition of the limit process $Y$ is as follows: define
\[
{\bm I}_2 \doteq \left( \begin{matrix}  -1 & 0 \\
0  & 1
\end{matrix}
\right).
\]
Define the constraining map
\[
\pi_1(y,v) = 
\begin{cases} v, &\text{ if } y + v \in {\mathbb Z}\times {\mathbb Z}_+,\\
0, &\text{otherwise.}
\end{cases}
\]
\nomenclature{$\pi_1$}{Constraining map for $Y$}
Then the limit process $Y$ is the $M$-modulated constrained random walk on 
${\mathbb Z} \times {\mathbb Z}_+$ 
with increments 
\begin{equation}\label{d:Jk}
J_k\doteq {\bm I}_2I_k:
\end{equation}
\[
Y_{k+1} = Y_k  + \pi_{1}(Y_k, J_k).
\]
\nomenclature{$Y$}{Limit process}
\nomenclature{$J_k$}{Increments of $Y$}
Define the region 
\[
B \doteq \left\{y \in {\mathbb Z} \times {\mathbb Z}_+: y(1) \ge y(2) \right\} 
\]
and the exit boundary
\[
\partial B\doteq 
\left\{y \in {\mathbb Z} \times {\mathbb Z}_+: y(1) = y(2) \right\}.
\]
Let $\tau$ be the hitting time
\[
\tau\doteq \inf \left\{k\ge 0: Y_k \in \partial B \right\}.
\]
$Y$ is a process constrained to ${\mathbb Z} \times {\mathbb Z}_+$
with the constraining boundary $\partial_2$; 
we will denote the interior ${\mathbb Z}\times{\mathbb Z}_+ -\partial_2$
of this set by
${\mathbb Z}\times {\mathbb Z}_+^o$.
Define the affine transformations
\[
T_n = n e_1 + {\bm I}_2
\]
\nomenclature{$T_n$}{Affine transformation}
where $(e_1,e_2)$ is the standard basis for ${\mathbb R}^2.$
Our main approximation result is the following:
\begin{theorem*}[Theorem \ref{t:mainapprox}]
For any
$x \in {\mathbb R}_+^2$, 
$x(1) + x(2) < 1$, $x(1) > 0$,
and $m \in {\mathcal M}$
there exist  constants $c>0$, $\rho \in (0,1)$ and $N > 0$ such that	
\begin{equation}\label{e:approximationerror}
\frac{ |{\mathbb P}_{(x_n,m)}(\tau_n < \tau_0) - 
{\mathbb P}_{(T_n(x_n),m)}(\tau <  \infty)|}{
{\mathbb P}_{(x_n,m)}(\tau_n < \tau_0)} < \rho^{c n},
\end{equation}
for $n > N$, where $x_n = \lfloor n x  \rfloor$.
\end{theorem*}

Theorem \ref{t:mainapprox} states that, as $n$ increases, ${\mathbb P}_{(T_n(x_n),m)}(\tau < \infty)$ gives a 
very good approximation of ${\mathbb P}_{(x_n,m)}(\tau_n < \tau_0)$. 
Parallel to the non-modulated case treated in \cite{ sezer2018approximation},
the proof of Theorem \ref{t:mainapprox} consists of the following steps
1) the difference between the events
$\{ \tau_n < \tau_0\}$ and  $\{\tau < \infty\}$ can be characterized by the event
``$X$ first hits $\partial_1$ then $\partial_2$ and then $\partial A_n$''~ 2) the probability of this detailed event
is very small compared to the probabilities of the events
$\{ \tau_n < \tau_0\}$ and  $\{\tau < \infty\}$. 
The challenges arise from the implementation of these
steps in the Markov modulated framework.

To bound the probabilities appearing in \eqref{e:approximationerror} 
we will use $(Y,M)$-(super)harmonic functions constructed
from single and conjugate points on a characteristic surface ${\mathcal H}$ (see \eqref{d:Hfull})
associated 
with $(Y,M).$ 
The characteristic surface is the $0$-level set of the characteristic polynomial of the characteristic
matrix ${\bm A}$ (see \eqref{d:defA}) defined
in terms of the transition matrix ${\bm P}$ and the jump probabilities
$\lambda(\cdot)$, $\mu_1(\cdot)$ and $\mu_2(\cdot).$ The characteristic polynomial
is of degree $3|{\mathcal M}|$ and therefore
the characteristic curve doesn't have a simple algebraic parametrization;
for this reason, in the modulated case, the identification of points on the characteristic surface relies on
the decomposition of the
the surface
into $|{\mathcal M}|$ components, by an eigenvalue analysis of ${\bm A}$
and the implicit function theorem.
The decomposition is given in subsection \ref{ss:geometry} and
the points relevant for our analysis are identified in
Propositions \ref{p:point1}, \ref{p:point2} and \ref{p:point2conj}.
These points all lie on the innermost
component corresponding to the largest eigenvalue of ${\bm A}.$

In the presence of a modulating Markov chain, harmonic functions are
constructed in general from $|{\mathcal M}|+1$ points on 
the characteristic surface, which makes
analysis based on them more complex. For this reason, we will switch
to superharmonic functions whenever we can, which can be constructed from
just two points.
An
upper bound for ${\mathbb P}_{(y, m)}(\tau<\infty)$ using these functions is given in Section \ref{s:ub1}.
Section \ref{s:ub2} constructs
an upper bound for the detailed event described above characterizing the difference 
of the events
$\{ \tau_n < \tau_0\}$ and  $\{\tau < \infty\}$. 
A lower bound for ${\mathbb P}_{(x, m)}(\tau_n < \tau_0)$ based on subharmonic functions
constructed from the functions of Section \ref{s:Yharmonic} is given  in Section \ref{s:lb}.
These elements are combined in Section \ref{s:together} to prove our main approximation theorem, Theorem \ref{t:mainapprox}.

With Theorem \ref{t:mainapprox} we know that 
${\mathbb P}_{(x, m)}(\tau_n < \tau_0)$ can be approximated very  well
with ${\mathbb P}_{(T_n(x), m)}(\tau<\infty)$. In the non-modulated case,
a linear combination of two $Y$-harmonic functions constructed from points on the characteristic surface
gives an exact formula for
${\mathbb P}_y(\tau < \infty)$.
This is no longer possible when there is modulation;
Sections \ref{s:computation} and \ref{s:improveapp} 
develops increasingly accurate approximate formulas for 
${\mathbb P}_{(y, m)}(\tau<\infty)$
using 
$(Y,M)$-harmonic functions constructed from further points
on the characteristic surface
under further linear independence assumptions 
(see \eqref{as:mu1neqmu2s2}, \eqref{as:bm1}),
see Propositions \ref{p:hrho2c} and Propositions \ref{p:hrho2star}
for the $(Y,M)$-harmonic functions constructed in these sections.
As opposed to the limit analysis which uses points only on the innermost 
component of the characteristic surface, the construction of harmonic functions uses points on
all components of the characteristic surface.
Propositions \ref{p:firstrelerror} and \ref{p:boundonrelerror} find
bounds on the relative error of the 
approximations of ${\mathbb P}_{(y,m)}(\tau < \infty)$ provided
by these functions based on the values they take on $\partial B.$
Section \ref{s:numeric} gives a numerical example 
showing the effectiveness of the resulting approximations.
Section \ref{s:comparison} compares the analysis of the current work with the non-modulated tandem walk treated in \cite{sezer2015exit, sezer2018approximation} 
and the non-modulated parallel walk treated in \cite{unlu2018excessive}.
Section \ref{s:conclusion} comments on future work.

\section{(sub/super)Harmonic functions of $(Y,M)$}\label{s:Yharmonic}

A function $h$ on ${\mathbb Z}\times{\mathbb Z}_+ \times {\mathcal M}$ 
is said to be $(Y,M)$-harmonic if
\begin{equation}\label{e:harmonicity}
{\mathbb E}_{(y,m)}[h(Y_1,M_1)]= h(y,m), 
(y,m)  \in {\mathbb Z} \times {\mathbb Z}_+ \times {\mathcal M};
\end{equation}
if we replace $=$ with $\ge$ [$\le$], $h$ is said to be $(Y,M)$-subharmonic
[superharmonic].

For the case $|{\mathcal M}| =1$ (i.e., no modulation), \cite{sezer2015exit, sezer2018approximation} use $Y$-harmonic functions which
are linear combinations of exponential functions
\begin{equation}\label{e:buildingblock0}
y \mapsto [(\beta,\alpha), y] = 
\beta^{y(1)-y(2)} \alpha^{y(2)}, (\beta,\alpha) \in {\mathbb C},
\end{equation}
and $(\beta,\alpha)$ lies on a characteristic surface associated with
the process.
Markov modulation introduces an additional state variable $m$, which
leads to the following generalization of  \eqref{e:buildingblock0}
\begin{equation}\label{e:buildingblock1}
(y,m) \mapsto \beta^{y(1)-y(2)} \alpha^{y(2)}{\bm d}(m),
\end{equation}
where ${\bm d} : {\mathcal M} \mapsto {\mathbb C}$ is an arbitrary function
on ${\mathcal M}.$
Let $[(\beta,\alpha,{\bm d}), \cdot]$ denote the function
given \eqref{e:buildingblock1}.
We would like to choose $(\beta,\alpha,{\bm d})$ so that 
$[(\beta,\alpha,{\bm d}), \cdot]$ is $(Y,M)$-harmonic at least
over the interior ${\mathbb Z} \times {\mathbb Z}_+^o.$
To this end, introduce the local characteristic polynomial
for the modulating state $m \in {\mathcal M}$:
\begin{equation}\label{d:defcharp}
{\bm p}(\beta,\alpha,m) \doteq \lambda(m) \frac{1}{\beta} +
\mu_1(m) \alpha + \mu_2(m) \frac{\beta}{\alpha};
\end{equation}
To define the global characteristic polynomial introduce the
$|{\mathcal M}| \times |{\mathcal M}|$ matrix ${\bm A}$:
\[
{\bm A}(\beta,\alpha)_{m_1,m_2} \doteq
\begin{cases}{\bm P}(m_1,m_2), &~~m_1 \neq m_2, \\
{\bm P}(m_1,m_1)  {\bm p}(\beta,\alpha,m), &~~m_1 = m_2,
\end{cases}
\]
$(m_1,m_2) \in {\mathcal M}^2.$
Let ${\bm I}$ denote the $|{\mathcal M}| \times |{\mathcal M}|$ identity
matrix.
Attempting to find functions of the form $[(\beta,\alpha,{\bm d}),\cdot]$
that satisfy \eqref{e:harmonicity} leads to the following characteristic
equation:
\begin{equation}\label{d:defA}
{\bm A}(\beta,\alpha){\bm d} = {\bm d},
\end{equation}
i.e,
\begin{equation}\label{d:charpoly}
{\bm p}(\beta,\alpha) \doteq \det({\bm I} - {\bm A}(\beta,\alpha)) = 0,
\end{equation}
\nomenclature{${\bm p}(\beta,\alpha)$}{global 
characteristic polynomial}
and ${\bm d}$ is an eigenvector of ${\bm A}(\beta,\alpha)$ for the eigenvalue $1$.
The ${\bm p}(\cdot,\cdot)$ of \eqref{d:charpoly} is the global characteristic
polynomial for the modulated process $(Y,M).$ Define the characteristic surface for the interior:
\begin{equation}\label{d:Hfull}
{\mathcal H} \doteq \left\{ (\beta,\alpha,{\bm d}) \in {\mathbb C}^{2 + |{\mathcal M}|}:
{\bm A}(\beta,\alpha){\bm d} = {\bm d},~ {\bm d} \neq 0 \right\}.
\end{equation}
\nomenclature{$I$}{$|{\mathcal M}| \times |{\mathcal M}|$
 identity matrix}
Points on ${\mathcal H}$ give us $(Y,M)$-harmonic functions on
$ {\mathbb Z} \times {\mathbb Z}_+^o.$
\begin{proposition}\label{p:interiorharm}
If $(\beta,\alpha,{\bm d}) \in {\mathcal H}$ then $[(\beta,\alpha,{\bm d}),\cdot]$
satisfies \eqref{e:harmonicity} for 
$y \in {\mathbb Z} \times {\mathbb Z}_+^o.$
\end{proposition}
\begin{proof}
By definition	
\begin{align*}
&{\mathbb E}_{(y,m)}\left[ (\beta,\alpha,{\bm d}), (Y_1, M_1) \right]\\
&~~~= \sum_{n \in {\mathcal M},n \neq m}  {\bm P}(m,n) [(\beta,\alpha,{\bm d}), (y,n)] \\
&~~~~~+ {\bm P}(m,m) ( \lambda(m)  [(\beta,\alpha,{\bm d}), (y + (-1,0), m)] + \mu_1(m) [(\beta,\alpha,{\bm d}), (y+(1,1), m)] \\
&~~~~~~+ \mu_2(m) [(\beta,\alpha,{\bm d}), (y+(0,-1), m)])
\intertext{Expand $[(\beta,\alpha,{\bm d}),((y+v), m)]$ terms:}
&~~~=\sum_{n \in {\mathcal M},n \neq m}  {\bm P}(m,n) [(\beta,\alpha,{\bm d}), (y,n)] \\
&~~~~~+ {\bm P}(m,m) ( \lambda(m) \beta^{y(1)-y(2)-1}\alpha^{y(2)} {\bm d}(m) +   \mu_1(m) \beta^{y(1)-y(2)} \alpha^{y(2)+1} {\bm d}(m)\\
&~~~~~~+ \mu_2(m) \beta^{y(1)-y(2) + 1} \alpha^{y(2)-1} {\bm d}(m) )
\intertext{Factor out $[(\beta,\alpha,{\bm d}),(y,m)]$ from the last three terms:}
&~~~=\sum_{n \in {\mathcal M},n \neq m}  {\bm P}(m,n) [(\beta,\alpha,{\bm d}),(y,n)] + {\bm P}(m,m)[(\beta,\alpha,{\bm d}),(y,m)] {\bm p}(\beta,\alpha, m)\\
&~~~= \beta^{y(1) - y(2)}\alpha^{y(2)} 
\left(
\sum_{n \in {\mathcal M},n \neq m}  {\bm P}(m,n) {\bm d}(n) + {\bm P}(m,m){\bm d}(m){\bm p}(\beta,\alpha, m) \right).
\intertext{
The expression in parenthesis equals the $m^{th}$ term of the vector ${\bm A}(\beta,\alpha){\bm d}$, which
	equals ${\bm d}(m)$ because $(\beta,\alpha,{\bm d}) \in {\mathcal H}$ means ${\bm A}(\beta,\alpha){\bm d} = {\bm d}.$ Therefore,}
&~~~= \beta^{y(1)-y(2)} \alpha^{y(2)} {\bm d}(m) = [(\beta,\alpha,{\bm d}), (y,m)].
\end{align*}	
This proves the claim of the proposition.
\end{proof}

The previous proposition gives us $(Y,M)$-harmonic functions 
on ${\mathbb Z}\times {\mathbb Z}_+^o.$
we next study the geometry of ${\mathcal H}$, this will be useful
in defining $(Y,M)$-(super/sub) harmonic functions over all
of ${\mathbb Z}\times{\mathbb Z}_+.$

\subsection{Geometry of the characteristic surface}\label{ss:geometry}
Define ${\mathcal H}^{\beta\alpha}$, 
the projection of ${\mathcal H}$ onto its first two dimensions:
\[
{\mathcal H}^{\beta\alpha} \doteq \{(\beta,\alpha) \in {\mathbb C}^2: {\bm p}(\beta,\alpha) = 0 \};
\]
we will to refer to
${\mathcal H}^{\beta\alpha}$ as the characteristic surface for the
interior as well, which is justified by the next lemma; its proof
follows from basic linear algebra.
\begin{lemma}\label{l:projection}
For each $(\beta,\alpha) \in {\mathcal H}^{\beta\alpha}$ there is at least one parameter family of points \\
$\{(\beta,\alpha,c{\bm d}),~ c \in {\mathbb C}-\{0\} \} \subset {\mathcal H}$, for some ${\bm d} \in {\mathbb C}^{|{\mathcal M}|}-\{0\}$. Conversely, for each $(\beta,\alpha,{\bm d}) \in {\mathcal H}$, we have $(\beta,\alpha) \in {\mathcal H}^{\beta\alpha}.$
Furthermore, all points on ${\mathcal H}$ can be obtained from those on ${\mathcal H}^{\beta\alpha}.$
\end{lemma}

$\beta^{|{\mathcal M}|}\alpha^{|{\mathcal M}|}{\bm p}$ is a polynomial of degree $3|{\mathcal M}|$ in
$(\beta,\alpha)$, which makes, in general, the analysis of the geometry
of ${\mathcal H}^{\beta\alpha}$ nontrivial. A natural
approach to the study of the geometry of this curve is through
the eigenvalues of ${\bm A}(\beta,\alpha)$. 
The next two propositions show that
the curve 
${\mathcal H}^{\beta\alpha}$ decomposes into $|{\mathcal M}|$ distinct
pieces over any region $D$
where ${\bm A}(\beta,\alpha)$ has simple eigenvalues.
\begin{proposition}\label{p:eigfun}
Let $D \subset{\mathbb C}^2$ or $D \subset {\mathbb R}^2$ 
be open and simply connected and
suppose ${\bm A}(\beta,\alpha)$ has simple eigenvalues for all
$(\beta,\alpha) \in D.$ Then the eigenvalues of ${\bm A}$ can be written
as $|{\mathcal M}|$ distinct smooth functions $\Lambda_j(\beta,\alpha)$ on $D$.
\end{proposition}
\nomenclature{$\Lambda_j$}{$j^{th}$ eigenvalue of ${\bm A}(\beta,\alpha)$}
\begin{proof}
The argument is the same for both real and complex variables.
That 
the eigenvalues
$\Lambda_j$ can be defined smoothly in a neighborhood of any
$(\beta,\alpha) \in D$ follows from \cite[Theorem 5.3]{serrematrices} 
and the assumption
that they are distinct. Once defined locally, one extends
them to all of $D$ through continuous extension, which is possible
because $D$ is simply connected.
\end{proof}
Most of our analysis will be based on 
$(\beta,\alpha) \in D= {\mathbb R}_+^{2,o} \doteq {\mathbb R}_+^2 - \partial_1 \cup \partial_2$. 
For $(\beta,\alpha) \in {\mathbb R}_+^{2,o}$, ${\bm A}(\beta,\alpha)$ is
an irreducible matrix with positive entries. Perron-Frobenius Theorem implies
that ${\bm A}(\beta,\alpha)$ has a simple positive eigenvalue dominating all of the
other eigenvalues in absolute value with an eigenvector with strictly positive entries. 
$\Lambda_1(\beta,\alpha)$ will always denote this largest eigenvalue.
Furthermore, if
${\bm A}(\beta,\alpha)$ has distinct real eigenvalues for $ (\beta,\alpha)
\in {\mathbb R}_+^{2,o}$,
we will label them so that
\[
\Lambda_j(\beta,\alpha) >  \Lambda_i(\beta,\alpha), \text{ for } j < i,
\]
i.e., the eigenvalues are assumed to be sorted in descending order. 

For $D$ and $\Lambda_j$ as in Proposition \ref{p:eigfun} define
\[
{\mathcal L}_j^D \doteq \{ (\beta,\alpha) \in D: \Lambda_j(\beta,\alpha)
=1\}.
\]
\nomenclature{${\mathcal L}_j^D$}{Level curve of the $j^{th}$ eigenvalue of 
${\bm A}$ lying in $D$}
The last proposition implies
\begin{proposition}\label{p:decomposition1}
Let $D$ and $\Lambda_j$, $j=1,2,3,..., |{\mathcal M}|$, be as in Proposition
\ref{p:eigfun}. Then
\begin{equation}\label{e:decompositionHba}
D \cap {\mathcal H}^{\beta\alpha} = \sqcup_{j=1}^{|{\mathcal M}|} 
{\mathcal L}_j^D
\end{equation}
where $\sqcup$ denotes disjoint union.
\end{proposition}
The proof follows from the definitions involved. 
\nomenclature{${\mathbb R}_+^{2,o}$}{${\mathbb R}_+^2 - \partial_1 \cup \partial_2$}
For $D = {\mathbb R}_+^{2,o}$ we will omit the superscript
$D$ and write ${\mathcal L}_j$ instead of ${\mathcal L}_j^{{\mathbb R}_+^{2,o}}.$
\nomenclature{${\mathcal L}_j$}{ ${\mathcal L}_j^{{\mathbb R}_+^{2,o}}$}

If ${\bm A}(\beta,\alpha)$ has simple real eigenvalues for
$(\beta,\alpha) \in {\mathbb R}_+^{2,o}$ we can define
\[
{\mathcal R}_j \doteq \{ (\beta,\alpha) \in {\mathbb R}_+^{2,o}: 
\Lambda_j(\beta,\alpha) \le 1\}.
\]
The continuity of $\Lambda_j$ implies
${\mathcal L}_j = \partial {\mathcal R}_j.$
\nomenclature{${\mathcal R}_j$}{ Subset of ${\mathbb R}_+^{2,o}$ where $\Lambda_j(\beta,\alpha) \le 1$}

\begin{proposition}\label{p:componentsofH}
Suppose ${\bm A}(\beta,\alpha)$ has simple real eigenvalues for
$(\beta,\alpha) \in {\mathbb R}_+^{2,o}.$
Then the curve ${\mathcal L}_{j}$ is strictly contained
inside the curve ${\mathcal L}_{j+1}$ for $j=1,2,...,|{\mathcal M}|-1.$
\end{proposition}
\begin{proof}
All diagonal entries of ${\bm A}(\beta,\alpha)$ tend to $\infty$
when $(\beta,\alpha) \rightarrow \partial {\mathbb R}_+^{2,o}$.
This and
Gershgorin's Theorem \cite[Appendix 7]{lax1996linear}, imply 
$\Lambda_j(\beta,\alpha) \rightarrow \infty$ for
$(\beta,\alpha) \rightarrow \partial {\mathbb R}_+^{2,o}$.
This implies in particular that ${\mathcal R}_j$ is a compact subset
of ${\mathbb R}_+^{2,o}$. Secondly,
$\Lambda_{j+1} < \Lambda_j$ implies 
${\mathcal R}_{j} \subset {\mathcal R}_{j+1}$; the compactness
of these sets,
the strictness of the inequality $\Lambda_{j+1}(\beta,\alpha) < \Lambda_j(\beta,\alpha)$ imply that $\partial {\mathcal R}_j = {\mathcal L}_j$ lies
strictly within ${\mathcal R}_{j+1}$ with strictly positive distance from the
boundary ${\mathcal L}_{j+1}$ of ${\mathcal R}_{j+1}$; this proves
the claim of the proposition.
\end{proof}

In Sections \ref{s:computation} and \ref{s:improveapp} we will employ the following assumption and
the above decomposition of ${\mathcal H}^{\beta\alpha}$ to identify
points on ${\mathcal H}^{\beta\alpha}$ to be used
in the construction of $(Y,M)$-harmonic functions:
\begin{assumption}\label{as:simpleeig}
${\bm A}(\beta,\alpha)$ has real distinct eigenvalues for 
$(\beta,\alpha) \in {\mathbb R}_+^{2,o}.$
\end{assumption}

To show that
Assumption \ref{as:simpleeig}
is not vacuous,
we now give a class of matrices ${\bm A}$ that satisfies it.
The following definitions are from \cite[page 57]{berman1994nonnegative}:
a matrix is said to be totally nonnegative (totally positive) if all of its minors of any degree are 
nonnegative (positive). A  totally nonnegative matrix is said to be {\em oscillatory} if some positive integer power
of the matrix is totally positive.
If ${\bm A}$ is oscillatory, Assumption \ref{as:simpleeig} holds:

\begin{proposition}
Suppose ${\bm A}(\beta,\alpha)$ is an oscillatory matrix for
all $(\beta,\alpha) \in {\mathbb R}_+^{2,o}$, then 
${\bm A}(\beta,\alpha)$ has $|{\mathcal M}|$ distinct eigenvalues over ${\mathbb R}_+^{2,o}.$
\end{proposition}

This proposition is a basic fact on oscillatory
matrices \cite[(6.28)]{berman1994nonnegative}. \cite[(6.26)]{berman1994nonnegative} identifies a particularly simple
class of oscillatory matrices:

\begin{proposition}\label{p:conditionforsimple}
Suppose ${\bm G}(1,1)$, ${\bm G}(1,2)$, ${\bm G}(|{\mathcal M}|-1,|{\mathcal M}|)$, ${\bm G}(|{\mathcal M}|,|{\mathcal M}|)$ and
${\bm G}(j,j-1)$, ${\bm G}(j,j)$, ${\bm G}(j,j+1)$, $j=2,3,...,|{\mathcal M}|-1$ are all strictly
positive and the rest of the components of ${\bm G}$ are all zero, i.e.,
${\bm G}$ is tridiagonal with strictly positive entries.
Then ${\bm G}$ is an oscillatory matrix. 
\end{proposition}

We will call any tridiagonal matrix with strictly positive entries
on the three diagonals ``strictly tridiagonal.'' By the above
proposition any strictly tridiagonal matrix is oscillatory.
In particular, if the transition matrix ${\bm P}$ is strictly tridiagonal, ${\bm A}(\beta, \alpha)$ will also be of the same form for all $(\beta,\alpha)\in {\mathbb R}^2_+$; therefore, for such ${\bm P}$ Assumption \ref{as:simpleeig} 
holds.

The decomposition of ${\mathcal H}^{\beta\alpha} \cap {\mathbb R}_+^{2,o}$ 
into ${\mathcal L}_j$ is
shown in Figure \ref{f:levelcurves2} for 
the transition matrix
\begin{equation}\label{e:Pexample}
{\bm P} = \left( \begin{matrix}
0.6 & 0.4 &0\\ 0.1 & 0.4 & 0.5\\ 0 & 0.2 & 0.8\end{matrix}
\right).
\end{equation}

The matrix ${\bm P}$ of \eqref{e:Pexample}
is strictly tridiagonal; therefore, Proposition  \ref{p:conditionforsimple}
applies and ${\bm A}(\beta,\alpha)$ has distinct real eigenvalues for
all $(\beta,\alpha) \in {\mathbb R}_+^{2,o}$ and we have the decomposition 
\eqref{e:decompositionHba}
of
${\mathcal H}^{\beta\alpha} \cap {\mathbb R}_+^{2,o}$ 
given by Propositions \ref{p:decomposition1} and \ref{p:componentsofH}; Figure \ref{f:levelcurves2}
shows ${\mathcal H}^{\beta\alpha}$ and its components
${\mathcal L}_j$; the jump probabilities for this example are
\begin{equation}\label{e:Ratesexample}
\left(
\begin{matrix} 0.1 & 0.4 & 0.5\\ 0.12 &0.41 &0.47\\ 0.09 & 0.39 & 0.52
\end{matrix}
\right)
\end{equation}
where the $i^{th}$ row equals $(\lambda(i),\mu_1(i),\mu_2(i))$.
\begin{figure}[h]
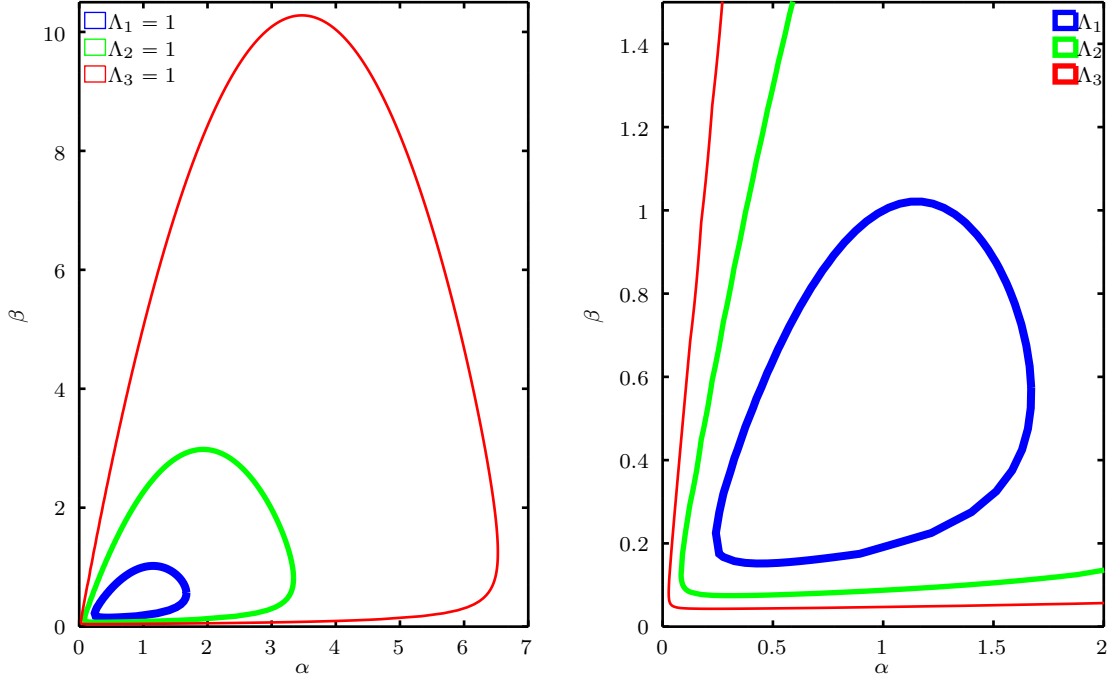

\hspace{-0.5cm}
\begin{subfigure}{0.48\textwidth}
\input{levelcurve}
\end{subfigure}
\begin{subfigure}{0.48\textwidth}
\input{levelcurvedetail}
\end{subfigure}
\caption{Real section of the characteristic surface
${\mathcal H}^{\beta\alpha} = \cup_{j=1}^3 {\mathcal L}_j$
for the parameter values given in \eqref{e:Pexample} and \eqref{e:Ratesexample}.
On the right: detailed graph around the origin}\label{f:levelcurves2}
\end{figure}
\FloatBarrier
By Proposition  \ref{p:interiorharm} and Lemma \ref{l:projection}
each point on any of the curves depicted in Figure \ref{f:levelcurves2}
gives a $Y$-harmonic function on ${\mathbb Z} \times {\mathbb Z}_+^o.$
Most of our analysis will be based on points on the innermost curve
${\mathcal L}_1$, the $1$-level curve of the largest eigenvalue
$\Lambda_1$; before identifying the relevant points,
let us look at two different methods of constructing
$(Y,M)$-(super)harmonic functions from points on ${\mathcal H}^{\beta\alpha}.$
\nomenclature{${\mathcal L}_1$}{Innermost component of ${\mathcal H}^{\beta\alpha}$, the 1-level curve of the largest eigenvalue $\Lambda_1$}

\subsection{Construction of $(Y,M)$-harmonic and superharmonic functions}\label{ss:YMharm}
We can proceed in two ways
to get functions that satisfy \\
${\mathbb E}_{(y,m)}[h(Y_1,M_1)]= h(y,m)$ 
or
${\mathbb E}_{(y,m)}[h(Y_1,M_1)]\le h(y,m)$
for $ y \in \partial_2$
as well as the interior.
The first is by defining
the characteristic polynomial ${\bm p}_2$, the
\nomenclature{${\bm p}_2$}{Characteristic polynomial for $\partial_2$}
boundary matrix ${\bm A}_2$, and the boundary surface ${\mathcal H}_2$ associated with $\partial_2$
and using points on ${\mathcal H} \cap {\mathcal H}_2$:
\nomenclature{${\bm A}_2$}{Characteristic matrix for $\partial_2$}
\begin{align}
{\bm p}_2(\beta,\alpha,m) &\doteq \lambda(m) \frac{1}{\beta} +
\mu_1(m) \alpha + \mu_2(m),~~m \in {\mathcal M}, \notag \\
{\bm A}_2(\beta,\alpha)_{m_1,m_2} &\doteq
\begin{cases} {\bm P}(m_1,m_2), &~~~m_1 \neq m_2 \\
{\bm P}(m_1,m_1)  {\bm p}_2(\beta,\alpha,m), &~~m_1 = m_2,
\end{cases},
(m_1,m_2) \in {\mathcal M}^2, \notag
\\
{\mathcal H}_2 &\doteq \left\{ (\beta,\alpha,{\bm d}) \in {\mathbb C}^{2+|{\mathcal M}|}:
{\bm A}_2(\beta,\alpha)){\bm d} = {\bm d}, {\bm d}\neq 0 \right\}.
\label{d:H2}
\end{align}
\nomenclature{${\mathcal H}_2$}{Characteristic surface for $\partial_2$}
Define $\Lambda_{2,1}(\beta,\alpha)$ to be the largest eigenvalue of ${\bm A}_2(\beta,\alpha)$.
\nomenclature{$\Lambda_{2,1}$}{Largest eigenvalue of ${\bm A}_2$ when $(\beta,\alpha) \in {\mathbb R}_+^2$}
Parallel to the interior case, define 
\begin{align*}
{\mathcal H}^{\beta\alpha}_2 &\doteq \{(\beta,\alpha) \in {\mathbb C}^2: {\bm p}_2(\beta,\alpha) = 0 \},\\
{\mathcal L}_{2,1} &\doteq
\{(\beta,\alpha) \in {\mathbb R}^2_+: \Lambda_{2,1}(\beta,\alpha) = 1 \}.
\end{align*}
\nomenclature{${\mathcal L}_{2,1}$}{Level curve of $\Lambda_{2,1} $ in ${\mathbb R}_+^2$}
\nomenclature{${\mathcal H}^{\beta\alpha}_2$}{Projection of ${\mathcal H}_2$ onto its first two coordinates}
\begin{proposition}\label{p:pointonboth}
$[(\beta,\alpha,{\bm d}),\cdot]$ is $(Y,M)$-harmonic if 
$(\beta,\alpha,{\bm d}) \in {\mathcal H} \cap {\mathcal H}_2.$
\end{proposition}
\begin{proof}
Proposition \ref{p:interiorharm} says that for $(\beta,\alpha,{\bm d}) \in {\mathcal H}$, $[(\beta,\alpha,{\bm d}),\cdot]$ satisfies the harmonicity condition when $y \in {\mathbb Z} \times {\mathbb Z}_+ -\partial_2$. Similar to the proof of Proposition \ref{p:interiorharm}, we would like to show that $[(\beta,\alpha,{\bm d}),\cdot]$ is $(Y,M)$-harmonic on $\partial_2$ when $(\beta,\alpha,{\bm d}) \in {\mathcal H}_2.$
By definition
\begin{align*}
&{\mathbb E}_{(y,m)}\left[ (\beta,\alpha,{\bm d}), (Y_1, M_1) \right]\\
&~~~= \sum_{n \in {\mathcal M},n \neq m}  {\bm P}(m,n) [(\beta,\alpha,{\bm d}), (y,n)] \\
&~~~~~+ {\bm P}(m,m) ( \lambda(m)  [(\beta,\alpha,{\bm d}), (y + (-1,0),m)] + \mu_1(m) [(\beta,\alpha,{\bm d}), (y+(1,1),m] \\
    &~~~~~+ \mu_2(m) [(\beta,\alpha,{\bm d}), (y,m)])\\
&~~~=\sum_{n \in {\mathcal M},n \neq m}  {\bm P}(m,n) [(\beta,\alpha,{\bm d}), (y,n)] \\
&~~~~~+ {\bm P}(m,m) ( \lambda(m) \beta^{y(1)-1} {\bm d}(m) +   \mu_1(m) \beta^{y(1)} \alpha {\bm d}(m) + \mu_2(m) \beta^{y(1)} {\bm d}(m) )\\
&~~~=\sum_{n \in {\mathcal M},n \neq m}  {\bm P}(m,n) [(\beta,\alpha,{\bm d}), (y,n)] + {\bm P}(m,m)[(\beta,\alpha,{\bm d}), (y,m)] {\bm p}_2(\beta,\alpha, m)\\
&~~~= \beta^{y(1)} 
\left(
\sum_{n \in {\mathcal M},n \neq m}  {\bm P}(m,n) {\bm d}(n) + {\bm P}(m,m){\bm d}(m){\bm p}_2(\beta,\alpha, m) \right).
\intertext{
	The expression in parenthesis equals the $m^{th}$ term of the vector ${\bm A}_2(\beta,\alpha)){\bm d}$, which
	equals ${\bm d}(m)$ because $(\beta,\alpha,{\bm d}) \in {\mathcal H}_2$ means ${\bm A}_2(\beta,\alpha)){\bm d} = {\bm d}.$ Therefore,}
&~~~= \beta^{y(1)} {\bm d}(m) = [(\beta,\alpha,{\bm d}), (y,m)].
\end{align*}
This argument and Proposition \ref{p:interiorharm} prove the claim of the proposition.
\end{proof}

The real sections of
${\mathcal H}^{\beta\alpha}$ and ${\mathcal H}_2^{\beta\alpha}$ are $1$ dimensional curves
and their intersection will in general consist of finitely many points. In the analysis of the tandem walk with no modulation, these points can easily be identified
explicitly. There turn out to be three of them, 
of which only one is nontrivial (i.e., different from $0$ and $1$).
In the present
case, there will in general be $3|{\mathcal M}|-2$ nontrivial points
on ${\mathcal H}^{\beta\alpha} \cap {\mathcal H}^{\beta\alpha}_2$;
one of these which lies on ${\mathcal L}_1 \cap {\mathcal L}_{2,1}$
can be identified using the implicit function theorem
and the stability assumption \eqref{e:stabilityas};
this point and the $(Y,M)$-harmonic function it defines are given in 
Proposition \ref{p:point1} and \ref{p:hrho1} below. 
For the argument we need two auxiliary linear algebra results,
Lemmas \ref{l:positivity} and \ref{l:deteigenvec} given
in the appendix.

\begin{proposition}\label{p:point1}
Under the stability assumption (\ref{e:stabilityas}) there exists unique $ 0 < \rho_1 <1$ such that
$(\rho_1,\rho_1) \in 
{\mathcal L}_1 \cap {\mathcal L}_{2,1} \subset
{\mathcal H}^{\beta\alpha} \cap {\mathcal H}_2^{\beta\alpha}$,
i.e., $1$ is the largest eigenvalue of ${\bm A}(\rho_1,\rho_1)$ and ${\bm A}_2(\rho_1,\rho_1).$ 
\end{proposition}
\nomenclature{$\rho_1$}{$(\rho_1,\rho_1)$, $\rho_1 \in (0,1)$ 
is a point on identified on ${\mathcal L}_1 \cap {\mathcal L}_{2,1}$}
\begin{proof}
For $q  \in {\mathbb R}^2$ define
\begin{equation}\label{d:H}
H(q) \doteq -\log \Lambda_1\left(e^{q(1)},e^{q(2)} \right).
\end{equation}
\nomenclature{$H$}{The hamiltonian function}
By \cite[Lemma 4.2, 4.3]{sezer2009importance}, $H$ is convex in $q$.
Proceeding parallel to \cite[Proof of Lemma 4.4, page 515]{sezer2009importance}
define $f(\Lambda,r) \doteq \det(\Lambda \bm I - {\bm A}(e^r,e^r)).$
We know that $f(\Lambda_1(e^r,e^r),r) = 0$ for $r \in {\mathbb R}.$
To prove our proposition, we will apply the implicit function theorem to 
$f$ at $(1,0)$ to prove that $r \mapsto \Lambda_1(e^r,e^r)$ is strictly increasing
at $ r=0.$
Differentiating $f$ at $(1,0)$ with respect to $r$
gives
\begin{align*}
\left.\frac{\partial f}{\partial r}\right|_{(1,0)} 
&=
\sum_{m \in {\mathcal M}} (\lambda(m) - \mu_1(m)) {\bm P}(m,m) \det({\bm I} - 
{\bm P})^{m,m},
\intertext{which equals, by Lemma \ref{l:deteigenvec}, 
for some constant $c > 0$,}
&= c \sum_{m \in {\mathcal M}} (\lambda(m) - \mu_1(m)) {\bm P}(m,m) \bm\pi(m) \\
&< 0
\end{align*}
where the last inequality follows from the stability assumption
(\ref{e:stabilityas}). Similarly, differentiation of $f$ at $(1,0)$ 
with respect to $\Lambda$ gives:
\[
\left.\frac{\partial f}{\partial \Lambda}\right|_{(1,0)}  = 1.
\]
This implies that the implicit function theorem is applicable to $f$;
the last two display give:
\[
\frac{d}{dr} \Lambda_1(e^{r},e^r)|_{(0,0)}  > 0.
\]
On the other hand, 
Gershgorin's Theorem implies $\Lambda_1(e^r,e^r) \rightarrow \infty$
as $r\rightarrow -\infty$ (because of the $\lambda(m)/\beta$ term 
appearing in the diagonal terms of ${\bm A}$, tending
to $+\infty$ with $\beta= e^r$). 
To sum up: we have that $\Lambda_1(e^r,e^r)$ is strictly monotone 
at $r=0$ (decreases when $r$ decreases)
and it tends to infinity as $r\rightarrow -\infty.$ 
Then, by the continuity of $\Lambda_1$, there must exist at least one point 
in $(-\infty,0)$ where $\Lambda_1(e^r,e^r)$ takes the value $1$;
the convexity of $H$ implies that such a point is unique, i.e.,
there is a unique point $r^* < 0$ such that $\Lambda_1(e^{r^*},e^{r^*}) = 1.$
Setting $\rho_1 =e^{r^*}$ proves the proposition.
\end{proof}
Let ${\bm d}_1$ be an eigenvector of ${\bm A}(\rho_1,\rho_1)$ corresponding
to the eigenvalue $1$; because $1$ is the largest eigenvalue
of ${\bm A}(\rho_1,\rho_1)$ and because ${\bm A}(\rho_1,\rho_1)$ 
is irreducible and aperiodic,
we can choose ${\bm d}_1$ so that all of its components are strictly
positive.  
The point $(\rho_1,\rho_1, {\bm d}_1) \in {\mathcal H}\cap {\mathcal H}_2$
and Proposition \ref{p:pointonboth} 
give us our first $(Y,M)$-harmonic function:
\begin{proposition}\label{p:hrho1}
\begin{equation}\label{d:hrho1}
h_{\rho_1} \doteq [ (\rho_1,\rho_1, {\bm d}_1),\cdot]
\end{equation}
is $(Y,M)$-harmonic.
\end{proposition}
\nomenclature{$h_{\rho_1}$}{ $[(\rho_1,\rho_1,{\bm d}_1,\cdot]$}

The second way of obtaining $(Y,M)$-harmonic functions is through 
conjugate points on ${\mathcal H}^{\beta\alpha}.$ The function
$\alpha^{|{\mathcal M}|}{\bm p}$ is a polynomial of degree $2|{\mathcal M}|$
in $\alpha.$
By the fundamental of theorem of algebra, 
$\alpha^{|{\mathcal M}|}{\bm p}$ 
has $2|{\mathcal M}|$ roots,
${\bm \alpha}_1(\beta)$, ..., ${\bm \alpha}_2(\beta)$,..., ${\bm \alpha}_{2|{\mathcal M}|}(\beta)$,
in ${\mathbb C}$ for each fixed $\beta \in {\mathbb C}$; points $(\beta,{\bm \alpha}_i) \in {\mathcal H}^{\beta\alpha}$, 
$i=1,2,...,2|{\mathcal M}|$
are said to be conjugate points.
In the non-modulated case, i.e., when $|{\mathcal M}| = 1$,  $\alpha {\bm p}$ is only of second order, therefore, the conjugate points come in pairs,
and given one of the points in the pair, the other can be computed easily; in the modulated case, there are obviously no simple
formulas to obtain all of the conjugate points given one among them, because computation of conjugate points involves finding
the roots of a polynomial of degree $2|{\mathcal M}|$.

For $(\beta,\alpha,{\bm d}) \in {\mathcal H}$ define
\begin{equation}\label{d:C}
{\bm c}(\beta,\alpha,{\bm d})  \in {\mathbb C}^{{\mathcal M}},~ 
{\bm c}(\beta,\alpha,{\bm d})(m)  \doteq {\bm P}(m,m)\mu_2(m) {\bm d}(m)\left( 1 - \frac{\beta}{\alpha}\right).
\end{equation}
\nomenclature{${\bm c}(\beta,\alpha,{\bm d})$}{function of $(\beta,\alpha,{\bm d})$ taking values in ${\mathbb C}^{\mathcal M}$; determines
how much $[(\beta,\alpha,{\bm d}),\cdot]$ deviates from harmonicity on $\partial_2$,  \pageref{l:roleofC}}
One can take linear combinations of functions defined by 
conjugate points to define $(Y,M)$-harmonic functions.
This is based on the following lemma

\begin{lemma}\label{l:roleofC}
Suppose $(\beta,\alpha,{\bm d}) \in {\mathcal H}.$ Then, for $(y,m) \in \partial_2 \times {\mathcal M}$,
\begin{equation}\label{e:l1}
{\mathbb E}_{(y,m)}\left[(\beta,\alpha,{\bm d}),(Y_1, M_1)\right] - [(\beta,\alpha,{\bm d}), (y,m)] = \beta^{y(1)} {\bm c}(\beta,\alpha,{\bm d})(m),
\end{equation}	
where ${\bm c}$ is defined as in \eqref{d:C}.
\end{lemma}
\begin{proof}
The computation in the proof of Proposition \ref{p:interiorharm} gives
\begin{align}\label{e:onesteponp2}
&{\mathbb E}_{(y,m)}\left[ (\beta,\alpha,{\bm d}), (Y_1, M_1) \right]\\
&\notag=
\beta^{y(1)}\left(
\sum_{n \in {\mathcal M},n \neq m}  {\bm P}(m,n) {\bm d}(n) + {\bm P}(m,m){\bm d}(m){\bm p}_2(\beta,\alpha, m) \right).
\end{align}	
On the other hand, $(\beta,\alpha,{\bm d}) \in {\mathcal H}$ means	
\begin{align}\label{e:sumrep}
&[(\beta,\alpha,{\bm d}), (y,m)]\\ 
&\notag= \beta^{y(1)}  {\bm d}(m) = \beta^{y(1)} \left( \sum_{n \in {\mathcal M},n \neq m}  {\bm P}(m,n) {\bm d}(n) + {\bm P}(m,m) {\bm d}(m){\bm p}(\beta,\alpha,m) \right).
\end{align}	
Subtracting the last display from \eqref{e:onesteponp2} gives	
\begin{align*}
&{\mathbb E}_{(y,m)}[h(Y_1,M_1)] - h(y,m)\\& = 
\beta^{y(1)}  {\bm P}(m,m)\mu_2(m){\bm d}(m) \left(1 - \frac{\beta}{\alpha} \right) = \beta^{y(1)}{\bm c}(\beta,\alpha,{\bm d})(m),
\end{align*}	
which proves \eqref{e:l1}.
\end{proof}

We now identify a family of $(Y,M)$-harmonic functions constructed from
conjugate points on ${\mathcal H}$:
\begin{proposition}\label{p:harconjugate}
For $\beta \in {\mathbb C}$ 
let $(\beta,\alpha_i,d_i)$ $i=1,2,...,l$ $ \le 2|{\mathcal M}|$ 
be distinct conjugate
points on ${\mathcal H}$. Take any subcollection $\{i_1,i_2,...,i_k\}$, $k \le l$ such that
${\bm c}(\beta,\alpha_{i_j},d_{i_j})$  are linearly dependent, i.e.,
there exists $b \in {\mathbb C}^k$, $b \neq 0$, such that	
\begin{equation}\label{e:linearcombC}
\sum_{j=1}^k b(j){\bm c}(\beta,\alpha_{i_j},d_{i_j}) = 0.
\end{equation}
Then
\begin{equation}\label{e:linearhar}
h(y,m) = \sum_{j=1}^{k} b(j) [(\beta,\alpha_{i_j},d_{i_j}),\cdot]
\end{equation}
is $(Y,M)$-harmonic.
\end{proposition}

\begin{proof}
We already know from Proposition \ref{p:interiorharm}, 
harmonic functions of the form\\ $ [(\beta,\alpha_i,d_i),\cdot]$ 
are $(Y,M)$-harmonic in the interior
${\mathbb Z} \times{\mathbb Z}_+ - \partial_2$. 
So, their linear combinations are also $(Y,M)$-harmonic in the interior
and we need to check the harmonicity for $y\in\partial_2$.
By Lemma \ref{l:roleofC}
\begin{equation*}
{\mathbb E}_{(y,\cdot)}\left[ (\beta,\alpha_i,d_i), (Y_1,M_1) \right]
-[ (\beta,\alpha_i,d_i), (y,\cdot)]
= 
\beta^{y(1)}{\bm c}(\beta,\alpha_i,d_i).
\end{equation*}	
Taking linear combinations of these with weight vector $b$ gives:
\[
{\mathbb E}_{(y,\cdot)}\left[ h(Y_1,M_1)\right] -h(y,\cdot)=
\beta^{y(1)} \left(
\sum_{j=1}^k b(j) {\bm c}(\beta,\alpha_{i_j}, d_{i_j}) \right)
\]
which equals $0 \in {\mathbb R}^{|{\mathcal M}|}$ by \eqref{e:linearcombC}.
This proves that $h$ is $(Y,M)$-harmonic on $\partial_2.$
\end{proof}

For any $\beta \in {\mathbb C}$ 
such that ${\bm p}(\beta,\alpha)=0$ has distinct roots,
$\alpha_1$, $\alpha_2$,...,$\alpha_{2|{\mathcal M}|}$,
all different from $\beta$,
we have, by definition,
${\bm c}(\beta,\alpha_j,d_j) \neq 0$ for all $j=1,2,...,2|{\mathcal M}|$.  Therefore,
for such $\beta$, and for any subcollection $\alpha_{j_1},\alpha_{j_2},...,
\alpha_{j_k}$, with $k \ge |{\mathcal M}| +1$,
we can find a nonzero vector $b$ satisfying \eqref{e:linearcombC}.

We will call a $(Y,M)$-harmonic function $\partial B$-determined
if it of the form,
\[
(y,m) \mapsto {\mathbb E}_{(y,m)}[ f(Y_\tau,M_\tau) 1_{\{\tau<\infty\}}]
\]
for some function $f$. 
The function $(y,m)\mapsto {\mathbb P}_{(y,m)}(\tau < \infty)$ is the unique $\partial B$-determined $(Y,M)$-harmonic function
taking the value $1$ on $\partial B$. Among the functions of the form $[(\beta,\alpha, {\bm d}), \cdot]$, the closest
we get to this type of behavior is when $\alpha=1$: for $\alpha=1$,
 $[(\beta,1,{\bm d}),(y,m)]$ depends only on $m$ for $y \in \partial B$.
Therefore, $\alpha=1$ play a key role in computing/approximating 
${\mathbb P}_{(y,m)}(\tau < \infty)$.
The next proposition identifies a point on ${\mathcal L}_1$
of the form $(\rho_2,\alpha=1)$ with $0 < \rho_2 < 1$.

\begin{proposition}\label{p:point2}
Under assumption (\ref{e:stabilityas}) there exists $ 0 < \rho_2 <1$ such that
$(\rho_2,1) \in {\mathcal L}_1 \subset {\mathcal H}^{\beta\alpha}$;
i.e., $1$ is the largest eigenvalue of  ${\bm A}(\rho_2,1)$.
\end{proposition}
\nomenclature{$\rho_2$}{ $0 < \rho_2 < 1$, the $\beta$ component of a point of the form $(\rho_2,1)$ identified
on ${\mathcal L}_1$}
\begin{proof}
The proof is parallel to that of Proposition \ref{p:point1}. We now define
$f(\Lambda,r) = \det(\Lambda {\bm I} - {\bm A}(e^r,1))$ and observe, by assumption
(\ref{e:stabilityas}) and Lemma \ref{l:deteigenvec},
\begin{align*}
\left.\frac{\partial f}{\partial r}\right|_{(1,0)} 
&=
\sum_{m \in {\mathcal M}} (\lambda(m) - \mu_2(m)) {\bm P}(m,m) \det({\bm I} - {\bm P})^{m,m}\\
&=c\sum_{m \in {\mathcal M}} (\lambda(m) - \mu_2(m)) {\bm P}(m,m) \bm\pi(m) < 0
\end{align*}
for some constant $c > 0.$
The rest of the proof proceeds as in the proof of 
Proposition \ref{p:point1}.
\end{proof}

Recall that $(\rho_2,1) \in {\mathcal L}_1$, i.e., 
$1$ is the largest eigenvalue of ${\bm A}(\rho_2,1)$; the irreducibility of ${\bm A}$
implies that the eigenvectors corresponding to $1$ 
have strictly negative or positive
components; 
let ${\bm d}_2$ denote a right eigenvector of ${\bm A}(\rho_2,1)$ corresponding
to the eigenvalue $1$ with strictly positive components. Proposition \ref{p:interiorharm} and the previous proposition imply that
$[(\rho_2,1,{\bm d}_2),\cdot]$ is $(Y,M)$-harmonic on ${\mathbb Z} \times {\mathbb Z}_+ -\partial_2.$
All of the prior works (\cite{sezer2015exit, sezer2018approximation, unlu2018excessive}),
use a conjugate point of $(\rho_2,1)$ to construct
a $Y$-harmonic function. 
In the present case, in general, $(\rho_2,1)$ will have
$2|{\mathcal M}|-1$ conjugate points. Figure \ref{f:levelcurves2} suggests
that
only one of these conjugate points lies on ${\mathcal L}_1$; we will use $(\rho_2,1)$ along with this conjugate to define a $(Y,M)$-{\em super}harmonic function. 
This will be in
two steps. Proposition \ref{p:point2conj} identifies the 
relevant conjugate point;
Proposition \ref{p:hrho2} constructs the superharmonic function. 
We will use the superharmonic function in Sections \ref{s:ub1} and \ref{s:ub2} below in our 
analysis of the relative error \eqref{e:approximationerror}.
\nomenclature{${\bm d}_2$}{A right eigenvector of ${\bm A}(\rho_2, 1)$ corresponding to the eigenvalue 1 with all positive components}

The identification of the conjugate point 
requires the following assumption:
\begin{equation}\label{as:conj}
\sum_{m \in {\mathcal M}} (\rho_2 \mu_2(m) -\mu_1(m)) {\bm P}(m,m) \det({\bm I} - {\bm A}(\rho_2,1))^{m,m} < 0.
\end{equation}

Remark \ref{r:asconj} comments on this assumption and Proposition
\ref{p:whenasconj} gives simple conditions under which \eqref{as:conj} holds.

\begin{proposition}\label{p:point2conj}
Let $(\rho_2,1)$, $\rho_2 \in (0,1)$,
be the point on ${\mathcal L}_1$ identified
in Proposition \ref{p:point2}. Then there exists a unique
point $(\rho_2, \alpha^*_1) \in {\mathcal L}_1$, $\alpha^*_1 \in (0,1)$ if
\eqref{as:conj} holds.
\end{proposition}
\nomenclature{$\alpha^*_1$}{ $\alpha^*_1 \in (0,1)$,$\alpha$ component of  a point of the form $(\rho_2, \alpha^*_1)$, which is the conjugate point of $(\rho_2, 1)$}

\begin{proof}
Set $r_2 = \log(\rho_2).$
Proof is parallel to those of Propositions \ref{p:point1} and
\ref{p:point2} and is based on the analysis of the function
$H$ of \eqref{d:H} at the point $(r_2,0)$ via the implicit function theorem.
Define $f(\Lambda,r) = \det(\Lambda {\bm I} - {\bm A}(\rho_2,e^r))$
and observe
\[
\left.\frac{\partial f}{\partial r}\right|_{(r_2,0)}
=
\sum_{m \in {\mathcal M}} (\rho_2\mu_2(m) - \mu_1(m)) {\bm P}(m,m) \det({\bm I} - {\bm A}(\rho_2,1))^{m,m},
\]
which, by assumption \eqref{as:conj}, is strictly less than $0$. The
rest of the proof goes as that of Proposition \ref{p:point1}.
\end{proof}
\nomenclature{$r_2$}{$\log(\rho_2)$}
\begin{remark}\label{r:asconj}
Assumption (\ref{as:conj}) ensures that $(\rho_2,1)$ has a conjugate
point on the principal characteristic surface ${\mathcal L}_1$  with
$\alpha$ component less than $1$. There is no corresponding assumption
in the non-modulated tandem case, because, in that setup, the conjugate
of $(\rho_2,1)$ is $(\rho_2,\rho_1)$ whose $\alpha$ component $\rho_1$
is always less than $1$ by the stability assumption. 
In the simple constrained random walk case (treated in \cite{unlu2018excessive})
the corresponding assumption is $r^2 < \rho_1 \rho_2$ (see \cite[Display
(14)]{unlu2018excessive}).
	
The condition $\alpha^*_1 < 1$ is needed for the superharmonic function
constructed in Proposition \ref{p:hrho2} 
to be bounded on $\partial B$, see Proposition \ref{p:upperfortau}.
\end{remark}

\begin{proposition}\label{p:whenasconj}
Each of the following conditions is sufficient for \eqref{as:conj} to hold:	
\begin{enumerate}
	\item  $\lambda(m)/\mu_2(m) < 1$,
	$\lambda(m) < \mu_1(m)$ for all $m \in {\mathcal M}$
	and the ratio $\lambda(m)/\mu_2(m)$ does not depend on $m$,
	\item $\mu_2(m) < \mu_1(m)$ for all $m \in {\mathcal M}$.
\end{enumerate}
\end{proposition}
\begin{proof}
If $\lambda(m)/\mu_2(m) < 1$ does not depend on $m$ we can denote
the common ratio by $\rho_2' < 1$.
Substituting $(\beta,\alpha) = (\rho_2',1)$ 
we see that ${\bm A}(\rho_2',1) = {\bm P}.$ This implies that the root $\rho_2$
identified in Proposition \ref{p:point2} must equal $\rho_2'.$
Setting $\rho_2 = \rho_2'$ on the left side of \eqref{as:conj}
gives
\begin{align*}
&\sum_{m \in {\mathcal M}} (\rho_2\mu_2(m) -\mu_1(m)) {\bm P}(m,m) \det({\bm I} - {\bm A}(\rho_2,1))^{m,m} \\
&~~~=
\sum_{m \in {\mathcal M}} (\rho_2'\mu_2(m) -\mu_1(m)) {\bm P}(m,m) \det({\bm I} - {\bm A}(\rho_2,1))^{m,m} \\
&~~~
=\sum_{m \in {\mathcal M}} (\lambda(m) -\mu_1(m)) {\bm P}(m,m) \det({\bm I} - {\bm A}(\rho_2,1))^{m,m} 
\intertext{
	$\det({\bm I} - {\bm A}(\rho_2,1))^{m,m} > 0$
	by Lemma \ref{l:positivity}, and $\lambda(m) < \mu_1(m)$ by assumption;
	these and the last line imply \eqref{as:conj}:}
&~~~~~~~~ < 0.
\end{align*}	
That the condition $\mu_2(m) < \mu_1(m)$ for all
$m \in {\mathcal M}$ implies \eqref{as:conj}
follows from a similar argument.
\end{proof}

\begin{remark}
	The argument used in the proof above can be used to prove
	that the conjugate point $(\rho_2,\alpha^*_1)$ satisfies $\alpha^*_1 > 1$
	if one replaces $<$ with $>$ in \eqref{as:conj}.
\end{remark}

For the rest of our analysis we will need a further assumption:
\begin{equation}\label{as:mu1neqmu2}
\rho_1 \neq \rho_2,
\end{equation}
where $\rho_1$ is the first (or the second) component
of the point on ${\mathcal L}_1 \cap {\mathcal L}_{2,1}$ identified
in Proposition \ref{p:point1}
and
$\rho_2$ is the $\beta$ component of the point on 
${\mathcal L}_1$ identified in Proposition \ref{p:point2}.
Assumption \eqref{as:mu1neqmu2} generalizes the assumption
$\mu_1 \neq \mu_2$ in \cite{sezer2015exit,sezer2018approximation,unlu2018excessive}.The following lemma identifies sufficient conditions
for \eqref{as:mu1neqmu2} to hold.

\begin{lemma}
If $\mu_1(m) > \mu_2(m)$ for all $m \in {\mathcal M}$, or $\mu_1(m) < \mu_2(m)$ for all $m \in {\mathcal M}$,
then
\eqref{as:mu1neqmu2} holds.
\end{lemma}
\begin{proof}
The matrix 
${\bm D} = {\bm A}(\rho_2,\rho_2) - {\bm A}(\rho_2,1)$ is a diagonal matrix
whose $m^{th}$ entry equals $(1-\rho_2)(\mu_2(m) - \mu_1(m)).$
Suppose $\mu_2(m) > \mu_1(m)$ for all $m \in {\mathcal M}$;
then $\rho_2 \in (0,1)$ implies that ${\bm D}$ has strictly positive entries.

We have then:
\begin{align}\label{e:Ar2r2ineq}
{\bm A}(\rho_2,\rho_2) {\bm d}_2 &= {\bm A}(\rho_2,1) {\bm d}_2 + {\bm D} {\bm d}_2 \notag \\
			   &= {\bm d}_2 + {\bm D} {\bm d}_2 \notag \\
			   &> (1+\epsilon) {\bm d}_2
\end{align}
for some $\epsilon >0$;
here we have used 1) ${\bm d}_2$ is an eigenvector
of ${\bm A}(\rho_2,1)$ corresponding to the eigenvalue $1$ and 2) ${\bm D}$
has strictly positive entries. 
We know by  \cite[Proof of Theorem 1, Chapter 16]{lax1996linear} that
\begin{equation}\label{e:charl1}
\Lambda_1({\bm A}(\rho_2,\rho_2)) = 
\sup\{ c: \exists x \in {\mathbb R}^{|{\mathcal M}|}_+,
{\bm A}(\rho_2,\rho_2)x  \ge cx \}.
\end{equation}
This and \eqref{e:Ar2r2ineq} imply that
the largest eigenvalue of ${\bm A}(\rho_2,\rho_2)$ is strictly greater than $1$.
This implies $\rho_2 < \rho_1.$
That $\mu_1(m) > \mu_2(m)$ for all $m \in {\mathcal M}$ implies $\rho_2>\rho_1$ follows from the same argument applied to ${\bm A}(\rho_2,1) {\bm d}_{2,1}.$
\end{proof}
\nomenclature{${\bm d }_{2,1}$}{A right eigenvector of ${\bm A}(\rho_2, \alpha_1^*)$ with strictly positive entries}
\begin{figure}[h]
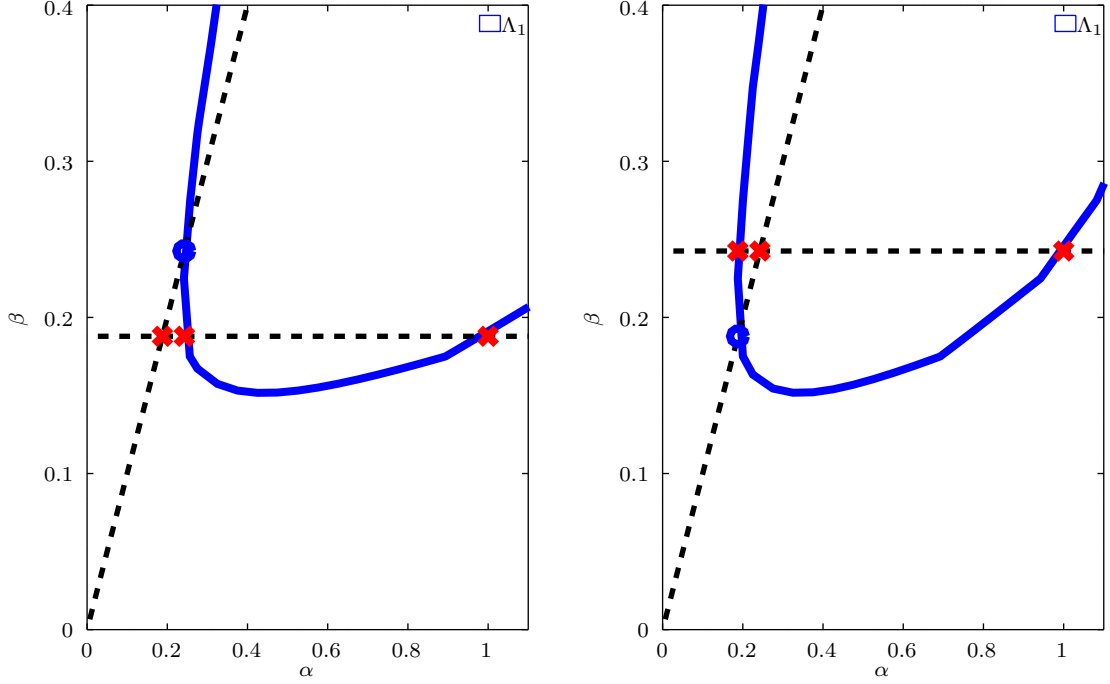

\hspace{-0.5cm}
\begin{subfigure}{0.48\textwidth}
	\input{rho1grho2}
\end{subfigure}
\begin{subfigure}{0.48\textwidth}
	\input{rho2grho1}
\end{subfigure}
\caption{$\rho_1-\rho_2$ and $\alpha^*_1-\rho_2$ have the same sign (Lemma \ref{l:orderofalphastar}); 
	the points marked with 'x' are $(\rho_2,\rho_2)$, $(\rho_2,\alpha^*_1)$ and $(1,\rho_2)$;
	the point marked with 'o' is $(\rho_1,\rho_1)$ }\label{f:roots}
\end{figure}
\FloatBarrier
\begin{lemma}\label{l:orderofalphastar}
Let $(\rho_2,\alpha^*_1)$ be the conjugate point of
$(\rho_2,1)$ on ${\mathcal L}_1$ identified
in Proposition \ref{p:point2conj}. 
Then
$\rho_1 > \rho_2$ implies $\alpha^*_1 > \rho_2$ and
$\rho_1 < \rho_2$ implies $\alpha^*_1 < \rho_2.$
\end{lemma}
Figure \ref{f:roots} illustrates this lemma.
\begin{proof}
By definition $\rho_1$ is the unique positive number strictly less than $1$
satisfying $\Lambda_1(\rho_1,\rho_1) = 1$;
$\rho_2 < \rho_1$ implies $\Lambda_1(\rho_2,\rho_2)  > 1.$ 
But $\alpha^*_1$ satisfies 
$\Lambda_1(\rho_2,\alpha^*_1) =1$ and $\Lambda_1(\rho_2,\rho) \le 1$
for  $\rho \in (\alpha^*_1,\rho_2]$. It follows that $\rho_2 < \alpha^*_1.$
The argument for the opposite implication is similar.
\end{proof}

\begin{remark}
By the previous lemma the assumption \eqref{as:mu1neqmu2} is
equivalent to 
\begin{equation}\label{as:mu1neqmu2alt}
\alpha^*_1 \neq \rho_2.
\end{equation}
\end{remark}

\begin{remark}
$\rho_1$ is the unique solution of $\Lambda_1(\beta,\beta)=1$
on $(0,1)$; similarly $\rho_2$ is the unique solution of
$\Lambda_1(\beta,1) =1$ on $(0,1)$. That $\Lambda_1$ is the largest
eigenvalue of ${\bm A}(\beta,\alpha)$ and the above facts imply that
$\rho_1$ [$\rho_2$] 
is the largest root of 
${\bm p}(\beta,\beta)$ [${\bm p}(\beta,1)$] 
on $(0,1)$. Therefore, one can state the assumption
\eqref{as:mu1neqmu2} also as follows: ``the largest roots of
${\bm p}(\beta,\beta)$ and ${\bm p}(\beta,1)$ on $(0,1)$ differ.''
\end{remark}

By definition, $1$ is the largest eigenvalue of
${\bm A}(\rho_2,\alpha^*_1)$; let
${\bm d}_{2,1}$ denote a right eigenvector of this matrix
with strictly positive entries.
Next proposition  constructs a
$(Y,M)$-superharmonic function that we will use to find upper bounds
on approximation errors;
this is one of the key steps of our argument.
\begin{proposition}\label{p:hrho2c0}
\label{p:hrho2}
Under assumption \eqref{as:mu1neqmu2} one can choose a constant
$c_0 \in {\mathbb R}$ 
\\(
$c_0 > 0$ for $\alpha^*_1 < \rho_2$ and
$c_0 < 0$ for $\alpha^*_1 > \rho_2$)
so that	
\begin{equation}\label{d:hrho2}
h_{\rho_2} \doteq [(\rho_2,1,{\bm d}_2),\cdot] + c_0[(\rho_2,\alpha^*_1, {\bm d}_{2,1},\cdot],
\end{equation}	
is a $(Y,M)$-superharmonic function.
\end{proposition}
\begin{proof}
By their construction, the conjugate points $(\rho_2,1)$ and
$(\rho_2,\alpha^*_1)$ lie on ${\mathcal L}_1$.
This and Proposition \ref{p:interiorharm} imply that
the functions
$[(\rho_2,1,{\bm d}_2),\cdot]$ and  $[(\rho_2,\alpha^*_1, {\bm d}_{2,1}), \cdot]$
are $(Y,M)$-harmonic on ${\mathbb Z} \times{ \mathbb Z}_+ - \partial_2.$ This implies the same for their linear combination $h_{\rho_2}.$
Therefore, to prove that $h_{\rho_2}$ is $(Y,M)$-superharmonic, it suffices
to check this on $\partial_2.$

By definition $h_{\rho_2}$ is superharmonic on $\partial_2$ if
\[
{\mathbb E}_{(y,m)}[ h_{\rho_2}(Y_1,M_1)] \le h_{\rho_2}(y,m)
\]	
for $y = (k,0)$ and $m \in {\mathcal M}.$
\nomenclature{$h_{\rho_2}$}{$(Y, M)$-superharmonic function constructed
from points $(\rho_2,1)$ and $(\rho_2,\alpha_1^*)$}
By Lemma \ref{l:roleofC},
\begin{align*}
{\mathbb E}_{(y,m)}[(\rho_2,1,{\bm d}_2),(Y_1,M_1)] - 
[(\rho_2,1,{\bm d}_2),(y,m)] &= \rho_2^k {\bm c}(\rho_2,1,{\bm d}_2)(m), \\
{\mathbb E}_{(y,m)}[(\rho_2,\alpha^*_1,{\bm d}_{2,1}),(Y_1,M_1)] - 
[(\rho_2,\alpha^*_1,{\bm d}_{2,1}),(y,m)] &= \rho_2^k {\bm c}(\rho_2,\alpha^*_1,{\bm d}_{2,1})(m),
\end{align*}	
where ${\bm c}(\cdot,\cdot,\cdot)$ is defined as in \eqref{d:C}.
The last two lines give	
\begin{equation}\label{e:diffforhrho2}
{\mathbb E}_{(y,m)}[ h_{\rho_2}(Y_1,M_1)] - h_{\rho_2}(y,m)
=
\rho_2^k \left( {\bm c}(\rho_2,1,{\bm d}_2)(m) +c_0 {\bm c}(\rho_2,\alpha^*_1,{\bm d}_{2,1})(m)\right).
\end{equation}	
For $h_{\rho_2}$ to be superharmonic, the right side of the last
display must be negative. The sign of this expression is determined by
\begin{equation}\label{e:signdiff}
{\bm c}(\rho_2,1,{\bm d}_2)(m) +c_0 {\bm c}(\rho_2,\alpha^*_1,{\bm d}_{2,1})(m).
\end{equation}	
The definition \eqref{d:C} of ${\bm c}$ and $\rho_2 < 1$ and ${\bm d}_2(m) > 0$ for
all $m \in {\mathcal M}$ imply that the first term is strictly
positive for all $m \in {\mathcal M}.$ Define
\[
d_{\max} \doteq \max_{m \in {\mathcal M}} {\bm c}(\rho_2,1, {\bm d}_2)(m) > 0.
\]	

The sign of the second term in \eqref{e:signdiff} depends
on whether $\alpha^*_1 < \rho_2$ or $\alpha^*_1 > \rho_2.$
For $\alpha^*_1 < \rho_2$, the definition
\eqref{d:C} of ${\bm c}$ and ${\bm d}_{2,1}(m) > 0$ for all $m \in {\mathcal M}$
imply that the ${\bm c}$ term in \eqref{e:signdiff} is strictly negative for
all $m$.
Define	
\begin{equation}\label{d:dmaxstar}
d_{\max}^* \doteq \max_{m \in {\mathcal M}} {\bm c}(\rho_2,\alpha^*_1,{\bm d}_{2,1})(m) < 0.
\end{equation}

If we choose $c_0 > 0$ so that
\begin{equation}\label{e:condforc0}
d_{\max} + c_0 d_{\max}^* < 0,
\end{equation}	
\eqref{e:signdiff} will be strictly less than $0$ for all $m$.
This and \eqref{e:diffforhrho2} imply
that $h_{\rho_2}$ is superharmonic 
for any $c_0$ satisfying \eqref{e:condforc0}.

For $\alpha^*_1 > \rho_2$ the argument remains the same except that
we replace the $\max$ in \eqref{d:dmaxstar} with $\min$ and
$c_0 < 0$.
\end{proof}
In the next section we will use $h_{\rho_2}$ to find bounds on
the approximation error \eqref{e:approximationerror}.

\section{Upper bound for ${\mathbb P}_{(y, m)}(\tau <\infty)$}\label{s:ub1}

As we saw in Proposition
\ref{p:hrho2} above, $(Y,M)$-superharmonic functions can be constructed
from just two conjugate points on 
${\mathcal L}_{1} \subset {\mathcal H}^{\beta\alpha}.$

We will need an upper bound on ${\mathbb P}_{(y,m)}(\tau < \infty)$ in our analysis
of the relative error \eqref{e:approximationerror}; in the non-modulated tandem
walk treated in \cite{sezer2015exit, sezer2018approximation}, this probability can be represented
exactly using the harmonic functions constructed from points on the
characteristic surface, which also obviously serves as an upper bound.
In the present case, we will construct an upper bound for
${\mathbb P}_{(y,m)}(\tau < \infty)$ from 
$(Y,M)$-harmonic and superharmonic functions constructed in 
Propositions \ref{p:harconjugate} and \ref{p:hrho2}. The next proposition
constructs the necessary function the one following it derives
the upper bound.

\begin{proposition}\label{p:boundarylowerbound}
Let $h_{\rho_1} = [(\rho_1,\rho_1,{\bm d}_1),\cdot]$ be as in \eqref{d:hrho1} and
$h_{\rho_2}$ be as in \eqref{d:hrho2}.
One can choose $c_1 \ge 0$ so that	
\begin{equation}\label{e:boundarylowerbound}
c_2 \doteq \min_{y \in \partial B,m \in {\mathcal M}} 
h_{\rho_2}(y,m) + c_1 h_{\rho_1}(y,m) > 0;
\end{equation}	
for $\alpha^*_1 < \rho_2$ one can choose $c_1 = 0.$
\end{proposition}
\begin{proof}
By its definition, 	
\begin{equation}\label{e:hrho2onpartialB}
h_{\rho_2}(y,m) = {\bm d}_2(m) + c_0(\alpha^*_1)^{y(2)} {\bm d}_{2,1}(m),
\end{equation}	
for $y \in \partial B.$ We know by Proposition \ref{p:hrho2} that
$c_0>0$ for $\alpha^*_1 <\rho_2$. This, $\alpha^*_1 > 0$, ${\bm d}_{2,1}(m) > 0$
imply	
\[
\min_{y \in \partial B }h_{\rho_2}(y,m) \ge \min_{m \in {\mathcal M}} 
{\bm d}_2(m) > 0,
\]
which implies \eqref{e:boundarylowerbound} with $c_1 = 0.$
	
For $\alpha^*_1 > \rho_2$, $c_0 < 0$ and \eqref{e:hrho2onpartialB} 
can take negative values for small $y(2)$.
But $0 < \alpha^*_1 < 1$ implies that
there exists $k_0 > 0$ such that	
\begin{equation}\label{e:tailok}
h_{\rho_2}(y,m) \ge 
\min_{m \in {\mathcal M}}{\bm d}_2(m) /2
 > 0,~~~ y \in \partial B,~
y(2) \ge k_0.
\end{equation}	
On the other hand, ${\bm d}_1(m) > 0$ for all $m \in {\mathcal M}$
and $\rho_1 > 0$ imply that $h_{\rho_1}(y,m) > 0$ for all $y \in \partial B$,
$m \in {\mathcal M}.$ Then one can choose $c_1 > 0$ so that	
\begin{equation}\label{e:liftup}
c_1 {\bm d}_1(m) \rho_1^{y(2)}
+{\bm d}_2(m) + c_0(\alpha^*_1)^{y(2)} {\bm d}_{2,1}(m) > 
\min_{m \in {\mathcal M}}{\bm d}_2(m) /2,~~y \in \partial B,~ y(2) \le k_0,
\end{equation}
since this inequality concerns only finitely many $y \in \partial B$.
$c_1$ chosen thus, \eqref{e:tailok}
and \eqref{e:liftup} imply \eqref{e:boundarylowerbound}.
\end{proof}

\begin{proposition}\label{p:upperfortau}
Let $c_1 \ge 0$, $c_2 > 0$ be as in Proposition \ref{p:boundarylowerbound}
\begin{equation}\label{e:boundPytau}
{\mathbb P}_{(y,m)}(\tau < \infty) \le 
\frac{1}{c_2} \left( h_{\rho_2}(y,m) + c_1 h_{\rho_1}(y,m) \right).
\end{equation}	
\end{proposition}
\begin{proof}
For ease of notation set
\[
f= h_{\rho_2} + c_1 h_{\rho_1};
\]
$\rho_1,\rho_2, \alpha^*_1 \in (0,1)$ implies 
\[
\sup_{y \in B, m \in {\mathcal M}} | f(y,m)| < \infty.
\]	
Furthermore, by Propositions \ref{p:point1} and \ref{p:hrho2} $f$
is $(Y,M)$-superharmonic. These imply that 
$k\mapsto f(Y_{k\wedge \tau},M_{k\wedge \tau})$
is a bounded supermartingale. Then by the optional sampling theorem
(\cite[Theorem 5.7.6]{durrett2010probability}) 
\[
{\mathbb E}_{(y,m)}[ f(Y_\tau,M_\tau) 1_{\{\tau < \infty\}} ] \le f(y,m);
\]
this, $Y_\tau \in \partial B$ when $\tau < \infty$
and \eqref{e:boundarylowerbound} imply
\[
c_2 {\mathbb P}_{(y,m)}(\tau < \infty) \le f(y,m),
\]
which gives \eqref{e:boundPytau}.
\end{proof}
\section{Upper bound for ${\mathbb P}_{(x, m)}( \sigma_1 <\sigma_{1,2}<  \tau_n < \tau_0)$}\label{s:ub2}
Define
\begin{equation}\label{d:defsigma12}
\sigma_{i}\doteq \inf\{k \ge 0:X_k \in \partial_i\}; i=1,2,
\end{equation}
and
\begin{equation}\label{d:defsigma12}
\sigma_{1,2}\doteq \inf\{k \ge 0:X_k \in \partial_2,~k\ge\sigma_1\};
\end{equation}
$\sigma_i$ is the first time $X$ hits $\partial_i$ and
$\sigma_{1,2}$ is the first time $X$ hits $\partial_2$ after hitting $\partial_1.$
\nomenclature{$\sigma_1$}{first time when $X$ hits $\partial_1$}
\nomenclature{$\sigma_{1,2}$}{first time $X$ hits  $\partial_2$ after $\partial_1$}
In the next section we find an upper bound on the probability
${\mathbb P}_{(x, m)}( \sigma_1 <\sigma_{1,2}<  \tau_n < \tau_0)$,
we will use this bound in the analysis of the approximation error
in the proof of Theorem \ref{t:mainapprox}.
Define
\begin{equation}\label{d:defrho}
\rho \doteq \rho_1 \vee \rho_2.
\end{equation}
\nomenclature{$\rho$}{$\rho_1 \vee \rho_2$}
The goal of the section is to prove

\begin{proposition}\label{p:upperboundforXrare}
For any $\epsilon > 0$ there exists $n_0 > 0 $ such that	
\begin{equation}\label{e:upperboundforXrare}
{\mathbb P}_{(x,m)}( \sigma_1 < \sigma_{1,2} < \tau_n < \tau_0)
\le \rho^{n(1-\epsilon)}
\end{equation}	
for $n \geq n_0$ and $(x,m) \in A_n$.
\end{proposition}
We split the proof into cases $\rho_1 > \rho_2$ and $\rho_2 > \rho_1.$
The first subsection below
treats the first case $\rho_1 > \rho_2$, the next gives the
changes needed for the latter.

Let ${\bm A}_1$ denote the characteristic matrix for $\partial_1$:
\begin{align*}
{\bm p}_1(\beta,\alpha,m) &\doteq \lambda(m) \frac{1}{\beta} +
\mu_1(m) + \mu_2(m)\frac{\beta}{\alpha},~~m \in {\mathcal M}, \notag \\
{\bm A}_1(\beta,\alpha)_{m_1,m_2} &\doteq
\begin{cases} {\bm P}(m_1,m_2), &~~~m_1 \neq m_2 \\
{\bm P}(m_1,m_1)  {\bm p}_1(\beta,\alpha,m), &~~m_1 = m_2,
\end{cases},
(m_1,m_2) \in {\mathcal M}^2.
\end{align*}
We will use the following fact several times in our analysis.
\begin{lemma}\label{l:rho2harXM}
The function	
\begin{equation}\label{e:rho2simplefun}
(x,m) \mapsto [(\rho_2,1, {\bm d}_2), (T_n(x),m)] = \rho_2^{n-(x(1)+x(2))} {\bm d}_2(m)
\end{equation}	
is $(X,M)$-harmonic on ${\mathbb Z}_+^2 - \partial_2.$
\end{lemma}
\begin{proof}
We know by Proposition \ref{p:interiorharm} and $(\rho_2,1, {\bm d}_2) 
\in {\mathcal H}$ that
$[(\rho_2,1, {\bm d}_2), \cdot]$ is $(Y,M)$-harmonic on 
${\mathbb Z} \times {\mathbb Z}_+^o$, which
implies that \eqref{e:rho2simplefun} is $(X,M)$-harmonic on
${\mathbb Z}_+^{2,o}$;
this and
${\bm A}_1(\rho_2,1) = {\bm A}(\rho_2,1)$ imply the $(X,M)$-harmonicity of
\eqref{e:rho2simplefun} on $\partial_1$.
\end{proof} 

\subsection{$\rho_1 > \rho_2$}

To prove \eqref{e:upperboundforXrare} we will construct a
corresponding supermartingale; applying the optional sampling
theorem to the supermartingale will give our desired bound. The event $\{\sigma_1 < \sigma_{1,2} < \tau_n < \tau_0\}$
consists of three stages: $X$ first hits $\partial_1$
then $\partial_2$ and finally $\partial A_n$ without ever hitting
$0$.
If $h$ is an $(X,M)$-superharmonic function, it follows
from the definitions that $h(X,M)$ is a supermartingale. We will
construct our supermartingale by applying three
functions
(one for each of the above stages)
to $(X,M)$: the function for the first stage is the constant
$\rho_1^n$, which is trivially superharmonic.
The function for the second stage will be a constant multiple of
$(x,m)\mapsto h_{\rho_1}(T_n(x),m)$.
By Proposition \ref{p:point1}, $(x,m)\mapsto h_{\rho_1}(T_n(x),m)$ 
is $(X,M)$-harmonic
on ${\mathbb Z}_+^2 - \partial_1.$ One can check directly that
it is in fact subharmonic on $\partial_1.$ The definition of
the supermartingale $S$ will involve terms to compensate for this.
The function for the third stage is
\begin{align}\label{e:funforstage3}
&h_3: (x,m) \mapsto 
h_{\rho_2}(T_n(x),m) + c_1 h_{\rho_1}(T_n(x),m), ~x \in A_n,~ m \in {\mathcal M}, \notag \\
&= h_{\rho_2}((n-x(1),x(2)),m) + c_1 h_{\rho_1}((n-x(1),x(2)),m), \notag\\
&= \rho_2^{n-(x(1)+x(2))} \left( {\bm d}_2(m) + c_0 {\alpha^*_1}^{x(2)} {\bm d}_{2,1}(m) \right)
+ c_1 \rho_1^{n-x(1)} {\bm d}_1(m),
\end{align}
where $c_1 \ge 0$ is chosen as in Proposition \ref{p:boundarylowerbound}
and $c_0$ is as in Proposition \ref{p:hrho2}. 
The next two propositions imply that $h_3$ is $(X,M)$-superharmonic
on ${\mathbb Z}_+^2 -\partial_1.$

\begin{proposition}\label{p:superharmeverywhere}
For $\rho_1 > \rho_2$,~ $h_{\rho_2}(T_n(\cdot),\cdot)$ 
is superharmonic on all of ${\mathbb Z}_+^2.$ 
\end{proposition}
\begin{proof}
That $h_{\rho_2}(T_n(\cdot),\cdot)$ is $(X,M)$-superharmonic
on ${\mathbb Z}_+^2-\partial_1$ follows from 
Proposition \ref{p:hrho2} (i.e., from the fact that
$h_{\rho_2}(\cdot,\cdot)$ is $(Y,M)$-harmonic). Therefore, 
it suffices to prove that 
$h_{\rho_2}(T_n(\cdot),\cdot)$ is superharmonic on $\partial_1.$
$h_{\rho_2}(T_n(\cdot),\cdot)$  is a
sum of two functions:	
\begin{equation}\label{e:sumoftwoterms}
h_{\rho_2}(T_n(\cdot),\cdot) = 
[(\rho_2,1,{\bm d}_2),(T_n(\cdot),\cdot)] + c_0
[ (\rho_2,\alpha^*_1, {\bm d}_{2,1}), (T_n(\cdot),\cdot)].
\end{equation}	
Let us show that each of these summands is $(X,M)$- superharmonic on 
$\partial_1$. The first
summand is $(X,M)$-harmonic (and therefore, superharmonic)
on $\partial_1$ by Lemma \ref{l:rho2harXM}.
To treat the second term in \eqref{e:sumoftwoterms} recall
the following: $\rho_2 < \rho_1$ implies $\rho_2 < \alpha^*_1$ (Lemma \ref{l:orderofalphastar});
then, by Proposition \ref{p:hrho2}, $c_0 < 0$. 
Therefore, if we can show that
$[ (\rho_2,\alpha^*_1, {\bm d}_{2,1}), (T_n(\cdot),\cdot)]$ is
$(X,M)$-subharmonic on $\partial_1$ we will be done.
Let us now see that this
is indeed the case.

For ease of notation set	 
\[
 h(x,m) = [ (\rho_2,\alpha^*_1, {\bm d}_{2,1}), (T_n(x),m)] =
\rho_2^{n-(x(1) + x(2))} {\alpha^*_1}^{x(2)} {\bm d}_{2,1}(m).
\]	
A calculation parallel to the proof of Proposition 
\ref{p:interiorharm} shows	
\begin{equation}\label{e:subharmonicconjugate}
{\mathbb E}_{(x,m)}\left[ h(X_1,M_1) \right]
- h(x,m)
= {\bm d}_{2,1}(m)\mu_1(m) (1 - \alpha^*_1) \rho_2^{n-x(2)} > 0,
\end{equation}	
for $x \in \partial_1$, i.e., $h$ is $(X,M)$-subharmonic on
$\partial_1$. This completes the proof of this proposition.
\end{proof}

\begin{proposition}\label{p:hrho1harmwhere}
$h_{\rho_1}(T_n(\cdot),\cdot)$ is harmonic (and therefore superharmonic)
on ${\mathbb Z}_+^2 -\partial_1.$ 
It is subharmonic on $\partial_1$ where it satisfies	
\begin{equation}\label{e:errortermforrho1}
{\mathbb E}_{(x,m)} [h_{\rho_1}(T_n(X_1),M_1)] - h_{\rho_1}(T_n(x),m)
= {\bm d}_1(m)\mu_1(m) (1 - \rho_1) \rho_1^{n} > 0.
\end{equation}	
\end{proposition}

The proof is parallel to the computation given in the proof of 
Proposition \ref{p:interiorharm} and is omitted.
We can now define the supermartingale that we will use to prove
\eqref{e:upperboundforXrare}:
\begin{align*}
S'_k &\doteq \begin{cases} h_1, &k \le \sigma_1, \\
h_2(X_k,M_k), & \sigma_1 < k \le \sigma_{1,2},\\ 
h_3(X_k,M_k), & k > \sigma_{1,2},
\end{cases}\\
S_k  &\doteq S'_k - c_5k\rho_1^n,
\end{align*}
where
\begin{align}
c_3 &\doteq 
\frac{\max_{m \in {\mathcal M}} {\bm d}_2(m)  + c_1\max_{m \in {\mathcal M}}{\bm d}_1(m)}{ \min_{m \in {\mathcal M}} {\bm d}_1(m)} > 0
 \label{d:c3},\\
h_1 &\doteq  c_4\rho_1^n,~~~~ c_4 \doteq c_3 \max_{m \in {\mathcal M}}{\bm d}_1(m) > 0\notag, \\
h_2 &\doteq c_3 h_{\rho_1}(T_n(\cdot), \cdot) =  
c_3 [(\rho_1,\rho_1, {\bm d}_1), (T_n(\cdot),\cdot)] = c_3 \rho_1^{n-x(1)}{\bm d}_1(\cdot) ,
\label{d:h2}\\
c_5 &\doteq c_3 (1 - \rho_1) \max_{m \in {\mathcal M}} {\bm d}_1(m)\mu_1(m)\label{d:c4}.
\end{align}
Two comments: 
$h_1$ is a constant function, independent of $x$ and $m$,
and  $h_1 \ge h_2$ on $\partial_1$.

\begin{proposition}\label{p:supermartingalc1}
$S$ is a supermartingale.
\end{proposition}
\begin{proof}
The claim follows mostly from the fact that the functions
involved in the definition of $S'$ are $(X,M)$-superharmonic away
from $\partial_1.$ The term that breaks superharmonicity
on $\partial_1$ is
$[(\rho_1,\rho_1,{\bm d}_1), (T_n(X_k),M_k)]$;
the
$-c_5 k \rho_1^n$ term in the definition of $S$ is introduced to compensate
for this. The details are as follows.
	
The $(X,M)$-harmonicity of $h_1$,
$h_2$ and $h_3$ implies	
\[
{\mathbb E}_{(x,m)}[S'_{k+1} | {\mathscr F}_k] = S'_k
\]	
for $X_k \in {\mathbb Z}_+^2- \partial_1 \cup \partial_2$; i.e., $S'_k$ satisfies the martingale
equality condition 
for $X_k \in {\mathbb Z}_+^2 -\partial_1 \cup \partial_2$; this implies that $S_k$ satisfies
the supermartingale inequality condition over the same event.

$h_2$ and $h_3$ are  $(X,M)$-superharmonic on $\partial_2$
by Propositions \ref{p:superharmeverywhere}
and \ref{p:hrho1harmwhere} ($h_1$ is trivially so because it is constant); this implies	
\[
{\mathbb E}_{(x,m)}[S'_{k+1} | {\mathscr F}_k] \le S'_k
\]	
for $X_k \in \partial_2$ and $k \neq \sigma_{1,2}.$
For $k = \sigma_{1,2}$ we have 
$S'_{k+1} = h_3(X_{k+1},M_{k+1})$.
This, the $(X,M)$-superharmonicity of $h_3$ on $\partial_2$ implies
\begin{align}
{\mathbb E}_{(x,m)}[S'_{k+1} | {\mathscr F}_k] &=
{\mathbb E}_{(x,m)}[h_3(X_{k+1},M_{k+1})| {\mathscr F}_k] \notag\\
&\le h_3(X_k,M_k) \label{e:firststep}
\end{align}	
for $k = \sigma_{1,2}$.
On the other hand, 	
\begin{equation}\label{e:sprimek}
S'_k = h_2(X_k,M_k)\text{ for } k= \sigma_{1,2}.
\end{equation}
	The definitions of $c_3$, $h_2$ and $h_3$ in \eqref{d:c3},
\eqref{d:h2} and \eqref{e:funforstage3}, $\rho_2 < \rho_1$
and $c_0 < 0$ imply
\[
h_3(x,m) \le h_2(x,m)
\]
for $x\in \partial_2.$
This and \eqref{e:sprimek} imply	
\[
h_3(X_k,M_k) \le h_2(X_k,M_k) = S'_k
\]	
for $k =\sigma_{1,2}.$ The last display and \eqref{e:firststep}
imply	
\[
{\mathbb E}_{(x,m)}[S'_{k+1} | {\mathscr F}_k] \le S'_k,
\]	
i.e., $S'$ and $S$ are $(X,M)$-supermartingales for $k=\sigma_{1,2}$ as well.

It remains to prove 	
\begin{equation}\label{e:toprovepartial1}
{\mathbb E}_{(x,m)}[S_{k+1} | {\mathscr F}_k ] \le S_k,~ \text{when }X_k \in \partial_1.
\end{equation}
The cases to be treated here are:
$k=\sigma_1$, $\sigma_1 < k <  \sigma_{1,2}$ and $k > \sigma_{1,2}.$

For $k=\sigma_1$, we have $S'_k = h_1(X_k,M_k) = c_3 \rho_1^n{\bm d}_1(M_k)$
and $S'_{k+1} = h_2(X_{k+1},M_{k+1})$; these and $h_1 \ge  h_2$ on $\partial_1$ imply	
\begin{align}\label{e:argumentrho1}
&{\mathbb E}_{(x,m)} [S_{k+1}| {\mathscr F}_k] -S_k\\
&~~={\mathbb E}_{(x,m)} [c_3 h_{\rho_1}(T_n(X_{k+1}),M_{k+1})| {\mathscr F}_k] 
-
c_3 \rho_1^n {\bm d}_1(M_k)
-c_5 \rho_1^{n},\notag
\intertext{By (\ref{e:errortermforrho1}) and $\sigma_1 = k$, this equals}
&~~\le c_3 {\bm d}_1(M_k) \mu_1(M_k) (1-\rho_1)\rho_1^{n} -c_5 \rho_1^{n}.
\notag
\intertext{By the definition of $c_5$:}
&~~= \rho_1^n c_3(1-\rho_1) ( {\bm d}_1(M_k)\mu_1(M_k) - 
\max_{m\in {\mathcal M}} {\bm d}_1(m)\mu_1(m) )\le 0,\notag
\end{align}	
which proves \eqref{e:toprovepartial1} for $k=\sigma_1.$

For $\sigma_1 < k < \sigma_{1,2}$,~ $S_k' = h_2(X_k,M_k) = 
c_3 h_{\rho_1}( T_n(X_k),M_k)$; 
therefore the above argument applies
to this case as well (except for the last step which is not needed
here because $S_k'$ and $S'_{k+1}$ are defined by
applying the same function $h_2$ to $(X_{k+1},M_{k+1})$  and $(X_k,M_k)$).

Finally, to treat the case
$X_k \in \partial_1$ and $k > \sigma_{1,2}$ we start with
\begin{align*}
{\mathbb E}_{(x,m)}[ S_{k+1} | {\mathscr F}_k] - S_k &=
{\mathbb E}_{(x,m)}[ S'_{k+1} | {\mathscr F}_k] - S'_k -c_5 \rho_1^n, \\
\intertext{ $S_k' = h_3(X_k,M_k)$ for $k > \sigma_{1,2}.$
	Then by the definition of $h_3$:}
&=
{\mathbb E}_{(x,m)} 
[h_{\rho_2}(T_n(X_{k+1}),M_{k+1}) + 
c_1 h_{\rho_1}(T_n(X_{k+1}),M_{k+1}) | {\mathscr F}_k] \\
&~~~~~~~~~~~~- 
h_{\rho_2}(T_n(X_k),M_k) - c_1 h_{\rho_1}(T_n(X_k),M_k) 
-c_5 \rho_1^n, \\
&=
\left(
{\mathbb E}_{(x,m)} 
[h_{\rho_2}(T_n(X_{k+1}),M_{k+1})  |{\mathscr F}_k] -h_{\rho_2}(T_n(X_k),M_k) 
\right)\\
&+ 
{\mathbb E}[c_1 h_{\rho_1}(T_n(X_{k+1}),M_{k+1}) | {\mathscr F}_k] 
- c_1 h_{\rho_1}(T_n(X_k),M_k)  -c_5 \rho_1^n. 
\intertext{The $(X,M)$-superharmonicity of 
	$h_{\rho_2}(T_n(\cdot),\cdot)$  implies that
	the difference inside the parenthesis is negative, therefore:}
&\le
{\mathbb E}[c_1 h_{\rho_1}(T_n(X_{k+1}),M_{k+1}) | {\mathscr F}_k] 
- c_1 h_{\rho_1}(T_n(X_k),M_k)  -c_5 \rho_1^n. 
\intertext{
	Proposition \ref{p:hrho1harmwhere} (\eqref{e:errortermforrho1}) now gives}
&=
c_1{\bm d}_1(M_k)\mu_1(M_k) (1 - \rho_1) \rho_1^{n}  -c_5 \rho_1^n.
\intertext{By its definition \eqref{d:c4}, 
	$c_5 > c_1{\bm d}_1(m)\mu_1(m) (1 - \rho_1)$ for all $m \in {\mathcal M}$,
	which implies:}
&\le 0.
\end{align*}	
This proves \eqref{e:toprovepartial1} for $k >\sigma_{1,2}$
and completes the proof of this proposition.
\end{proof}

We are now ready to give a proof of Proposition \ref{p:upperboundforXrare}
for $\rho_1 > \rho_2$:

\begin{proof}[Proof of Proposition \ref{p:upperboundforXrare};
case $\rho_1 > \rho_2$]
By its definition (\ref{d:defrho}),
$\rho$ of \eqref{e:upperboundforXrare} equals $\rho_1$
for $\rho_1 > \rho_2$.
We begin by truncating time: \cite[Theorem A.2]{sezer2009importance}
implies that there exists $c_6 > 0$ and $N_0> 0$ such that
\[
{\mathbb P}_{(x,m)}( \tau_n \wedge \tau_0 > c_6 n) \le \rho_1^{2n},
\]
for $n > N_0.$ Then:
\begin{align}\label{e:truncatetime}
&{\mathbb P}_{(x,m)}( \sigma_1 < \sigma_{1,2} < \tau_n < \tau_0) \\
&~~~=\notag
{\mathbb P}_{(x,m)}( \sigma_1 < \sigma_{1,2} < \tau_n < \tau_0, \tau_n \wedge \tau_0 \le c_6 n)\\
&~~~+\notag {\mathbb P}_{(x,m)}( \sigma_1 < \sigma_{1,2} < \tau_n < \tau_0, \tau_n \wedge \tau_0 > c_6 n)\\
&~~~\le\notag
{\mathbb P}_{(x,m)}
( \sigma_1 < \sigma_{1,2} < \tau_n < \tau_0, \tau_n \wedge \tau_0 \le c_6 n) + \rho_1^{2n}
\end{align}
\nomenclature{$c_6$}{a constant used in proof of Proposition \ref{p:upperboundforXrare}}
for $n > N_0.$ Therefore, to prove \eqref{e:upperboundforXrare} it
suffices to bound the first term on the right side of the last inequality.
Now apply the optional sampling theorem to the supermartingale $S$
at the bounded stopping time $\tau = \tau_0 \wedge \tau_n \wedge c_6n$:	
\[
{\mathbb E}_{(x,m)}\left[ S_{\tau_0 \wedge \tau_n \wedge c_6 n}\right]
\le S_0 = c_4 \rho_1^n.
\]
By definition, $S_k = S_k' - c_5 k\rho_1^n$; substituting this
in the last display gives:
\begin{align*}
-c_5 c_6 n \rho_1^n + {\mathbb E}_{(x,m)}[ S_{\tau}'] &\le c_4 \rho_1^n\\
{\mathbb E}_{(x,m)}[ S_{\tau}'] &\le (c_4  + nc_5 c_6)\rho_1^n.
\intertext{By its definition, $S_k' > 0$, therefore restricting
	it to an event makes the last expectation smaller:}
{\mathbb E}_{(x,m)}[ S_\tau' 1_{\{\sigma_1 < \sigma_{1,2} < \tau_n < \tau_0 \le c_6n\}}] 
&\le(c_4 + nc_5c_6) \rho_1^n.
\end{align*}

On the set 
$\{\sigma_1 < \sigma_{1,2} < \tau_n < \tau_0 \le c_6n\}$,
we have $\tau= \tau_n$ and $S'_{\tau_n} = h_3(X_{\tau_n}, M_{\tau_n})$;
by definition $X_{\tau_n} \in \partial A_n.$
By definition of $h_3$ and by Proposition \ref{p:boundarylowerbound}
$h_3(x,m) \ge c_2 > 0$ for $x \in \partial A_n.$ These and the last
display imply
\[
c_2 {\mathbb P}_{(x,m)}(\sigma_1 < \sigma_{1,2} < \tau_n < \tau_0 \le c_6n) \le(c_4 + nc_5c_6) \rho_1^n.
\]	
Substitute this in \eqref{e:truncatetime} to get

\[
{\mathbb P}_{(x,m)}(\sigma_1 < \sigma_{1,2} < \tau_n < \tau_0 )
\le \rho_1^{n\left(1- \epsilon_n\right)}
\]
where
\[
\epsilon_n =\frac{1}{n}\log_{1/\rho_1}\left(\frac{c_4 + nc_5c_6}{c_2} \right);
\]

setting $n_0 \ge N_0$ so that $\epsilon_n < \epsilon$ for $n \geq n_0$
gives \eqref{e:upperboundforXrare}.	
\end{proof}

\subsection{$\rho_1 < \rho_2$}

The previous subsection gave a proof of
Proposition \ref{p:upperboundforXrare} for $\rho_2 < \rho_1.$
The only changes needed in this proof for
$\rho_1 < \rho_2$ concern the functions used in the definition
of the supermartingale $S$; the needed changes are:
\begin{enumerate}
\item Modify the function $h_2$ for the second stage,
\item The function $h_3$ is no longer superharmonic on $\partial_1$;
quantify how much it deviates from superharmonicity on $\partial_1$,
\item Modify the constants used in the definition of $S$ in accordance
with these changes.
\end{enumerate}

The next two propositions deal with the first two items above;
the definition of the supermartingale (taking also care of the
third item) is given after them.

The convexity of $q \mapsto -\log(\Lambda_1(e^q,e^q))$ and $\Lambda_1(\rho_1,\rho_1) =1$ imply
$\Lambda_1(\rho_2,\rho_2) < 1$ 
for $\rho_2 > \rho_1$.
Let ${\bm d}_{2}^+$ be a right eigenvector of ${\bm A}(\rho_2,\rho_2)$
with strictly positive entries.
\nomenclature{ ${\bm d}_{2}^+$ }{A right eigenvector of ${\bm A}(\rho_2,\rho_2)$ with strictly positive entries}
\begin{proposition}\label{p:guaranteed1}
The function
\[
f:(x,m) \mapsto [(\rho_2,\rho_2,{\bm d}_2^+), (T_n(x),m)]
\]
is superharmonic on ${\mathbb Z}_+^2 - \partial_1.$
On $\partial_1$ it satisfies
\begin{equation}\label{e:guaranteed1}
{\mathbb E}_{(x,m)} [f(X_1,M_1)] - f(x,m)
\le {\bm d}_2^+(m)\mu_1(m) (1 - \rho_2) \rho_2^{n}.
\end{equation}
\end{proposition}

The proof is parallel to that of Proposition \ref{p:hrho1harmwhere}
and follows from $\Lambda_1(\rho_2,\rho_2)<1$,
${\bm A}_2(\rho_2,\rho_2)={\bm A}(\rho_2,\rho_2)$ and the definitions involved.

\begin{proposition}\label{p:guaranteed2}
Let $h_3$ be as in \eqref{e:funforstage3}; $h_3$ is $(X,M)$-superharmonic
on ${\mathbb Z}_+^2 -\partial_1$; on $\partial_1$ it satisfies	
\begin{equation}\label{e:hrho2onpartial1c2}
{\mathbb E}_{(x,m)}\left[ h_{3}(X_1,M_1) \right]
- h_{3}(x,m)
= c_0{\bm d}_{2,1}(m)\mu_1(m) (1 - \alpha^*_1) \rho_2^{n-x(2)} > 0.
\end{equation}	
\end{proposition}
\begin{proof}
Lemma \ref{l:orderofalphastar} and $\rho_2 > \rho_1$ imply
$\alpha^*_1  < \rho_2$; this and Proposition \ref{p:boundarylowerbound}
imply that $c_1$ in the definition of $h_3$ is $0$; i.e., 	
\[
h_3(x,m) = h_{\rho_2}(T_n(x),m) = 
\rho_2^{n-(x(1)+x(2))} 
\left( {\bm d}_2(m) + c_0 (\alpha^*_1)^{x(2)} {\bm d}_{2,1}(m) \right);
\]	
That $h_3$ is $(X,M)$-superharmonic on ${\mathbb Z}_+^2-\partial_1$
follows from the same property of $h_{\rho_2}$ (see
Proposition \ref{p:hrho2}). On the other hand, again by Proposition \ref{p:hrho2}, 
$\alpha^*_1 < \rho_2$ implies that $c_0$ in the
definition of $h_{\rho_2}$ satisfies $c_0 > 0$. 
By Lemma \ref{l:rho2harXM} $(x,m)\mapsto
[(\rho_2,1,{\bm d}_2),(T_n(x),m)]$
is $(X,M)$-harmonic on $\partial_1$; \eqref{e:hrho2onpartial1c2}
follows from these and \eqref{e:subharmonicconjugate}.
\end{proof}
$\rho_2 > \rho_1$ implies $\rho_2 > \alpha^*_1$ (Lemma \ref{l:orderofalphastar});
this and Proposition \ref{p:boundarylowerbound} imply $c_1 = 0$;
$\rho_2 > \alpha^*_1$ and Proposition \ref{p:hrho2c0} imply
$c_0 > 0$. That $c_0 > 0$ and $c_1 =0$ lead to the following
modifications in the definition of $S'$:
\begin{align*}
S'_k &\doteq \begin{cases} h_1,&k \le \sigma_1, \\
h_4(X_k,M_k), & \sigma_1 < k \le \sigma_{1,2},\\ 
h_3(X_k,M_k), & k > \sigma_{1,2},
\end{cases}\\
S_k  &\doteq S'_k - c_5 k\rho_2^n,
\end{align*}
where
\begin{align*}
c_3 &\doteq 
\frac{\max_{m \in {\mathcal M}} \left( {\bm d}_2(m)   +c_0 {\bm d}_{2,1}(m)\right)  }{\min_{m \in {\mathcal M}} {\bm d}_2^+(m)},\\
h_1 &\doteq  c_4\rho_2^n,~~~~ c_4\doteq c_3 \max_{m \in {\mathcal M}}{\bm d}_2^+(m),\notag \\
h_4 &\doteq c_3 [(\rho_2,\rho_2, {\bm d}_2^+), (T_n(\cdot),\cdot)] = c_3 \rho_2^{n-x(1)}{\bm d}_2^+(\cdot),\\
c_5 &\doteq 
c_3 (1 - \rho_2) \max_{m \in {\mathcal M}} {\bm d}_2^+(m)\mu_1(m)+
c_0 (1 - \alpha^*_1) \max_{m \in {\mathcal M}} {\bm d}_{2,1}(m)\mu_1(m).
\end{align*}
The modification in $c_3$ ensures $h_4 \ge h_3$ on $\partial_2$;
$c_0 > 0$ implies that $h_3$ is no longer superharmonic on
$\partial_1$; the second term in $c_5$ compensates for this.
\begin{proposition}
$S$ as defined above is a supermartingale for $\rho_2 > \rho_1.$
\end{proposition}
\begin{proof}
With the modifications made as above, the proof proceeds exactly
as in the case $\rho_1 > \rho_2$ (Proposition \ref{p:supermartingalc1})
and follow from the following facts:
$h_1 \ge h_4$ on $\partial _1$, $h_4\ge h_3$ on $\partial_2$
(these are guaranteed by the choices of the constants $c_4$, $c_3$);
$(X,M)$-superharmonicity of $h_4$ and $h_3$ on ${\mathbb Z}_+^2-\partial_1$
(guaranteed by Propositions \ref{p:guaranteed1} and \ref{p:guaranteed2}), the $-c_5k \rho_2^{n}$ term compensating
for the lack of $(X,M)$-superharmonicity of $h_3$ and 
$h_4$ on $\partial_1$ (guaranteed by \eqref{e:guaranteed1} and \eqref{e:hrho2onpartial1c2}
and the choice of the constant $c_5$).
\end{proof}

\begin{proof}[Proof of Proposition \ref{p:upperboundforXrare};
case $\rho_2 > \rho_1$]
With $S$ defined as above, the proof given for the case
$\rho_1 > \rho_2$ works without change.
\end{proof}

\section{Lower bound for ${\mathbb P}_{(x,m)}(\tau_n < \tau_0)$}\label{s:lb}

To get an upper bound on the relative error \eqref{e:approximationerror},
we need a lower bound on the probability ${\mathbb P}_{(x,m)}(\tau_n < \tau_0).$
We will get the desired bound by applying the optional sampling theorem,
this time to an $(X,M)$-submartingale. This we will do, following
\cite{unlu2018excessive}, by constructing
a suitable $(X,M)$-subharmonic function. As opposed to superharmonic
functions, subharmonic functions are simpler to construct.

\begin{proposition}
\begin{align}\label{e:funforlowerbound}
(x,m) &\mapsto 
[(\rho_2,1, {\bm d}_2), (T_n(x),m)] \vee
[(\rho_1,\rho_1, {\bm d}_1), (T_n(x),m)] \\
&= \rho_2^{n-(x(1)+x(2))} {\bm d}_2(m) \vee \rho_1^{n-x(1)} {\bm d}_1(m)
\notag
\end{align}	
is $(X,M)$-subharmonic on ${\mathbb Z}_+^2.$
\end{proposition}
\begin{proof}
We know by Lemma \ref{l:roleofC} that
\begin{align*}
{\mathbb E}_{(x,m)} [(\rho_2,1, {\bm d}_2), (T_n(X_1),M_1)] 
&-
[(\rho_2,1, {\bm d}_2), (x,m)]\\ &= \rho_2^{n-x(1)}
{\bm P}(m,m) \mu_2(m) {\bm d}_2(m) (1-\rho_2) > 0,
\end{align*}
i.e, 
$(x,m) \mapsto [(\rho_2,1, {\bm d}_2), (x,m)]$ is $(X,M)$-subharmonic
on $\partial_2.$\\
That$(x,m) \mapsto [(\rho_2,1, {\bm d}_2), (T_n(x),m)]$ is 
$(X,M)$-subharmonic on ${\mathbb Z}_+^2-\partial_2$
follows from Lemma \ref{l:rho2harXM}. Then,
$(x,m) \mapsto [(\rho_2,1, {\bm d}_2), (x,m)]$ is $(X,M)$-subharmonic
on all of ${\mathbb Z}_+^2.$ 

Similarly, Proposition \ref{p:hrho1harmwhere} and \eqref{e:errortermforrho1}
imply that $(x,m) \mapsto [(\rho_1,\rho_1, {\bm d}_1), (x,m)]$
is $(X,M)$-subharmonic on all of ${\mathbb Z}_+^2$.

The maximum of two subharmonic functions is again subharmonic.
This and the above facts imply the $(X,M)$-subharmonicity of
\eqref{e:funforlowerbound}.
\end{proof}

\begin{proposition}\label{p:lowerbound}
\begin{align}\label{e:lowerbound}
&{\mathbb P}_{(x,m)}(\tau_n < \tau_0) \notag\\
&\ge
\left( \max_{m \in {\mathcal M}} ({\bm d}_2(m) \vee {\bm d}_1(m))\right)^{-1}
\\
&~~~~\times \left(
\rho_2^{n-(x(1)+x(2))} {\bm d}_2(m) \vee \rho_1^{n-x(1)} {\bm d}_1(m)\notag
-
\rho_2^{n} \max_{m \in {\mathcal M}}{\bm d}_2(m) 
\vee \rho_1^{n} \max_{m \in {\mathcal M}}{\bm d}_1(m)\right).
\end{align}
\end{proposition}
\begin{proof}
Set
\[
g(x,m) = \rho_2^{n-(x(1)+x(2))} {\bm d}_2(m) \vee \rho_1^{n-x(1)} {\bm d}_1(m);
\]	
by the previous proposition $g$ is $(X,M)$-subharmonic. By its definition,
$g$ is positive and bounded from above for $x\in {\mathbb Z}_+^2.$ It follows that
\[
s_k= g(X_{\tau_n \wedge \tau_0 \wedge k}, M_{\tau_n \wedge \tau_0 \wedge k})
\]
is a bounded positive submartingale.
By definition
\begin{equation}\label{e:identitysub}
\mathbb{E}[g(X_{\tau_n \wedge \tau_0 }, M_{\tau_n \wedge \tau_0} )]=\mathbb{E}[g(X_{\tau_n}, M_{\tau_n})1_{\{\tau_n<\tau_0\}}]+\mathbb{E}[g(X_{\tau_0}, M_{\tau_0})1_{\{\tau_0\le \tau_n\}}].
\end{equation}
That $X_{\tau_n} \in \partial A_n$ implies
$g(X_{\tau_n},M_{\tau_n}) = g(k,n-k)$ for some $k < n$;
then
\[
g(X_{\tau_n},M_{\tau_n}) \le \max_{m \in {\mathcal M }} ({\bm d}_2(m) \vee {\bm d}_1(m)).
\] This, \eqref{e:identitysub}
and the optional sampling theorem applied to $s$ at time $\tau_n\wedge \tau_0$ 
give
\[
{\mathbb P}_{(x,m)}(\tau_n<\tau_0)\left(
 \max_{m \in {\mathcal M}} ({\bm d}_2(m) \vee {\bm d}_1(m))\right)
+g(0, m){\mathbb P}_{(x,m)}(\tau_0\le \tau_n)
\ge g(x,m).
\]
${\mathbb P}_{(x,m)}(\tau_0\le\tau_n)\le 1$ implies
\[
\left( \max_{m \in {\mathcal M}} ({\bm d}_2(m) \vee {\bm d}_1(m))\right)
{\mathbb P}_{(x,m)}(\tau_n < \tau_0)
\geq g(x, m)-\max_{m \in {\mathcal M}}[g(0, m)];
\]
this and $\max \limits_{m \in {\mathcal M}}[g(0, m)]=\rho_2^{n} \max\limits_{m \in {\mathcal M}}{\bm d}_2(m) \vee \rho_1^{n} \max \limits_{m \in {\mathcal M}}{\bm d}_1(m)$ give \eqref{e:lowerbound}.
\end{proof}

\section{Completion of the limit analysis}\label{s:together}

This section puts together the results of the last two sections
to derive an exponentially decaying upper bound on the relative
error (\ref{e:approximationerror}). As in previous works
\cite{sezer2015exit, sezer2018approximation, unlu2018excessive}, this task is simplified if we express
the $Y$ process in the $x$ coordinates thus:
\[
\bar{X}_k \doteq T_n(Y_k);
\]
$\bar{X}$ has the same dynamics as $X$, except that it is not
constrained on $\partial_1.$
In this section we will set the initial condition using the scaled
coordinate $x \in {\mathbb R}_+^2$, $x(1) + x(2) < 1$,
the initial condition for the $X$ and $\bar{X}$
will be
\[
X_0 = \bar{X}_0 = \lfloor nx \rfloor.
\]
As in the non-modulated case, the following relation between
$\bar{X}$ and $X$ will be very useful:

\begin{lemma}\label{l:sumsequal}
Let $\sigma_{1,2}$ be as in \eqref{d:defsigma12}.
Then
\[
X_k(1) + X_k(2) = \bar{X}_k(1) + \bar{X}_k(2)
\]
for $ k \le \sigma_{1,2}.$
\end{lemma}

This lemma is the analog of \cite[Proposition 7.2]{sezer2015exit},
which expresses the same fact for the non-modulated two dimensional
tandem walk; the proof is unchanged because it
does not depend on the modulating process. 
Example sample paths of
$X$ and $\bar{X}$ up to time $\sigma_{1,2}$ demonstrating Lemma 
\ref{l:sumsequal} are shown in Figure \ref{f:exp}.

\begin{figure}[h]
\begin{center}
\scalebox{1}{
\centerline{\input{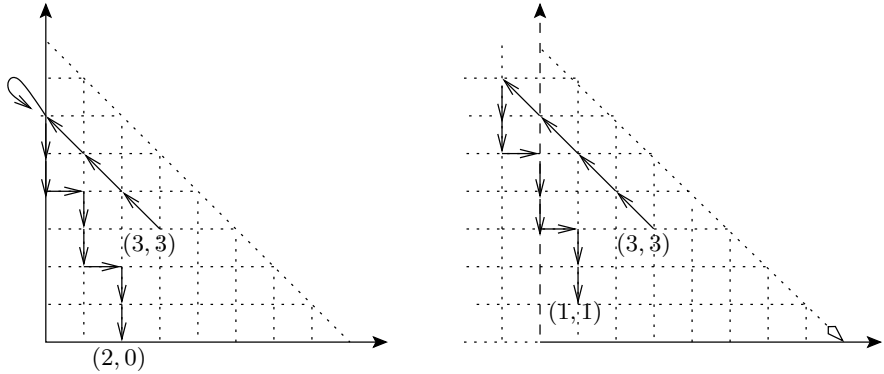}}}
\end{center}

\vspace{-0.65cm}
\caption{\hspace{0.25cm}A sample path of $X_k$(left) and $\bar{X}_k$(right) \label{f:exp}}
\end{figure}

Define 
\begin{align*}
\bar{\tau}_n &\doteq \inf\{k > 0: \bar{X}_k \in \partial A_n\},\\
\bar{\sigma}_{1,2} &\doteq \inf\{k > 0: \bar{X}_k(1) + \bar{X}_k(2) = 0,~
k \ge \sigma_1\}.
\end{align*}

$X$ and $\bar{X}$ have identical dynamics upto time
$\sigma_1$;
$\bar{\sigma}_{1,2}$ is the first time after ($\sigma_1$, i.e.,
the first time $X$ and $\bar{X}$ hit $\partial_1$) that the sum
of the components of $\bar{X}$ equals $0$. 
By the definitions of $\bar{X}$ and $Y$, $\bar{\tau}_n = \tau.$

What follows is 
an upper bound similar to \eqref{e:upperboundforXrare} for the $\bar{X}$
process. This is a generalization of \cite[Proposition 7.5]{sezer2015exit} to the present setup:

\begin{proposition}\label{p:upperboundforX}
For any $\epsilon > 0$ there exists $n_0 > 0 $ such that
\begin{equation}\label{e:upperboundforX}
{\mathbb P}_{(x,m)}( \sigma_1 < \sigma_{1,2} < \bar{\tau}_n < \infty)
\le \rho^{n(1-\epsilon)}
\end{equation}
for $n > n_0$ and $(x,m) \in A_n$.
\end{proposition}
\begin{proof}
As in \cite[Proposition 7.5]{sezer2015exit} we partition the event
$\{\sigma_1 < \sigma_{1,2} < \bar{\tau}_n < \infty\}$ into whether
$\bar{X}$ hits $\partial A_n$ before or after it hits 
$\{x \in {\mathbb Z}\times {\mathbb Z}_+: x(1) + x(2) = 0\}$:
\begin{align}\label{e:partbarsigma}
&{\mathbb P}_{(x,m)}( \sigma_1 < \sigma_{1,2} < \bar{\tau}_n < \infty)\\
&~~~=\notag
{\mathbb P}_{(x,m)}( \sigma_1 < \sigma_{1,2} < \bar{\tau}_n < \bar{\sigma}_{1,2} < \infty)+
{\mathbb P}_{(x,m)}( \sigma_1 < \sigma_{1,2} < \bar{\sigma}_{1,2} < \bar{\tau}_n < \infty)
\end{align}
Lemma \ref{l:sumsequal} implies
\[
\bar{X}_{\sigma_{1,2}}(1) +\bar{X}_{\sigma_{1,2}}(2)
=
X_{\sigma_{1,2}}(1) +X_{\sigma_{1,2}}(2)
\]
i.e., at time $\sigma_{1,2}$, $X$ and $\bar{X}$ will be on the same line 
$\{x \in {\mathbb Z} \times {\mathbb Z}_+: x(1) + x(2) = k \}$
for some $k \in \{1,2,...,n-1\}$.  Then 
for 
$\omega \in \{\sigma_1 < \sigma_{1,2} < \bar{\tau}_n\}$
the fully constrained sample
path $X(\omega)$ cannot hit $0$ before the
path
$\bar{X}(\omega)$ hits $\{x \in {\mathbb Z} \times {\mathbb Z}_+: x(1) +x(2) = 0\}$ and it cannot hit $\partial A_n$ after $\bar{X}$ hits
$\{x \in {\mathbb Z} \times {\mathbb Z}_+: x(1) +x(2) = n\}$ (intuitively:
more constraints on $X$ push it faster to $\partial A_n$ and slower
to $0$ than less constraints do the process $\bar{X}$):
these give
\[
\{ \sigma_1 < \sigma_{1,2} < \bar{\tau}_n < \bar{\sigma}_{1,2} < \infty\}
\subset
\{\sigma_1 < \sigma_{1,2} < \tau_n < \tau_0 \};
\]
the bound \eqref{e:upperboundforXrare} on the probability of the last
event and \eqref{e:partbarsigma} imply that there exists $n_1 > 0$
such that
\begin{equation}\label{e:propstep2}
{\mathbb P}_{(x,m)}( \sigma_1 < \sigma_{1,2} < \bar{\tau}_n < \infty)
\le
\rho^{n(1-\epsilon/2)}
+
{\mathbb P}_{(x,m)}( \sigma_1 < \sigma_{1,2} < \bar{\sigma}_{1,2} < \bar{\tau}_n < \infty)
\end{equation}	
for $n > n_1.$

To bound the last probability we observe that
$\bar{X}_{\bar\sigma_{1,2}}$ lies on 
$\{x \in {\mathbb Z}\times{\mathbb Z}_+: x(1) +x(2) = 0\};$
by Proposition \ref{p:upperfortau},
starting from this line, the probability of $\bar{X}$ ever  hitting 
$\{x \in {\mathbb Z}\times{\mathbb Z}_+: x(1) +x(2) = n\}$
is bounded from above by	
\begin{align*}
\frac{1}{c_2} ( h_{\rho_2}((n-x(1),x(2)),m) &+ 
c_1 h_{\rho_1}((n-x(1),x(2)),m) ) \\&\le
\frac{1}{c_2} \left( \rho_2^{n} {\bm d}_2(m) + 
c_1 \rho_1^{n} {\bm d}_1(m) \right);
\end{align*}	
this and the strong Markov property of $\bar{X}$ give:	
\[
{\mathbb P}_{(x,m)}( \sigma_1 < \sigma_{1,2} < \bar{\sigma}_{1,2} < \bar{\tau}_n < \infty)
\le c_7 \rho^n
\]	
where $c_7$ is a constant depending on ${\bm d}_1$, ${\bm d}_2$, $c_1$
and $c_2.$
\nomenclature{$c_7$}{a constant depending on ${\bm d}_1$, ${\bm d}_2$, $c_1$
and $c_2.$}
Substituting this in \eqref{e:propstep2} gives	
\[
{\mathbb P}_{(x,m)}( \sigma_1 < \sigma_{1,2} < \bar{\tau}_n < \infty)
\le
\rho^{n(1-\epsilon/2)} + c_7 \rho^n,
\]
for $n > n_1.$ This implies the statement of the proposition.	
\end{proof}

Finally, we state and prove our main theorem:

\begin{theorem}\label{t:mainapprox}
For any
$x \in {\mathbb R}_+^2$, 
$x(1) + x(2) < 1$, 
and $m \in {\mathcal M}$
(if $\rho_1 > \rho_2$ and $x(2) < 1- \log(\rho_1)/\log(\rho_2)$
we also require $x(1) > 0$ )
there exists $c_8 > 0$ and $N > 0$ such that
\begin{equation}\label{e:relativeerror}
\frac{ |{\mathbb P}_{(x_n,m)}(\tau_n < \tau_0) - 
	{\mathbb P}_{(T_n(x_n),m)}(\tau <  \infty)|}{
	{\mathbb P}_{(x_n,m)}(\tau_n < \tau_0)} < \rho^{c_8 n}
\end{equation}

for $n > N$, where $x_n = \lfloor n x \rfloor$.
\end{theorem}
\begin{proof}
Proposition \ref{p:lowerbound}, the choice of $x$
(i.e.,  $x(1) + x(2) < 1$ and 
$x(1) > 0$
and furthermore
$x(2) < 1-\log(\rho_1)/\log(\rho_2)$)
when $\rho_1 > \rho_2$)
imply the lower bound 
\begin{equation}\label{e:boundondenom}
\mathbb{P}_{(x,m)}(\tau_n<\tau_0)\geq \rho^{n(1-2c_8)}
\end{equation} 
for some constant $1/2 > c_8 > 0$ depending on $x$.

By definition $\bar{X}$ hits $\partial A_n$ exactly when
$Y$ hits $\partial B$, i.e., $\bar{\tau}_n = \tau$; therefore,
$\mathbb{P}_{(x_n, m)}({\bar \tau}_n <\infty)={\mathbb P}_{(T_n(x_n),m)}(\tau <  \infty)$ and
\begin{align}\label{e:ratioproof}
&\frac{ |{\mathbb P}_{(x_n,m)}(\tau_n < \tau_0) - 
	{\mathbb P}_{(T_n(x_n),m)}(\tau <  \infty)|}{
	{\mathbb P}_{(x_n,m)}(\tau_n < \tau_0)} \\
&~~= \frac{|\mathbb{P}_{(x_n, m)}(\tau_n<\tau_0)-
\mathbb{P}_{(x_n, m)}({\bar \tau}_n <\infty)|}{\mathbb{P}_{(x_n, m)}(\tau_n<\tau_0)} \notag
\end{align}

We partition the probabilities of events 
$\{\tau_n < \tau_0\}$ and $\{ \tau <  \infty\}$ as follows	
\begin{align}\label{e:part1}
\mathbb{P}_{(x_n, m)}(\tau_n <\tau_0) &\notag = \mathbb{P}_{(x_n, m)}(\tau_n<\sigma_1<\tau_0) + \mathbb{P}_{(x_n, m)}(\sigma_1<\tau_n\leq \sigma_{1, 2}\wedge \tau_0)\\ &+ \mathbb{P}_{(x_n, m)}(\sigma_1<\sigma_{1, 2}<\tau_n<\tau_0)
\end{align} 
\begin{align}\label{e:part2}
\mathbb{P}_{(T_n(x_n), m)}(\tau<\infty) &\notag= \mathbb{P}_{(T_n(x_n), m)}(\tau<\sigma_1) +\mathbb{P}_{(T_n(x_n), m)}(\sigma_1<\tau\leq \sigma_{1, 2})\\ &+\mathbb{P}_{(T_n(x_n), m)}(\sigma_1<\sigma_{1, 2}< \tau<\infty).
\end{align}

Lemma \ref{l:sumsequal} says the processes $X$ and $\bar{X}$ move together until they hit $\partial_1$, so 
\begin{equation*}
\mathbb{P}_{(x_n, m)}(\tau_n<\sigma_1<\tau_0)=\mathbb{P}_{(T_n(x_n), m)}(\tau<\sigma_1).
\end{equation*}
After hitting $\partial_1$, the sum of the components of $X$ and $\bar{X}$ are still equal until one of the processes hits $\partial_2$. Lemma \ref{l:sumsequal} now gives
\begin{equation*}
\mathbb{P}_{(x_n, m)}(\sigma_1<\tau_n\leq \sigma_{1, 2}\wedge \tau_0)=\mathbb{P}_{(T_n(x_n), m)}(\sigma_1<\tau\leq \sigma_{1, 2}).
\end{equation*}
The last two equalities, Propositions \ref{p:upperboundforXrare}, \ref{p:upperboundforX}, and partitions \eqref{e:part1}, \eqref{e:part2} imply that there
exists $n_0 > 0$ such that
\begin{equation}\label{e:boundnumerator}
\mid \mathbb{P}_{(x_n, m)}(\sigma_1<\sigma_{1,2}<\tau_n<\tau_0)-\mathbb{P}_{(T_n(x_n), m)}(\sigma_1<\sigma_{1, 2}<\tau<\infty) 
\mid \leq {\rho}^{n(1-c_8)}
\end{equation}
for $n > n_0.$
Substituting the last bound and \eqref{e:boundondenom}  in
\eqref{e:ratioproof} gives \eqref{e:relativeerror}.
\end{proof}

\section{Computation of ${\mathbb P}( \tau < \infty)$}\label{s:computation}

Theorem \ref{t:mainapprox} tells us that ${\mathbb P}_{(T_n(x_n),m)}(\tau < \infty)$
approximates ${\mathbb P}_{(x_n,m)}(\tau_n < \tau_0)$ very well. In this section
we develop approximate formulas for ${\mathbb P}_{(y,m)}(\tau < \infty).$ 
Recall that a $(Y,M)$-harmonic function is said to be $\partial B$-determined
if it of the form
\[
(y,m) \mapsto {\mathbb E}_{(y,m)}[ f(Y_\tau,M_\tau) 1_{\{\tau<\infty\}}]
\]
for some function $f$. 
The function
\begin{equation}\label{e:tauyfun}
(y,m) \mapsto {\mathbb P}_{(y,m)}(\tau < \infty)
\end{equation}
is $(Y,M)$-harmonic with $f=1$ on $\partial B$. Furthermore,
by definition it is $\partial B$-determined, 
(for \eqref{e:tauyfun}, $f$ is 
the function taking the constant value $1$ on $\partial B$). Our approach
to the approximation of ${\mathbb P}_{(y,m)}(\tau < \infty)$
is based on the classical superposition principle: take linear
combinations of
the $(Y,M)$-harmonic functions 
identified in Propositions \ref{p:pointonboth} and \ref{p:harconjugate}
to approximate the value $1$ on $\partial B$
as closely as possible. We need our $(Y,M)$-harmonic functions
to be $\partial B$-determined; the next lemma
 identifies a simple
condition for functions of the form \eqref{e:linearhar} to be 
$\partial B$-determined. 

\begin{lemma}\label{l:partialBdetermined}
Suppose $(\beta,\alpha_j,d_j)$ are points on ${\mathcal H}$ and
suppose
\begin{equation}\label{e:linearhar}
h(y,m) = \sum_{j=1}^{k} b(j) [(\beta,\alpha_j,d_j),\cdot],
\end{equation}
$k \ge 1$,
is $(Y,M)$-harmonic. If $|\beta | < 1$ and $|\alpha_j|\le 1$ then
$h$ is $\partial B$-determined.
\end{lemma}
This generalizes \cite [Proposition 2.2, 4.10]{sezer2015exit} 
to the Markov modulated setup.
\begin{proof}

Define the region 
$U=\{y\in \mathbb{Z}\times\mathbb{Z}_+:0\le y(1)-y(2) \le n\}$
and the boundaries of $U$
$\partial U_1=\{y\in \mathbb{Z}\times\mathbb{Z}_+:y(1)-y(2)= n\}$ and
 $\partial U_2 =\partial B$. Define $\upsilon_n \doteq \inf \{k: Y_k \in \partial U_1\}.$
We make the following claim:
starting from a point $y \in U$,
$(Y,M)$ hits $\partial U_1 \cup \partial U_2$ in finite time, 
i.e., $\upsilon_n \wedge \tau < \infty$ almost surely. Let us first prove this claim.
For each modulating state $m$, 
the sample path of $(Y,M)$ consisting only of increments
$(0,-1)$ hits $\partial U_2$ in at most $n$ steps and the probability of this 
path is $(\lambda(m){\bm P}(m,m))^n.$ Then if we set
\[\varepsilon=\min_{m \in {\mathcal M}} 
(\lambda(m){\bm P}(m,m))^n
\]
we have
\[\mathbb{P}_{(y,m)}(\tau \wedge \upsilon_n \ge n) \le (1-\varepsilon).
\]
An iteration of this inequality and the Markov property of $(Y,M)$ give
\[\mathbb{P}_{(y,m)}(\tau \wedge \upsilon_n \ge kn) \le (1-\varepsilon)^k.
\]
Letting $k\rightarrow \infty$ gives
\begin{equation}\label{e:tauvnfinite}
\mathbb{P}_{(y,m)}(\tau \wedge \upsilon_n = \infty ) = 0.
\end{equation}

Definition \eqref{e:linearhar}  and $|\alpha_j| \le 1$, $|\beta| < 1$
imply that $h$ is bounded on $B$. This and that $h$ is $(Y,M)$-harmonic
imply that 
\[
S_k=h(Y_{\tau \wedge \upsilon_n \wedge k}, M_{\tau \wedge \upsilon_n \wedge k})
\]
is a bounded martingale. 
The optional sampling theorem applied to this martingale and 
\eqref{e:tauvnfinite} imply
\begin{align}\label{e:prelimitch34}
h(y,m)&=\mathbb{E}_{(y,m)}[h(Y_{\tau \wedge \upsilon_n}, M_{\tau \wedge \upsilon_n})]\\
&=\mathbb{E}_{(y,m)}[h(Y_{\tau}, M_{\tau})1_{\{\tau < \upsilon_n\}}]+ \mathbb{E}_{(y,m)}[h(Y_{\upsilon_n}, M_{\upsilon_n})1_{\{\upsilon_n\le \tau\}}].\notag
\end{align}
That  $|\alpha_j|\le 1$ implies 
$|h(Y_{\upsilon_n}, M_{\upsilon_n})|\le c \beta^n$ for
some constant $c > 0$. Therefore,
\[
\lim_{n\rightarrow\infty}
\mathbb{E}_{(y,m)}[h(Y_{\upsilon_n}, M_{\upsilon_n})1_{\{\upsilon_n\le \tau\}}]\le 
c\lim_{n\rightarrow\infty}\beta^n=0. 
\] 
The last expression, that $\lim_{n\rightarrow\infty} \upsilon_n=\infty$ 
and letting $n \rightarrow\infty$ in \eqref{e:prelimitch34}
imply 
\[
h(y,m)=\mathbb{E}_{(y,m)}[h(Y_{\tau}, M_{\tau})1_{\{\tau < \infty\}}],
\]
i.e, $h(y,m)$ is $\partial B$-determined. 
\end{proof}

The last lemma and $0 < \rho_1 < 1$ imply
\begin{lemma}\label{l:rho1pB}
$h_{\rho_1}$ is $\partial B$-determined.
\end{lemma}

Recall that we have constructed a $(Y,M)$-superharmonic function, 
$ h_{\rho_2}$ from the roots $(\rho_2,1), (\rho_2,\alpha^*_1) \in {\mathcal H}^{\beta\alpha}.$
We would like to strengthen this to a $(Y,M)$-harmonic function. This
requires the use of further conjugate points of $(\rho_2,1)$ (in
addition to $(\rho_2,\alpha^*_1)$). The next lemma shows that
under Assumptions \ref{as:simpleeig} and \eqref{as:conj} 
we have sufficient number of  conjugate points of $(\rho_2,1)$ to work with:
\begin{lemma}\label{l:point2conjrest}
Let $(\rho_2,\alpha_1^*)$ be the point conjugate to $(\rho_2,1)$
identified in Proposition \ref{p:point2conj}.
Under Assumptions \ref{as:simpleeig} and \eqref{as:conj}, there exists
$|{\mathcal M}|-1$ additional
conjugate points $(\rho_2,\alpha^*_j)$, $j=2,3,...,|{\mathcal M}|$,
of $(\rho_2,1)$ with $ 0 < \alpha^*_j < \alpha^*_1.$
\end{lemma}
\nomenclature{$\alpha_j^*$}{the $\alpha$ component of $|{\mathcal M}|-1$ conjugate points of $(\rho_2,1)$}
\begin{proof}
We know that $\Lambda_1(\rho_2,\alpha_1^*) = 1$; 
then $\Lambda_j(\rho_2,\alpha_1^*) < 1$
for $j=2,3,...,|{\mathcal M}|.$ On the other hand, Gershgorin's Theorem
implies 
$\lim_{\alpha\rightarrow 0} \Lambda_j(\rho_2,\alpha)  = \infty.$
These and the continuity of $\Lambda_j$ 
imply the existence of $\alpha_j^* \in (0,\alpha_1^*)$ such that
$\Lambda_j(\rho_2,\alpha_j^*) =1.$
\end{proof}

To construct our $(Y,M)$-harmonic functions from the points identified
in the previous lemma we need the following assumption:
\begin{equation}\label{as:mu1neqmu2s2}
{\bm c}(\rho_2,1,{\bm d}_2) \in 
\text{Span}\left(
{\bm c}(\rho_2,\alpha^*_{j},{\bm d}_{2,j}), j=1,2,...,|{\mathcal M}|\right).
\end{equation}

\begin{remark}\label{r:mu1neqmu2s2}
By definition,
${\bm c}(\rho_2,\alpha_j, {\bm d}_{2,j}) = 0$ if
$\alpha_j = \rho_2$. Therefore,
only those $j$ satisfying $\alpha_{j}^* \neq \rho_2$ have a role
in determining 
$\text{Span}\left(
{\bm c}(\rho_2,\alpha^*_{j},{\bm d}_{2,j}), j=1,2,...,|{\mathcal M}|\right).$
In this sense, assumption \eqref{as:mu1neqmu2s2} can be seen
as an extension of 
\eqref{as:mu1neqmu2alt}
(or, equivalently, of \eqref{as:mu1neqmu2}).
\end{remark}

\begin{remark}\label{r:mu1neqmu2s2}
The linear independence of 
${\bm c}(\rho_2,\alpha_j^*, {\bm d}_{2,j})$, $j=1,2,...,|{\mathcal M}|$, is
sufficient for \eqref{as:mu1neqmu2s2} to hold. That 
${\bm c}(\beta,\beta,{\bm d}) = 0$ 
implies that $\rho_2\neq \alpha_j^*$ for all $j=1,2,..,|{\mathcal M}|$
is a necessary condition for this independence.
\end{remark}

Now on to the $(Y,M)$-harmonic function:

\begin{proposition}\label{p:hrho2c}
Let $(\rho_2,\alpha_j^*)$ be the conjugate points of $(\rho_2,1)$
identified in Proposition \ref{p:point2conj} and Lemma 
\ref{l:point2conjrest}. Under the additional assumption 
\eqref{as:mu1neqmu2s2}, one can find a vector 
${\bm b}_{2,1} \in {\mathbb R}^{m_1}$  such that
\begin{equation}\label{def:hrho2c}
{\mathfrak h}_{\rho_2} \doteq
[(\rho_2,1, {\bm d}_2),\cdot] + 
\sum_{j=1}^{|{\mathcal M}|} {\bm b}_{2,1}(j) [ (\rho_2, \alpha^*_{j}, {\bm d}_{2,j}),\cdot]
\end{equation}
is $(Y,M)$-harmonic and $\partial B$-determined.
\end{proposition}
\nomenclature{${\mathfrak h}_{\rho_2}$}{$(Y, M)$-harmonic function
constructed from points on ${\mathcal H}$ with $\beta=\rho_2$}
\begin{proof}
Assumption (\ref{as:mu1neqmu2s2}) implies that the collection
of vectors \\${\bm c}(\rho_2,1,{\bm d}_2)$, 
${\bm c}(\rho_2,\alpha^*_{j}, {\bm d}_{2,j})$,
$j=1,2,...,|{\mathcal M}|$ are linearly dependent. 
Therefore, by Proposition \ref{p:harconjugate}, 
there exists a vector $b' \in {\mathbb R}^{|{\mathcal M}| + 1}$
such that
\[
b'(0) [(\rho_2, 1, {\bm d}_2), \cdot] + 
\sum_{k=1}^{n_1} b'(j) [ (\rho_2,\alpha^*_{j_k}, {\bm d}_{2,j_k}),\cdot]
\]
is $(Y,M)$-harmonic. Assumption \eqref{as:mu1neqmu2s2} implies that
one can choose $b'$ so that $b'(0) \neq 0.$
Renormalizing the last display by $b'(0)$ gives 
\eqref{def:hrho2c}.
That ${\mathfrak h}_{\rho_2}$ is $\partial B$-determined follows from $0 < \alpha^*_j \le 1$, $\rho_2 < 1$
and Lemma \ref{l:partialBdetermined}.
\end{proof}
Next proposition constructs an
approximation of ${\mathbb P}_{(y,m)}(\tau <\infty)$ with
bounded relative error from functions ${\mathfrak h}_{\rho_2}$ and
$h_{\rho_1}$.
\begin{proposition}\label{p:firstrelerror}
There exist constants $c_9$, $c_{10}$ and $c_{11}$ such that
\begin{equation}\label{e:uplowbound}
{\mathbb P}_{(y,m)} ( \tau < \infty)  < h^{a,0}(y,m) < c_9 {\mathbb P}_{(y,m)}( \tau < \infty)
\end{equation}
where
\begin{equation}\label{d:ha0}
h^{a,0} \doteq c_{11} ({\mathfrak h}_{\rho_2}  + c_{10} h_{\rho_1}).
\end{equation}	
\end{proposition}
\nomenclature{$h^{a,0}$}{an approximation of 
${\mathbb P}_{(y,m)}(\tau <\infty)$ constructed from
a linear combination
of ${\mathfrak h}_{\rho_2}$ and $h_{\rho_1}.$}
\begin{proof}
The proof is similar to that of Proposition \ref{p:boundarylowerbound}.
That $0 < \alpha^*_j < 1$, $j=2,3,...,|{\mathcal M}|$ imply that	 
\begin{equation}\label{e:term1asymptotics}
[(\rho_2, \alpha^*_j, {\bm d}_{2,j}), (k,k,m)] =  (\alpha^*_j)^k {\bm d}_{2,j}(m) \rightarrow 0
\end{equation}	
as $k\rightarrow \infty.$ We further have	
\begin{equation}\label{e:term2asymtotics}
[(\rho_2, 1, {\bm d}_{2}), (k,k,m)] =  {\bm d}_{2}(m) >0,
\end{equation}
for all $k \ge 0$.
This and \eqref{e:term1asymptotics} imply that there exists $k_0 > 0$ such that
\begin{equation}\label{e:afterk0}
{\mathfrak h}_{\rho_2}(k,k,m) > \min_{m \in {\mathcal M}}{\bm d}_2(m)/2
\end{equation}
for all $k > k_0.$
On the other hand, 
\begin{equation}\label{e:term3asymptotics}
h_{\rho_1}(k,k,m) = [(\rho_1, \rho_1, {\bm d}_{1}), (k,k,m)] =  {\bm d}_{1}(m)
\rho_1^k >0,
\end{equation}
for all $k$. Then we can choose $c_{10}> 0$ large enough so that
\begin{equation}\label{e:lowerboundmbhrho2}
{\mathfrak h}_{\rho_2}(k,k,m) + c_{10}  h_{\rho_1}(k,k,m) \ge 
\min_{m \in {\mathcal M}}{\bm d}_2(m)/2
\end{equation}
for all $k \le k_0.$ The last display, \eqref{e:afterk0} and the positivity of
$c_{10} h_{\rho_1}$ imply that the last display holds for all $k$ and $m \in {\mathcal M}.$
Set 
\[
c_{11} \doteq \left(\min_{m \in {\mathcal M}}{\bm d}_2(m)/2\right)^{-1},
\]
and $h^{a,0}$ be as in \eqref{d:ha0}.
That \eqref{e:lowerboundmbhrho2} holds for $k \ge 0$ and $m \in {\mathcal M}$ implies
\[
h^{a,0} |_{\partial B} \ge 1.
\]
By Lemma \ref{l:rho1pB} and Proposition \ref{p:hrho2c} $h^{a,0}$ is $(Y,M)$-harmonic and $\partial B$-determined.
This and the last display imply,	
\begin{equation}\label{e:prooffirstbound}
h^{a,0}(y,m) = {\mathbb E}_{(y,m)}[ h^{a,0}(Y_\tau, M_\tau) 1_{\{\tau < \infty \}}] \ge {\mathbb P}_{(y,m)}(\tau < \infty).
\end{equation}
This proves the first inequality in \eqref{e:uplowbound}. To choose $c_9$ so that the second inequality in
\eqref{e:uplowbound} holds we note the following:  \eqref{e:term1asymptotics}, \eqref{e:term2asymtotics}
and \eqref{e:term3asymptotics} imply 
\[
c_9 \doteq \max_{k \ge 0, m \in {\mathcal M}} h^{a,0}(k, k, m) < \infty.
\]

Now the same argument giving \eqref{e:prooffirstbound} implies the second inequality in \eqref{e:uplowbound}.
\end{proof}

\begin{proposition}\label{p:boundedrelerrapp1}
Fix $m \in {\mathcal M}$  and
$x \in {\mathbb R}_+^2$, such that $0 < x(1) + x(2) < 1$; furthermore
assume $x(1) > 0$ if $\rho_1 > \rho_2$ and $x(2) \le 1- \log(\rho_1)/\log(\rho_2)$;
set $x_n = \lfloor nx \rfloor.$ Then $h^{a,0}(T_n(x_n))$ approximates
${\mathbb P}_{(x_n,m)}(\tau_n < \tau_0)$ with relative error whose $\limsup$ in $n$ is bounded by $|c_9 -1|$.
\end{proposition}
\begin{proof}
We know by the previous proposition that 
$h^{a,0}$ approximates ${\mathbb P}_{(y,m)}(\tau < \infty)$ with
relative error bounded by $|c_9-1|$; we also know by Theorem
\ref{t:mainapprox} that \\ ${\mathbb P}_{(T_n(x_n),m)}(\tau < \infty)$ approximates
${\mathbb P}_{(x_n,m)}(\tau_n < \tau_0)$ with vanishing relative error.
These two imply the statement of the proposition.
\end{proof}

\section{Improving the approximation}\label{s:improveapp}

Proposition \ref{p:boundedrelerrapp1} tells us that $h^{a,0}$ of \eqref{d:ha0} approximates
${\mathbb P}_{(y,m)}(\tau < \infty)$ and therefore ${\mathbb P}_{(x,m)}(\tau_n < \tau_0)$ with bounded relative error.
The works 
\cite{sezer2015exit,sezer2018approximation, unlu2018excessive} 
covering the non-modulated case
are able to construct progressively better approximations 
(i.e., reduction of the relative error)
by using more harmonic functions constructed from conjugate points 
(in the tandem case 
with no modulation, one is able
to construct an exact representation of ${\mathbb P}_y(\tau <\infty)$ so no reduction in relative error is necessary).
This is possible because the function in 
\cite{sezer2015exit, sezer2018approximation, unlu2018excessive} corresponding 
to ${\mathfrak h}_{\rho_2}$, takes the value
$1$ on $\partial B$ away from the origin. 
Thus, by and large, that single function provides an excellent
approximation of ${\mathbb P}_y(\tau < \infty)$ for points away from
 $\partial_2$.
Rest of the harmonic functions are added
to the approximation to improve the approximation along
$\partial_2.$

When a modulating chain is present, the situation is different. Note that \eqref{e:term1asymptotics}, \eqref{e:term2asymtotics} imply that 
the value
of ${\mathfrak h}_{\rho_2}$ on $\partial B$, away from the origin, 
is determined by the eigenvector ${\bm d}_2$ and in general, the components of
${\bm d}_2$ will change with $m$. We need to improve ${\mathfrak h}_{\rho_2}$
 itself so that we have a $(Y,M)$-harmonic function that 
is close to $1$ on $\partial B$ away from the origin.

How is this to be done? Remember that the 
construction of ${\mathfrak h}_{\rho_2}$
began with fixing $\alpha=1$
and solving 
\begin{equation}\label{e:solveforbeta}
\beta^{|{\mathcal M}|}{\bm p}(\beta,1) = 0;
\end{equation}
$\rho_2$ is the largest root of this equation in the interval $(0,1)$.
Then we fixed $\beta = \rho_2$ in
$\alpha^{|{\mathcal M}|}{\bm p}(\rho_2,\alpha) = 0$ and solved for $\alpha$
to find the conjugate points $(\rho_2,\alpha^*_j)$ of $(\rho_2,1)$; from these points we constructed
${\mathfrak h}_{\rho_2}$.
Now to get our $(Y,M)$-harmonic function that almost takes the value $1$ on $\partial B$ away from the origin we will
use the rest of the roots of \eqref{e:solveforbeta} in $(0,1)$.
The next lemma shows that under the stability assumption and the simpleness of all eigenvalues, 
$|{\mathcal M}|-1$ real $\beta$ roots exist that lies in the interval 
$(0,\rho_2).$
The proposition after that constructs the desired
$(Y,M)$-harmonic function from these roots.

\begin{lemma}\label{l:thees}
Under the stability assumption (\ref{e:stabilityas}),
and Assumption \ref{as:simpleeig} (
all eigenvalues of ${\bm A}(\beta,\alpha)$ are real and simple
for $(\beta,\alpha) \in {\mathbb R}_+^{2o}$)
there exist
$\rho_{2,j}$, $j=2,3,...,|{\mathcal M}|$, such that
$\rho_2 > \rho_{2,2} > \rho_{2,3} > \cdots > \rho_{2,|{\mathcal M}|} > 0$
and $\{{\bm e}_{2}\neq 0$, ${\bm e}_{3}\neq 0$,...,${\bm e}_{|{\mathcal M}|}\neq0\}
\subset {\mathbb R}^{{\mathcal M}}$ such that
\[
{\bm A}(\rho_{2,j},1) {\bm e}_j = {\bm e}_j,~ j =2,3,...,|{\mathcal M}|,
\]
holds.
\end{lemma}
\nomenclature{$\rho_{2,j}$}{ for $\alpha=1$, the $\beta$ roots less than $\rho_2$ in the interval $(0,1)$ }
\nomenclature{${\bm e}_j$}{A right eigenvector of ${\bm A}(\rho_{2,j},1)$}
The proof is parallel to that of Lemma \ref{l:point2conjrest} and is
based on Gershgorin's Theorem and the fact that
$\Lambda_j(\rho_2, 1) < 1$ for $j=2,3,...,|{\mathcal M}|.$

Each of the points $(\rho_{2,j},1)$ will in general have $2|{\mathcal M}|-1$
conjugate points. To get $\partial B$-determined $(Y,M)$-harmonic functions
from these we need the analog of \eqref{as:mu1neqmu2s2} for each 
$(\rho_{2,j},1)$:
\begin{assumption}\label{as:mu1neqmu2s3}
For each $j=2,3,...,|{\mathcal M|}$ there exists $m_j \le |{\mathcal M}|$
conjugate points $(\rho_{2,j}, \alpha^*_{j,l})$, $l=1,2,...,m_j$, of
$(\rho_{2,j},1)$
and eigenvectors $0\neq {\bm e}_{j,l} \in {\mathbb R}^{{\mathcal M}}$ such
that
\begin{align}
|\alpha^*_{j,l}| &< 1,~ l=1,2,...,m_j,\notag \\
{\bm A}(\rho_{2,j}, \alpha^*_{j,l}){\bm e}_{j,l} &= {\bm e}_{j,l}\notag \\
{\bm c}(\rho_{2,j}, 1, {\bm e}_j) &\in 
\text{Span}( {\bm c}(\rho_{2,j}, \alpha^*_{j,l},{\bm e}_{j,l}), l=1,2,..,m_j).
\label{e:lisinthespan}
\end{align}
\end{assumption}
\nomenclature{$\alpha^*_{j,l}$}{for $\beta=\rho_{2,j}$, the $\alpha$ values of the conjugate points of the form $(\rho_{2,j}, \alpha^*_{j,l})$ of $(\rho_{2,j},1)$}
\nomenclature{${\bm e}_{j,l}$}{A right eigenvector of ${\bm A}(\rho_{2,j}, \alpha^*_{j,l})$}
\begin{remark}
Similar to the comments made in Remark \ref{r:mu1neqmu2s2}, 
a set of sufficient conditions for \eqref{e:lisinthespan} is 
1) $m_j = |{\mathcal M}|$ and 2) ${\bm c}(\rho_{2,j},\alpha^*_{j,l},{\bm e}_{j,l})$,
$l=1,2,...,|{\mathcal M}|$ are linearly independent.
By ${\bm c}(\cdot,\cdot,\cdot)$'s definition, linear independence of these
vectors require $\alpha^*_{j,l} \neq \rho_{2,j}$, which is, yet another
generalization of the assumption $\rho_1 \neq \rho_2.$
\end{remark}

\begin{remark}
One can introduce assumptions similar to \eqref{as:conj} which imply, 
with an argument similar to the
proof of Lemma \ref{l:point2conjrest},
that $(\rho_{2,j},1)$ has $|{\mathcal M|}-j$ conjugate points in the interval
$(0,1)$. But in general, this number of conjugate points will not
suffice for \eqref{e:lisinthespan} to hold and 
when constructing $(Y,M)$-harmonic functions with $\beta=\rho_{2,j}$,
$j=2,3,...,|{\mathcal M}|$, we will use conjugate points
with complex or negative $\alpha$ components.
Instead of introducing even more assumptions similar to \eqref{as:conj},
we directly incorporate \eqref{e:lisinthespan} as an assumption.
\end{remark}

To get our $(Y,M)$-harmonic function converging to $1$ 
on the tail of $\partial B$ (see \eqref{e:hstarlimit} below for the precise
statement)
we need one more condition:
\begin{equation}\label{as:bm1}
{\bm 1} \in \text{Span}({\bm d}_2, {\bm e}_2, ..., {\bm e}_{|{\mathcal M}|}).
\end{equation}
A sufficient condition for \eqref{as:bm1} is that
the vectors listed on the right of this display are linearly
independent.

\begin{proposition}\label{p:hrho2star}
Let ${\bm e}_{j}$, $j=2,3,...,|{\mathcal M}|$ be as in Lemma \ref{l:thees}.
and let ${\bm d}_2$ be as in Proposition \ref{p:point2}. Under
Assumptions \ref{as:mu1neqmu2s3} and \eqref{as:bm1} there exist
vectors ${\bm b}_{2,j} \in {\mathbb R}^{m_j}$, $j=2,3,..,|{\mathcal M}|$
and $ {\bm b}_{2} \in {\mathbb R}^{|{\mathcal M}|}$
such that
\begin{align}\label{d:hrho2j}
{\mathfrak h}_{\rho_{_{2,j}}}(y,m) &\doteq [(\rho_{2,j},1, {\bm e}_{j}), (y,m)]\\
&~~~+\sum_{l=1}^{m_j} {\bm b}_{2,j}(l)  \notag
[(\rho_{2,j}, \alpha^*_{j,l}, {\bm e}_{j,l}), (y,m)],~ j=2,3,...,|{\mathcal M}|,
\end{align}
and
\begin{equation}\label{d:hrho2star}
{\mathfrak h}  \doteq 
{\bm b}_{2}(1) {\mathfrak h}_{\rho_2} + 
\sum_{j=2}^{|{\mathcal M}|} {\bm b}_{2}(j) {\mathfrak h}_{\rho_{_{2,j}}}
\end{equation}
are all $(Y,M)$-harmonic and $\partial B$-determined;  furthermore
\begin{equation}\label{e:hstarlimit}
\lim_{k\rightarrow \infty} {\mathfrak h}(k,k,m) \rightarrow 1
\end{equation}
for all $m \in {\mathcal M}.$
\end{proposition}
\nomenclature{${\mathfrak h}_{\rho_{2,j}}$}{$(Y, M)$-harmonic function constructed from points on ${\mathcal H}$ with $\beta=\rho_{2,j}.$}
\begin{proof}
The existence of the vector ${\bm b}_{2,j}$, $j=2,3,...,|{\mathcal M}|$,
so that ${\mathfrak h}_{\rho_{2,j}}$
defined in \eqref{d:hrho2j} is $(Y,M)$-harmonic follows
from \eqref{as:mu1neqmu2s3} and the argument given in
the construction of ${\mathfrak h}_{\rho_2}$ (see the proof of
Proposition \ref{p:hrho2c}).
By  \eqref{as:bm1} there is a vector ${\bm b}_{2}$ such that
\[
{\bm b}_{2}(1) {\bm d}_2(m) + \sum_{j=2}^{|\mathcal M|} {\bm b}_{2}(j) {\bm e}_j(m) = 1
\]
for all $m \in {\mathcal M}.$ Then
${\mathfrak h}$ as defined in \eqref{d:hrho2star} satisfies
\[
{\mathfrak h}(k,k,m)=1+
{\bm b}_2(1) \sum_{j=1}^{|{\mathcal M}|} {\bm b}_{2,1}(j) (\alpha^*_j)^k {\bm d}_{2,j}(m) + 
\sum_{j=2}^{|\mathcal M|}{\bm b}_{2}(j)\sum_{l=2}^{m_j}{\bm b}_{2,j}(i) (\alpha^*_{j,l})^k{\bm e}_{j,l}(m);
\]
$|\alpha^*_j| < 1$, 
$|\alpha^*_{j,l}|<1$ imply that the last two sums go to $0$ with $k$. 
This gives \eqref{e:hstarlimit}.
\end{proof}
\nomenclature{${\mathfrak h}$}{$(Y,M)$-harmonic function constructed from a
linear combination of ${\mathfrak h}_{\rho_2}$ and ${\mathfrak h}_{\rho_{2,j}}$'s}
\nomenclature{$ {\bm b}_{2,1}$}{the weight vector used to construct ${\mathfrak h}$}
In Lemma \ref{l:thees} we found points
on $\{\alpha=1\}\cap {\mathcal H}^{\beta\alpha}$
in addition to $(\rho_2,1)$ identified in,
we used these points above
in the construction of ${\mathfrak h}.$ Similarly, one can go
along the line $\beta=\alpha$ to find points
on ${\mathcal H}^{\beta\alpha}$
other than
$(\rho_1,\rho_1)$ 
defining further simple $\partial B$-determined
$(Y,M)$-harmonic functions:
\begin{lemma}\label{l:additionaldiagonal}
Under 
the stability assumption (\ref{e:stabilityas}),
and
Assumption \ref{as:simpleeig} ($
{\bm A}(\beta,\alpha)$ has real distinct eigenvalues for
$(\beta,\alpha) \in {\mathbb R}_+^{2,o})$ 
there exist
$\rho_{1,k}$, $k=2,3,..,|{\mathcal M}|$, such that
$\rho_1 > \rho_{1,2} > \rho_{1,3} > \cdots > \rho_{1,|{\mathcal M}|} > 0$
and $\{{\bm f}_{2}\neq 0$, ${\bm f}_{3}\neq 0$,...,${\bm f}_{|{\mathcal M}|}\neq0\}
\subset {\mathbb R}^{|{\mathcal M}|}$ such that
\[
{\bm A}(\rho_{1,j},\rho_{1,j}) {\bm f}_j = {\bm f}_j, ~j =2,3,...,|{\mathcal M}|,
\]
holds.
\end{lemma}
\nomenclature{$\rho_{1,j}$}{ $(\rho_{1,j},\rho_{1,j})$ are points on ${\mathcal H}^{\beta\alpha}\cap {\mathcal H}^{\beta\alpha}_2$}
\nomenclature{${\bm f}_j$}{A right eigenvector of ${\bm A}(\rho_{1,j},\rho_{1,j})$}
The proof is parallel to that of Lemma \ref{l:point2conjrest} and is
based on Gershgorin's Theorem and the fact that
$\Lambda_j(\rho_1, \rho_1) < \Lambda_1(\rho_1,\rho_1)=1$ for $j=2,3,...,|{\mathcal M}|.$

One can use the points identified in the previous lemma
to construct further $\partial B$-determined $(Y,M)$-harmonic functions.

\begin{lemma}\label{l:furthersimple}
Let $\rho_{1,j}$, ${\bm f}_j$, $j=2,3,...,|{\mathcal M}|$, be as
in Lemma \ref{l:additionaldiagonal}. Then
\[
[(\rho_{1,j},\rho_{1,j},{\bm f}_j), \cdot],~ j=2,3,...,|{\mathcal M}|,
\]
are $\partial B$-determined $(Y,M)$-harmonic.
\end{lemma}
\begin{proof}
By definition,
$(\rho_{1,j},\rho_{1,j}) \in {\mathcal H}^{\beta\alpha}$
and ${\bm A}(\rho_{1,j},\rho_{1,j}) {\bm f}_j = {\bm f}_j.$
Again,
${\bm A}(\beta,\beta) = {\bm A}_2(\beta,\beta)$
for all $\beta$ follows from 
and ${\bm p}_2(\beta,\beta,m) = {\bm p}(\beta,\beta)$ and
the definitions of ${\bm A}$ and ${\bm A}_2$.
Then 
${\bm A}(\rho_{1,j},\rho_{1,j}) {\bm f}_j =
{\bm A}_2(\rho_{1,j},\rho_{1,j}) {\bm f}_j =
{\bm f}_j$, i.e., $(\rho_{1,j},\rho_{1,j}, {\bm f}) \in {\mathcal H}_2$
(i.e., the characteristic surface of $\partial_2$, see \eqref{d:H2}).
This and Proposition \ref{p:pointonboth} imply that
$[(\rho_{1,j},\rho_{1,j},{\bm f}_j), \cdot]$ is $(Y,M)$-harmonic.
That it is $\partial B$-determined follows from
$|\rho_{1,j}| < 1$ and Lemma \ref{l:partialBdetermined}.
\end{proof}

The function $[(\beta,\alpha, d),\cdot]$ is complex valued
for any $(\beta,\alpha, d) \in {\mathcal H}$ with complex components and
such points and the functions they define
can also be used to improve the approximation; 
see the next section for an example.
The next proposition gives an upper bound
on the relative error of an approximation of 
${\mathbb P}_{(y,m)}( \tau < \infty)$
in terms of the values the approximation takes on the boundary $\partial B$; it covers
cases when complex valued $(\beta,\alpha, d) \in {\mathcal H}$  is used in the construction
of the approximation.
For any $z \in {\mathbb C}$, let $\Re(z)$ denote its real part.
\nomenclature{$\Re(\cdot)$}{Real part of a complex value}
\begin{proposition}\label{p:boundonrelerror}
Let $h:{\mathbb Z}\times {\mathbb Z}_+\mapsto {\mathbb C}$ 
be $\partial B$-determined and $(Y,M)$-harmonic. Then
\begin{equation}\label{e:boundonrelerRe}
\max_{(y,m) \in B \times {\mathcal M}} 
\frac{| \Re(h)(y,m) -{\mathbb P}_{(y,m)}(\tau < \infty)|}{{\mathbb P}_{(y,m)}(\tau < \infty)} \le c^*
\end{equation}
where
\begin{equation}\label{d:cstar}
c^* \doteq \max_{y \in \partial B, m \in {\mathcal M}} | h(y,m) - 1|.
\end{equation}
\end{proposition}

The proof is similar to that of Proposition \ref{p:firstrelerror}:
\begin{proof}
That $h$ is $\partial B$-determined $(Y,M)$-harmonic implies the
same for its real and imaginary parts.
For any complex number $z$ we have $|\Re(z) -1| \le |z-1|$;
these and \eqref{d:cstar} give
\[
\max_{y' \in \partial B, m \in {\mathcal M}} | \Re(h)(y',m) - 1| \le c^*.
\]
Then
\[
(1-c^*)1_{\{\tau < \infty\}} \le \Re(h)(Y_\tau,M_\tau)1_{\{\tau < \infty\}} 
\le (1+c^*) 1_{\{\tau < \infty\}}.
\]
Applying ${\mathbb E}_{(y,m)}[\cdot]$ to all terms above implies 
\eqref{e:boundonrelerRe}.
\end{proof}
\section{Numerical example} \label{s:numeric}
This section demonstrates the performance of our approximation
results on a numerical example. For parameter values ${\bm P}$, 
$\lambda(\cdot)$, $\mu_1(\cdot)$ and $\mu_2(\cdot)$
we take those listed in \eqref{e:Pexample} and \eqref{e:Ratesexample},
for which $|{\mathcal M}|= 3.$ We know by Proposition \ref{p:conditionforsimple}
that for ${\bm P}$ as in \eqref{e:Pexample}, ${\bm A}(\beta,\alpha)$ has distinct
positive eigenvalues for $(\beta,\alpha) \in {\mathbb R}_+^{2,o}.$
Furthermore, the rates \eqref{e:Ratesexample} satisfy
$\lambda(m) <\mu_1(m),\mu_2(m)$ for all $m \in {\mathcal M}$, therefore,
the stability assumption \eqref{e:stabilityas} is also satisfied.
Computing the right side of \eqref{as:conj} at
$(\rho_2,1)$ shows that the parameter values \eqref{e:Pexample} and
\eqref{e:Ratesexample} satisfy \eqref{as:conj}. Therefore:
\begin{enumerate}
\item
By Proposition \ref{p:hrho2c}, the function ${\mathfrak h}_{_{2,1}}$ is 
well defined and $\partial B$-determined and $(Y,M)$-harmonic.
Furthermore, we know by Lemma \ref{l:thees} that
there are $\rho_{2,j}$, $j=2,3,...,|{\mathcal M}|$, such that 
$0 < \rho_{2,j} < \rho_2$ and $(\rho_{2,j},1) \in {\mathcal H}^{\beta\alpha}$ for all $j$.
We solve
\[
{\bm p}(\rho_{2,j},\alpha) = 0
\]
for $\alpha$ for the parameter values assumed in the section
and verify that Assumption \ref{as:mu1neqmu2s3} holds 
with $m_j = |{\mathcal M}|$
for all $j$;
this and Proposition \ref{p:hrho2star} imply that 
the $(Y,M)$-harmonic
$\partial B$-determined function ${\mathfrak h}$ defined in 
\eqref{d:hrho2star} 
and satisfying \eqref{e:hstarlimit}
is well defined.
\item Propositions \ref{p:point1}, \ref{p:hrho1} and Lemma \ref{l:rho1pB}
apply and give the $\partial B$-determined $(Y,M)$-harmonic function
$h_{\rho_1} = [(\rho_1,\rho_1, {\bm d}_1),\cdot]$,
\item Lemmas \ref{l:additionaldiagonal} and \ref{l:furthersimple}
apply and give the  $\partial B$-determined $(Y,M)$-harmonic functions
$h_{\rho_{1,j}} = [(\rho_{1,j},\rho_{1,j}, {\bm f}_j),\cdot]$, $j=2,3,....|{\mathcal M}|$.
\end{enumerate}
In addition to these functions, we can fix an integer $K > 0$, 
and construct $K \cdot |{\mathcal M}|$ further $(Y,M)$-harmonic
functions of the form
\begin{equation}\label{e:defhkj}
{\mathfrak h}_{k,j} \doteq 
\sum_{l=0}^{|{\mathcal M}|} {\bm b}_{k,j}(l) 
[(\beta_{k,j}, \alpha_{k,j,l} , {\bm d}_{k,j,l}), \cdot], 
\end{equation}
for $k=1,2,...,K$, $\beta_{k,j}$ and $j=1,2,....,|{\mathcal  M}|$,
as follows:
 \begin{enumerate}
\item Set $\alpha_{k,j,0} = R~e^{i k \frac{2\pi}{K+1}}$ , 
$R \in (0,1)$ to be determined below;
note that $\alpha_{k,j,0}$ depends only on $k$; including $j$ as
an index simplifies notation in \eqref{e:defhkj} and below.
\item For each $k$, $\beta_{k,j}$, $j=1,2,....,|{\mathcal  M}|$,
are the  $\beta$-roots of
\begin{equation}\label{e:tosolvecirclebeta}
{\bm p}(\alpha_{k,j,0}, \beta) = 0 
\end{equation}
satisfying $|\beta| < 1$;
\item $\alpha_{k,j,l}$, $l=1,2,...,|{\mathcal M}|$,
are the $\alpha$-roots of
\begin{equation}\label{e:tosolvecirclealpha}
{\bm p}(\beta_{k,j},\alpha)=0;
\end{equation}
with $|\alpha|<1$ which are distinct from $\alpha_{k,j,0}$.
\item ${\bm d}_{k,j,l}$ 
is an eigenvector of ${\bm A}(\beta_{k,j},\alpha_{k,j,l})$
i,e., 
$(\beta_{k,j}, \alpha_{k,j,l}, {\bm d}_{k,j,l}) \in {\mathcal H}$,
\item for each $(k,j)$ the vector 
${\bm b}_{k,j}$ solves
\begin{equation}\label{e:condforharm}
\sum_{l=0}^{|{\mathcal M}|} {\bm b}_{k,j}(l) 
{\bm c}(\beta_{k,j},\alpha_{k,j,l}) = 0,
\end{equation}
\end{enumerate}
where ${\bm b}_{k,j}(l)$ is the $l^{th}$ component of
the vector ${\bm b}_{k,j}.$
For ${\mathfrak h}_{k,j}$, $k=1,2,...,K$ and $j=1,2,...|{\mathcal M}|$ 
to be well defined, $(Y,M)$-harmonic and $\partial B$-determined we need
1) for each $k$, the equation \eqref{e:tosolvecirclebeta}
needs to have at least $|{\mathcal M}|$ $\beta$-roots with absolute value
less than $1$;  2) for each $k$ and $j$, the equation
\eqref{e:tosolvecirclealpha} needs to have at least $|{\mathcal M}|$ solutions
different from $\alpha_{k,j,0}$
with absolute value less than $1$; 3) for each $k$ and $j$ the equation
\eqref{e:condforharm} needs to have a nontrivial solution ${\bm b}_{k,j}.$
Here we have two parameters to set: $K$ and $R$; for the purposes
of this numerical example we set $R=0.7$, and $K=5$. Upon 
solving \eqref{e:tosolvecirclebeta}, 
\eqref{e:tosolvecirclealpha} and \eqref{e:condforharm} with these
parameter values
we observe that they have sufficient number of solutions for
${\mathfrak h}_{k,j}$ 
to be well defined and $(Y,M)$-harmonic and $\partial B$-determined.

We have now $1 + 6|{\mathcal M}|$, $\partial B$-determined
$(Y,M)$-harmonic functions to construct our approximation of 
${\mathbb P}_{(y,m)}(\tau < \infty)$; the approximation will be of the form
\begin{equation}\label{d:ha}
h^{a,K} \doteq \Re(h^{a*,K}),~{h}^{a*,K} \doteq {\mathfrak h}  + \phi_1 h_{\rho_1} + \sum_{j=2}^{|{\mathcal M}|}
{\phi}_j [(\rho_{1,j},\rho_{1,j}, {\bm d}_{1,j}),\cdot]
+ \sum_{k=1,j=1}^{K,|{\mathcal M}|} {\phi}_{j,k} {\mathfrak h}_{j,k},
\end{equation}
\nomenclature{${h}^{a,K}$}{Approximation using $K$ additional $\alpha$ points
other than $\alpha=1$ and $\alpha=\rho_1.$}
\nomenclature{${\bm d}_{1,j}$}{A right eigenvector of ${\bm A}(\rho_{1,j},\rho_{1,j})$}
where ${\phi}_j$ and ${\phi}_{j,k}$ are ${\mathbb C}$ valued coefficients
to be chosen so that $ h^{a,K}|_{\partial B}$ is as close to $1$ as possible.
As in \cite[Section 8.2]{sezer2015exit},
one simple way to do this is to choose
these $(K+1)|{\mathcal M}|$ coefficients so that 
${ h}^{a,K}(y,y,m) = 1$ for $y =0,1,2,..,K$ and $m \in {\mathcal M}.$
This defines a 
$(K+1)|{\mathcal M}| \times (K+1)|{\mathcal M}|$ system; for our parameter
values ($K=5$ and $|{\mathcal M}| = 3$) this is an $18 \times 18$ system,
and it does turn out to have a unique solution. 
Once the ${\phi}_j$ and ${\phi}_{j,k}$ are determined through this
solution, an upper bound on the approximation relative error
can be computed via Proposition \ref{p:boundonrelerror}; it suffices
to compute $c^*$ of \eqref{d:cstar}; for ${ h}^{a*,K}$ of \eqref{d:ha}
it turns out to be
\[
c^* = 0.00367;
\]
therefore, by Proposition \ref{p:boundonrelerror}, 
${h}^{a,K}$ approximates
${\mathbb P}_{(y,m)}( \tau < \infty)$ with relative error bounded by this quantity.
By Theorem \ref{t:mainapprox} we know that 
${\mathbb P}_{(T_n(x_n),m)}( \tau < \infty)$ approximates 
${\mathbb P}_{(x_n,m)}(\tau_n < \tau_0)$ with vanishing relative error for $x_n = \lfloor nx \rfloor$,
$x(1) > 0$; it follows from these that
${h}^{a,K}(n-x_n(1),x_n(2))$ will approximate
${\mathbb P}_{(x_n,m)}(\tau_n < \tau_0)$ with relative error bounded
by   $c^*$ 
for $n$ large. Let us see how well this approximation works
in practice. Figure \ref{f:levelcurves} gives the level curves of
$-\log({h}^{a,K}(n-x(1),x(2),1))$ and $-\log {\mathbb P}_{(x,m)}(\tau_n <\tau_0)$;
${\mathbb P}_{(x,m)}(\tau_n < \tau_0)$ is computed by iterating the harmonic
equation satisfied by this probability; for $n=60$, this iteration
converges in less than $1000$ steps.

\begin{figure}[h]
\begin{center}
\scalebox{0.75}{
\centerline{\input{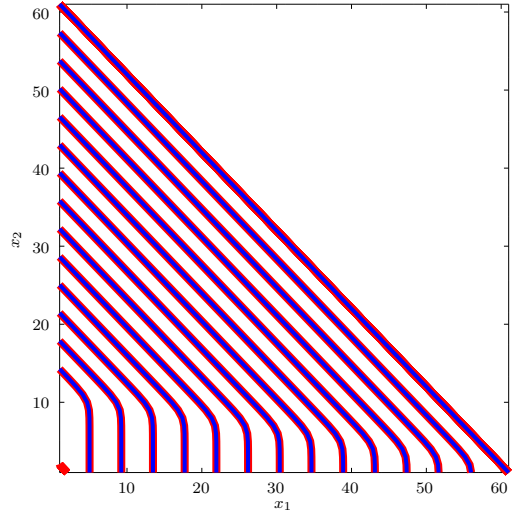}}}
\end{center}

\vspace{-0.65cm}
\caption{\hspace{0.25cm}Level curves of $-\log({ h}^{a,K}(n-x(1),x(2),1))$ and $-\log {\mathbb P}_{(x,m)}(\tau_n <\tau_0)$\label{f:levelcurves}}
\end{figure}
As can be seen, and agreeing with the analysis above, these lines
completely overlap except for a narrow region around the origin.

\begin{figure}
	\centering
	\scalebox{.95}{\input{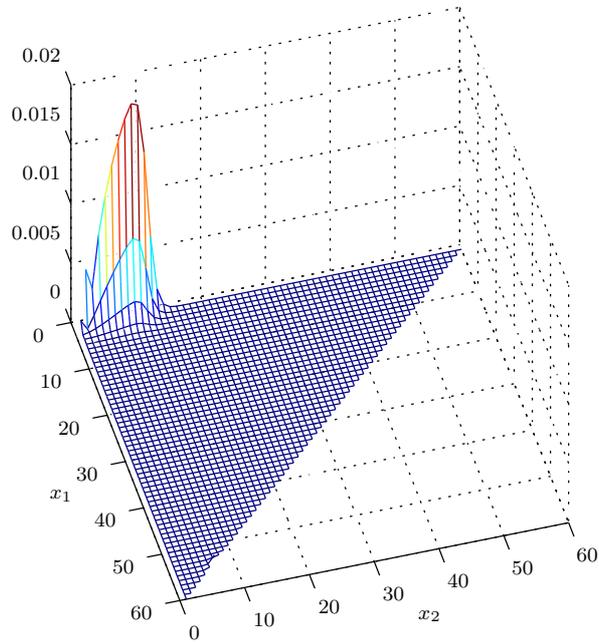}}
\caption{The relative error $\frac{ |\log({h}^{a,K}(n-x(1),x(2),1)-\log {\mathbb P}_{x}(\tau_n <\tau_0)|}{
			|\log {\mathbb P}_{x}(\tau_n <\tau_0)|}$ }
	\label{f:relerr}
\end{figure}

\FloatBarrier
Figure \ref{f:relerr} shows the relative error
\[
\frac{ |\log({h}^{a,K}(n-x(1),x(2),1))-\log {\mathbb P}_{(x,m)}(\tau_n <\tau_0)|}{
	|\log {\mathbb P}_{(x,m)}(\tau_n <\tau_0)|},
\]
we see that it is virtually $0$ except for the same region around $0$ where
it is bounded by $0.02.$ This narrow layer of where the relative error spikes corresponds to the region $1-x(2) < \log(\rho_2)/\log(\rho_1)$ identified in Theorem \ref{t:mainapprox}.

\section{Comparison with earlier works}\label{s:comparison}

The present work shows how one can approximate
the probability $\mathbb{P}_{(x,m)}(\tau_n<\tau_0)$ 
by\\ $\mathbb{P}_{(y,m)}(\tau<\infty)$ with exponentially
vanishing relative error and constructs analytical approximation formulas 
for the latter. This is done by extending the approach of 
\cite{sezer2015exit,sezer2018approximation} to Markov modulated dynamics. 
In this section, we compare the analysis of the modulated case treated in this work 
with the non-modulated two tandem case treated in \cite{sezer2015exit,sezer2018approximation}
and the non-modulated two dimensional simple random walk treated in \cite{unlu2018excessive}. 

\paragraph{Harmonic functions}
The nonmodulated analysis uses functions 
of the form\\ $y\mapsto[(\beta,\alpha), y] =\beta^{y(1)-y(2)}\alpha^{y(2)}$ 
where $(\beta, \alpha)$ are chosen from the roots of a characteristic polynomial of second order associated with the process $Y$. 
Markov modulation brings an additional state variable $m$, leading to 
functions of the form $(y,m)\mapsto[(\beta,\alpha, {\bm d}), (y,m)] =\beta^{(y(1)-y(2)}\alpha^{y(2)}{\bm d}(m)$. 
The characteristic surface is now defined in terms of eigenvalue and 
eigenvector equations of a characteristic matrix depending on
$(\beta,\alpha) \in {\mathbb C}^2$.
\paragraph{Geometry of the characteristic surface}
The characteristic surface in 
\cite{sezer2015exit,sezer2018approximation, unlu2018excessive} is
the $1$-level curve of a rational function 
which can be represented as a second degree polynomial
in each of the $\beta,\alpha$ variables; the projection of the
characteristic surface to ${\mathbb R}_+^2$ is a smooth closed curve
bounding a convex region.
Conjugate points on this curve come in pairs and have elementary formulas.
The characteristic curve in the modulated 
case is the $0$-level curve of the characteristic polynomial of a
characteristic matrix
and can be represented as a $2|{\mathcal M}|$ degree polynomial in
each of the variables; its projection to 
${\mathbb R}_+^2$ consists of  $|{\mathcal M}|$ components, one for each eigenvalue $\Lambda_j$ of the characteristic matrix. The error analysis is based on
the level curve of the largest eigenvalue while the computation of
${\mathbb P}_{(y,m)}(\tau < \infty)$ uses points on all components.
There are in general no simple formulas for the roots of a polynomial greater than degree $4$ and the formulas
for degree $4$ are fairly complex; therefore,
for $|{\mathcal M}| \ge 2$ (i.e., even for the simplest nontrivial Markov modulated constrained random walk with two modulating states) the points on these
curves no longer have simple formulas and identification of the relevant
points (Propositions \ref{p:point1}, \ref{p:point2} and \ref{p:point2conj},
Lemmas \ref{l:point2conjrest},  \ref{l:thees} and \ref{l:additionaldiagonal})  
requires matrix /
eigenvalue analysis
and the implicit function theorem.
\paragraph{Assumptions}
We use the point $(\rho_2,1)$ and its conjugate $(\rho_2, \alpha^*_1)$ lying on $\mathcal{L}_1$ to define $(Y,M)$-superharmonic functions to use in our limit analysis. The existence of $(\rho_2,1) \in {\mathcal L}_1$, follows
from the stability assumption \eqref{e:stabilityas}.
The identification of the conjugate point $(\rho_2, \alpha^*_1)$ requires 
the additional assumption \eqref{as:conj} ensuring $\alpha^*_1<1$. 
A similar assumption is not needed in the non-modulated tandem case, 
because when there is no modulation, i.e., when $|{\mathcal M}|=1$,
the conjugate of $(\rho_2,1)$ is $(\rho_2,\rho_1)$ and
$\rho_1 < 1$ by the stability assumption. 
For the constrained random walk representing two parallel queues 
treated in \cite{unlu2018excessive},
the assumption corresponding to \eqref{as:conj} is 
$r^2<\rho_2\rho_1$, where $r$ is utilization rate of the whole system.

The assumption $\rho_1\neq \rho_2$ (see \eqref{as:mu1neqmu2})
generalizes the assumption $\mu_1\neq\mu_2$ 
from the non-modulated tandem case and the parallel case treated
in \cite{sezer2018approximation, unlu2018excessive}.
The computation of $\mathbb{P}_{(y,m)}(\tau<\infty)$ needs progressively more general versions of this assumption
(see \eqref{as:mu1neqmu2s2}, Remark \ref{r:mu1neqmu2s2} and Assumption \ref{as:mu1neqmu2s3}).
\paragraph{Analysis}The approximation error analysis in the non-modulated case is based on the subsolutions of a limit HJB equation and $Y$-harmonic functions.
These works use these subsolutions to construct supermartingales which are then used to find upper bounds on error probabilities.
In this work we construct the supermartingales directly using $(Y,M)$-{\em super}harmonic functions constructed from points
on the characteristic surface. Because $Y$ has one less constraint compared to $X$, these functions can be subharmonic on the
boundary where $Y$ is not constrained. To overcome this, we introduce a decreasing term to the definition of the supermartingale.

In the tandem case there is an explicit formula for $P_y(\tau < \infty)$; 
this formula is used directly in the analysis of
the error probability. There is obviously no explicit formula for the corresponding probability in the Markov modulated case. Instead,
we derive an upper bound  on it in Section \ref{s:ub1} 
using again $(Y,M)$-superharmonic functions; this upperbound is used in
the error analysis of Section \ref{s:together}.
\paragraph{Computation of the limit probability}
In the non-modulated tandem case  treated in \cite{sezer2018approximation},
${\mathbb P}_y(\tau < \infty)$ 
can be represented exactly as a linear combination
of $h_{\rho_2}$ and $h_{\rho_1}$; so the computation of
${\mathbb P}_y(\tau < \infty)$ is trivial for the nonmodulated two
dimensional tandem walk.
In the parallel case treated in \cite{unlu2018excessive}, 
$P_y(\tau <\infty)$ can be represented exactly as a linear combination of 
$h_{\rho_1}$ and $h_r$  when $r^2 =\rho_1\rho_2$;
when this doesn't hold 
\cite{unlu2018excessive} develops approximations of 
$P_y(\tau <\infty)$ from harmonic functions constructed from conjugate
points on the characteristic surface, which is an application of
the principle of superposition.
For the modulated case we use the same principle but Markov
modulation complicates the construction of the functions
used in the approximation.
The identification of the points on the characteristic surface
requires the solution of $2|{\mathcal M}|$
degree polynomial equations 
(first the $\alpha$ component is fixed
to identify possible $\beta$ components; then for each of the identified
$\beta$'s,  the polynomial is solved in $\alpha$ to find the relevant
conjugate points).
Eigenvectors corresponding to these points are then computed and
finally
we solve a linear equation to find the coefficients of the
exponential functions (see, for example, the ${\bm b}_{k,j}$ vector in 
\eqref{e:defhkj} and \eqref{e:condforharm}). The corresponding process is trivial
when there is no modulation.
In \cite{sezer2018approximation} and \cite{unlu2018excessive}
the function $h_{\rho_2}$ plays a central role in the approximation
of ${\mathbb P}_y(\tau < \infty)$ because it equals approximately $1$
away from the origin; due to Markov modulation there can be in general
no 
function constructed from a single point and its conjugates that takes a fixed value
on $\partial B.$ To deal with this, we use an appropriate linear combination
of functions constructed from multiple points and their conjugates 
on the characteristic surface
so that the linear combination takes the value $1$ away from the
origin (Proposition \ref{p:hrho2star}).

\section{Conclusion}\label{s:conclusion}

The current work develops approximate 
formulas for the exit probability of the two dimensional
tandem walk with modulated dynamics. 
Our main approximation Theorem \ref{t:mainapprox} says that $\mathbb{P}_{(T_n(x_n), m)}(\tau<\infty)$ approximates $\mathbb{P}_{(x_n, m)}(\tau_n<\tau_0)$ with relative
error vanishing exponentially fast with $n$. 
To compute the exit probability, we first construct $\partial B$-determined $(Y,M)$-harmonic functions from single and conjugate points on the corresponding characteristic surface and then with their linear combinations, 
approximate the boundary value $1$ of the harmonic function $\mathbb{P}_{(y,m)}(\tau<\infty)$.
In the non-modulated tandem case treated in \cite{sezer2015exit}, 
the probability $\mathbb{P}_y(\tau<\infty)$ can be represented in any
dimension
exactly using harmonic functions constructed from points on the  characteristic
surface.
As is seen in the present work, even dimension two entails considerable difficulties. Whether an extension to higher
dimensions is possible is a question we would like to tackle in future work.

The work \cite{sezer2015exit} gives
a formula for ${\mathbb P}_y(\tau < \infty)$ 
for the non-modulated tandem walk
when $\rho_1=\rho_2$ 
based on harmonic functions with polynomial terms.
Whether similar computations can be carried out for
$\mathbb{P}_{(y, m)}(\tau<\infty)$ in the modulated case when $\rho_1=\rho_2$ is another question for future research.

The assumption \eqref{as:conj} plays a key role in our analysis;
it ensures that various functions such as $h_{\rho_2}$ whose construction involves the point $(\rho_2,\alpha_1^*)$ remain bounded on $\partial B.$
We think that new ideas will be needed to treat the case
when \eqref{as:conj} doesn't hold; this remains for future work.
 
 The computations and the error analysis in the present work depend on the dynamics of the process and the geometry of the exit boundary. 
A significant problem for future research is to extend these to 
other dynamics in two or higher dimensions and to other 
exit boundaries. The simple random walk dynamics  (i.e., increments $(1,0)$, $(-1,0)$, $(0,1)$ and $(0,-1)$) and the rectangular
exit boundary
 appear to be the most natural to study in immediate future work.

\appendix

\section{Two lemmas}
For a square matrix ${\bm G}$, let 
${\bm G}^{i,j}$ denote the matrix obtained by removing the
$i^{th}$ row and $j^{th}$ column of ${\bm G}$.

\begin{lemma}\label{l:positivity}
For $n_0 \in \{2,3,4,...\}$,
suppose ${\bm G}$ is an $n_0  \times n_0$ irreducible and aperiodic matrix
with nonnegative entries. 
Then
$\det\left((\Lambda_1({\bm G}) {\bm I} - {\bm G})^{i,i}\right) > 0$ 
for all $i \in \{1,2,...,n_0\}$, 
where ${\bm I}$ is the $n_0 \times n_0$ identity matrix.
\end{lemma}
\begin{proof}
The argument is the same for all 
$ i \in \{1,2,...,n_0\}$; so it suffices
to argue for $i=1.$ Suppose the claim is not true and 	
\begin{equation}\label{e:wrongas}
\det\left((\Lambda_1({\bm G}){\bm I} - {\bm G})^{1,1}\right) \le 0. 
\end{equation}	
Consider the function
$u \mapsto g(u) = \det\left(( u{\bm I} -{\bm G})^{1,1} \right)$, $u \ge 0.$
The multilinearity and continuity of $\det$ implies $\lim_{u\nearrow \infty}
g(u) = \infty.$ This implies that if \eqref{e:wrongas} is true
there must be $u_0 \ge \Lambda_1({\bm G})$ such that
\begin{equation}\label{e:step1}
\det\left( (u_0 {\bm I} - {\bm G})^{1,1} \right) = 0.
\end{equation}
The matrix
${\bm G}^{1,1}$ 
is nonnegative, therefore,
it has a largest eigenvalue $\Lambda_1({\bm G}^{1,1})$ with
an eigenvector ${\bm v}_1 \ge 0$.
The equality \eqref{e:step1} implies 
\begin{equation}\label{e:contineq}
\Lambda_1({\bm G}^{1,1}) \ge u_0 \ge \Lambda_1({\bm G}).
\end{equation}
That ${\bm G}$ is irreducible and aperiodic 
implies that ${\bm G}^{n_0}$ is strictly
positive; its largest eigenvalue is
\[
\Lambda_1({\bm G}^{n_0}) =
\Lambda_1({\bm G})^{n_0}.
\]
The matrix $({\bm G}^{n_0})^{1,1}$ has strictly positive entries
and therefore its largest eigenvalue $\Lambda_1( ({\bm G}^{n_0})^{1,1})$
has an eigenvalue ${\bm v}_2$ with strictly positive entries.
For two vectors $x,y \in {\mathbb R}^d$, let $x \ge y$ 
and $x > y$ denote componentwise
comparison.
The inequality
\[
({\bm G}^{n_0})^{1,1} \ge
({\bm G}^{1,1})^{n_0} 
\]
implies
\begin{equation}\label{e:step0}
({\bm G}^{n_0})^{1,1} {\bm v}_1 \ge  \Lambda_1({\bm G}^{1,1})^{n_0} {\bm v}_1.
\end{equation}
On the other hand
\begin{equation}\label{e:charl1}
\Lambda_1(({\bm G}^{n_0})^{1,1}) = 
\sup\{ c: \exists x \in {\mathbb R}^{n_0-1}_+,
({\bm G}^{n_0})^{1,1}x  \ge cx \},
\end{equation}
(see \cite[Proof of Theorem 1, Chapter 16]{lax1996linear}). This
and \eqref{e:step0} imply
\begin{equation}\label{e:contineq2}
\Lambda_1(({\bm G}^{n_0})^{1,1})  \ge  
\Lambda_1({\bm G}^{1,1})^{n_0}.
\end{equation}

Define ${\bm v}_3 = [ 1; {\bm v}_2] \in {\mathbb R}^{n_0}$;
it
follows from 
$({\bm G}^{n_0})^{1,1} {\bm v}_2 = \Lambda_1(({\bm G}^{n_0}){1,1}) {\bm v}_2$,
the strict
positivity of the components of ${\bm G}^{n_0}$ and ${\bm v}_2$ 
that one can choose $\delta > 0$ small enough so that	
\[
{\bm G}^{n_0} {\bm v}_3 > 
\left( \Lambda_1(({\bm G}^{n_0})^{1,1}+ \delta\right) {\bm v}_3;
\]
This and
\[
\Lambda_1({\bm G}^{n_0})=
\sup\{ c: \exists x \in {\mathbb R}^{n_0}_+, 
{\bm G}^{n_0}x  \ge cx \}
\]
imply
\[
\Lambda_1({\bm G}^{n_0})  
>
\Lambda_1\left(({\bm G}^{n_0})^{1,1}\right).
\]
The last inequality, \eqref{e:contineq2} and \eqref{e:contineq} 
imply
\[
\Lambda_1({\bm G})^{n_0}=
\Lambda_1({\bm G}^{n_0})  
>
\Lambda_1(({\bm G}^{n_0})^{1,1}) \ge
\Lambda_1({\bm G}^{1,1})^{n_0} \ge \Lambda_1({\bm G})^{n_0},
\]
which is a contradiction.
\end{proof}
In our analysis we need the
following fact from \cite{sezer2009importance};
its proof is elementary and follows from 
the multilinearity of the determinant function and the previous lemma.

\begin{lemma}\label{l:deteigenvec}
Let ${\bm G}$ be an aperiodic and
irreducible transition matrix. 
Then the row vector whose $i^{th}$ component equals
$\det\left( ({\bm I} - {\bm G})^{i,i}\right)$ 
is the unique (upto scaling by a positive
number) left eigenvector associated with the eigenvalue $1$ of ${\bm G}$.
\end{lemma}

\bibliography{p7}

\end{document}